\documentclass[12pt,showkeys]{amsart}
\def\makeheadbox{{%
\hbox to0pt{\vbox{\baselineskip=10dd\hrule\hbox
to\hsize{\vrule\kern3pt\vbox{\kern3pt
\hbox{\bfseries Draft for discussion }
\hbox{Date of this version: 13.04.21}
\kern3pt}\hfil\kern3pt\vrule}\hrule}%
\hss}}}

\usepackage{color}
\usepackage{graphicx}
\usepackage{subfigure}
\usepackage{amssymb}
\usepackage{amsmath}
\usepackage{accents}
\usepackage{multirow}
\usepackage{enumerate}
\usepackage{epstopdf}
\usepackage{cite,stmaryrd,txfonts}
\usepackage{tikz}
\usetikzlibrary{snakes}
 \usepackage{float}
\usepackage{bbding}
 \usepackage{booktabs}
 \usepackage[misc]{ifsym}

\usepackage{epic}%% package for dash line \dashline
\usepackage{empheq}
\usepackage{cases}

\def\cequiv{\raisebox{-1.5mm}{$\;\stackrel{\raisebox{-3.9mm}{=}}{{\sim}}\;$}}

\newtheorem{theorem}{Theorem}[section]
\newtheorem{remark}[theorem]{Remark}\newtheorem{proposition}[theorem]{Proposition}
\newtheorem{lemma}[theorem]{Lemma}

\newtheorem{definition}[theorem]{Definition}

\newcounter{mnote}
\let\oldmarginpar\marginpar
\renewcommand\marginpar[1]{\-\oldmarginpar[\raggedleft\footnotesize #1]
  {\raggedright\footnotesize #1}}
  
\usepackage{footnote}
\usepackage{booktabs}
\newcommand{\ud}{\,d}

\newcommand{\jump}[1]{\llbracket {#1} \rrbracket}
\usepackage{rotating}

\numberwithin{equation}{section}

\usepackage{cite}

\usepackage[colorlinks,linkcolor=black,
anchorcolor=black,
citecolor=black]{hyperref}
\usepackage[shortlabels]{enumitem} 
\setlist[enumerate]{nosep}
\usepackage{color, soul}
\soulregister\cite7
\usepackage{setspace}

\def\uu{\undertilde{u}}
\def\uw{\undertilde{w}}
\def\uv{\undertilde{v}}
\def\uV{\undertilde{V}}
\def\uW{\undertilde{W}}
\def\uZ{\undertilde{Z}}
\def\uL{\undertilde{L}}
\def\uH{\undertilde{H}}

\def\uC{\undertilde{C}}

\def\uphi{\undertilde{\varphi}}
\def\upsi{\undertilde{\psi}}

\def\uf{\undertilde{f}}
\def\ug{\undertilde{g}}

\def\uz{\undertilde{z}}
\def\rot{{\rm rot}}
\def\curl{{\rm curl}}
\def\dv{{\rm div}}
\begin{document}
\title[A low-degree strictly conservative finite element method]{A low-degree strictly conservative finite element method for incompressible flows}

%\author{Huilan Zeng}
%
%\address{zhl@lsec.cc.ac.cn}
%
%\author{Chen-Song Zhang}
%
%\address{zhangcs@lsec.cc.ac.cn}
%
%\author{Shuo Zhang}

\author{Huilan Zeng, Chen-Song Zhang, and Shuo Zhang}

\address{
LSEC, Institute of Computational Mathematics and Scientific/Engineering Computation, Academy of Mathematics and System Sciences, Chinese Academy of Sciences, Beijing 100190, China;
 University of Chinese Academy of Sciences, Beijing 100049,  China} 
 
\email{
\big\{zhl, zhangcs, szhang\big\}@lsec.cc.ac.cn
}

\thanks{This work is supported by the Strategic Priority Research Program of Chinese Academy of Sciences (XDB 41000000), National Key R\&D Program of China (2020YFA0711904), and the Natural Science Foundation of China (11871465 and 11971472).}



\maketitle

\begin{abstract}
In this paper, a new $P_{2}-P_{1}$ finite element pair is proposed for incompressible fluid. For this pair, the discrete inf-sup condition and the discrete Korn's inequality hold on general triangulations. It yields exactly divergence-free velocity approximations when applied to models of incompressible flows. The robust capacity of the pair for incompressible flows are verified theoretically and numerically.  
\end{abstract}

\section{Introduction}
\label{intro}

The property of conservation plays a key role in the modeling of many physical systems. For the Stokes problem, for example, if a stable finite element pair can inherit the mass conservation, the approximation of the velocity can be independent of the pressure and the method does not suffer from the locking effect with respect to large Reynolds' numbers (c.f., e.g.,~\cite{Brezzi;Fortin1991}). The importance of conservative schemes is also significant in, e.g., the nonlinear mechanics~\cite{Auricchio;Veiga;Lovadina;Reali2010,Auricchio2013} and the magnetohydrodynamics~\cite{Hu;Ma;Xu2017,Hiptmair;Li;Mao;Zheng2018,Hu;Xu2019}. In this paper, we focus on the conservative scheme for the Stokes-type problems for incompressible flows. As Stokes-type problems are applied widely to not only fluid problems but also elastic models such as the earth model with a fluid core~\cite{Dahlen1974} and they are immediately related to many other model problems, their conservative schemes can be relevant and helpful to more other equations.
~\\

Most classical stable Stokes pairs relax the divergence-free constraint by enforcing the condition in the weak sense, and the conservation can  be preserved strictly only for special examples. Though, during the past decade, the conservative schemes have been recognized more clearly as \emph{pressure robustness} and widely studied and surveyed in, e.g.,~\cite{Nicolas;Alexander;Philipp2019,Guzman;NeilanGeneral2014,Neilan2017,Schroeder;Lube2018}.  This conservation is also related to other key features such as  ``viscosity-independent"~\cite{Uchiumi2019} and ``gradient-robustness"~\cite{Linke;Merdon2020} for numerical schemes. There have been various successful examples along different technical approaches. Efforts have been devoted to the construction of conforming conservative pairs, and extra structural assumptions are generally needed for the subdivision and finite element functions. Examples include conforming elements designed for special meshes, such as $P_{k}-P_{k-1}$ triangular elements for $k\geqslant 4$ on singular-vertex-free meshes~\cite{Scott;Vogelius2009} and for smaller $k$ constructed on composite grids~\cite{Arnold;Qin1992,Scott;Vogelius2009,Qin;Zhang2007,Zhang2008,XuZhang2010}, and the pairs given in~\cite{FALK;NEILAN;2013,Guzman;NeilanGeneral2014} which work for general triangulations but with extra smoothness requirement and more complicated shape function spaces. A natural way to relax the constraints is to use $H({\rm div})$-conforming but $(H^1)^2$-nonconforming finite element functions for the velocity. For example, in \cite{Mardal;Tai;Winther2002}, a reduced cubic polynomial space which is $H(\dv)$-conforming and $(H^1)^2$-nonconforming is used for the velocity and piecewise constant for the pressure. The pair is both stable and conservative on general triangulations. The velocity space of \cite{Mardal;Tai;Winther2002} can be recognized as a modification of an $H(\dv)$-conforming space by adding some normal-bubble-like functions to enforce weak continuity of tangential component. Several conservative pairs are constructed subsequently in, e.g.,~\cite{Guz;Neilan2012,Tai;Winther2006,Xie;Xu;Xue2008}. Generally, to construct a conservative pair that works on general triangulations without special structures, cubic and higher-degree polynomials are used for the velocity. Besides, For conservative pairs in three-dimension, we refer to, e.g.,~\cite{Guzman;Neilan2018,Zhang3D2005,ZhangPS2011} where composite grids are required, as well as \cite{Guzman;Neilan2013,Zhang3D2011} where high degree local polynomials are utilized. We refer to~\cite{Chen;Dong;Qiao2013,Huang;Zhang2011,ZhangSY2009} for rectangular grids and~\cite{Neilan;Sap2016} for cubic grids where full advantage of the geometric symmetry of the cells are taken. 
~\\

In this paper, we propose a new $P_2-P_1$ finite element pair on triangulations; for the velocity field, we use piecewise quadratic $H(\dv)$ functions whose tangential component is continuous in the average sense, and for the pressure, we use discontinuous piecewise linear functions. The pair is stable and immediately strictly conservative on general triangulations. Further, a discrete Korn's inequality holds for the velocity.  The capability of the pair is verified both theoretically and numerically. When applied to the Stokes and the Darcy--Stokes--Brinkman problems, the approximation of the velocity is independent of the small parameters and thus locking-free; numerical experiments verify the validity of the theory. We note that, as the tangential component of the velocity function is continuous only in the average sense, the convergence rate can only be proved to be of $\mathcal{O}(h)$ order. However, since the pair is conservatively stable on general triangulations, it plays superior to some $\mathcal{O}(h^2)$ schemes numerically in robustness with respect to triangulations and with respect to small parameters. The performance of the pair on the Navier--Stokes equation is also illustrated numerically. More applications in other model problems for both the source problems and the eigenvalue problems may be studied in future. 
~\\

For the newly designed space for velocity, all the degrees of freedom are located on edges of the triangulation. It is thus impossible to construct a commutative nodal interpolator with a non-constant pressure space. To prove the inf-sup condition, we adopt Stenberg's macroelement technique~\cite{Stenberg1990technique}. On every macroelement, the surjection property of the divergence operator is confirmed by figuring out its kernel space. This figures out a structure of discretized Stokes complex on any local macroelements. On the other hand, similar to the study of conservative pairs in \cite{FALK;NEILAN;2013,Guzman;NeilanGeneral2014} and the study of biharmonic finite elements in \cite{Shuo.Zhang2020,Falk;Morley;1990,ShuoZhang2016,ZZZ2021}, the proposed global space will be embedded in a discretized Stokes complex on the whole triangulation; this will be studied in detail in future.   
~\\

The method given uses an $H({\rm div})$-conforming and $(H^1)^2$-nonconforming finite element for the velocity. Indeed, the space given here is a reduced subspace of the second order Brezzi-Douglas-Marini element \cite{BrezziDouglasMarini1985} space by enhancing smoothness. This way, the proposed pair is different from most existing $H(\dv)$-conforming and $(H{}^1)^{2}$-nonconforming methods which propose to add bubble-like basis functions on some specific $H(\dv)$ finite element space. Moreover, as we use  quadratic polynomials only for velocity, to the best of our knowledge, this is the lowest-degree conservative stable pair for the Stokes problem on general triangulations. More stable and conservative pairs may be designed by reducing other $H(\dv)$-conforming elements. The possible generalization of the proposed pair to three-dimensional case will also be discussed.
~\\

Finally we remark that, besides these finite element methods mentioned above, an alternative is to construct specially discrete variational forms onto $H({\rm div})$ functions where extra stabilizations may play roles; works such as the discontinuous Galerkin method, the weak Galerkin method, and the virtual element method all fall into this category. There have been many valuable works of these types, but we do not seek to give a complete survey and thus will not discuss them in the present paper. We only note that natural connections between the proposed pair and DG-type methods may be expected under the framework of~\cite{Hong;Wang;Wu;Xu}; along the lines of~\cite{Dios;Brezzi;Marini;Xu;Zikatanov2014}, these connections may be expected helpful for the construction of optimal solvers for the DG schemes.  
~\\

The rest of the paper is organized as follows. At the remaining of this section, some notations are given. In Section~\ref{sec:new element}, a new $P_{2}-P_{1}$ element method is proposed, and significant properties of it are presented. In Sections~\ref{sec:Model problems}, the convergence analysis of the element applied to the Stokes problem and the Darcy-Stokes-Brinkman problem is provided. In Section~\ref{sec:numerical experiments}, numerical experiments are presented to reflect the efficiency of the strictly conservative method when compared with some classical elements. A meticulous proof of a significant lemma devoted to the verification of the inf-sup condition is put in Appendix~\ref{sec:app}.

\subsection{Notations}
Throughout this paper, $\Omega$ is a bounded and connected polygonal domain in~$\mathbb{R}^{2}$.  We use $\nabla$, $\Delta$, $\dv$, $\rot$, $\curl$ to denote the gradient, Laplace, divergence, rotation, and curl operators, respectively.  As usual, we use $L^{p}(\Omega)$, $H^{s}(\Omega)$, $H(\dv,\Omega)$, $H(\rot,\Omega)$, $H_{0}^{s}(\Omega)$, and $H_{0}(\dv,\Omega)$ for standard Sobolev spaces. Denote $\displaystyle L^{2}_0(\Omega):=\Big\{w\in L^{2}(\Omega):\int_\Omega w \ud \Omega =0\Big\}$. We use `` $\undertilde{~}$ " for vector valued quantities. Specifically, we denote $\uL{}^{p}(\Omega):=\big(L^{p}(\Omega)\big)^{2}$, $\undertilde{H}{}^{s}(\Omega):=\big(H^{s}(\Omega)\big)^{2}$, $\undertilde{H}{}(\dv,\Omega):= \big(H(\dv,\Omega)\big)^{2}$, and $\uH{}(\rot,\Omega):= \big(H(\rot,\Omega)\big)^{2}$. Denote, by $H^{-s}(\Omega)$ and $\uH{}^{-s}(\Omega)$, the dual spaces of $H_{0}^{s}(\Omega)$ and $\uH{}_{0}^{s}(\Omega)$, respectively. We utilize the subscript $``\cdot_h"$ to indicate the dependence on grids. Particularly, an operator with the subscript $``\cdot_h"$ implies the operation is done cell by cell. We denote $(\cdot, \cdot)$ and $\langle \cdot, \cdot \rangle$ as the usual inner product and the dual product, respectively. Finally, $\lesssim$, $\gtrsim$, and $\cequiv$ respectively denote $\leqslant$, $\geqslant$, and $=$ up to some multiplicative generic constant~\cite{J.Xu1992}, which only depends on the domain and the shape-regularity of subdivisions. 

Let $\big\{\mathcal{T}_h\big\}$ be in a family of triangular grids of domain $\Omega$. The boundary $\partial \Omega = \Gamma_{D} \cup \Gamma_{N}$. Let $\mathcal{N}_h$ be the set of all vertices, $\mathcal{N}_h=\mathcal{N}_h^i\cup\mathcal{N}_h^b$, with $\mathcal{N}_h^i$ and $\mathcal{N}_h^b$ comprising the interior vertices and the boundary vertices, respectively. Similarly, let $\mathcal{E}_{h}=\mathcal{E}_{h}^{i}\bigcup\mathcal{E}_{h}^{b}$ be the set of all the edges, with $\mathcal{E}_{h}^{i}$ and $\mathcal{E}_{h}^{b}$ comprising the interior edges and the boundary edges, respectively. For an edge $e$, $\mathbf{n}_e$ is a unit vector normal to $e$ and $\boldsymbol{t}_e$ is a unit tangential vector of $e$ such that $\mathbf{n}_e\times \boldsymbol{t}_e>0$. On the edge $e$, we use $\llbracket\cdot\rrbracket_e$ for the jump across $e$. We stipulate that, if $e = T_{1}\cap T_{2}$, then $\llbracket v \rrbracket_e = \big(v|_{T_{1}} - v|_{T_{2}}\big)|_{e}$ if the direction of $\mathbf{n}_{e}$ goes from $T_{1}$ to $T_{2}$, and if $e\subset\partial\Omega$, then $\llbracket\cdot\rrbracket_e$ is the evaluation on $e$. 

Suppose that $T$ represents a triangle in $\mathcal{T}_{h}$.  Let $h_{T}$ and $\rho_{T}$ be the circumscribed radius and the inscribed circles radius of $T$, respectively.
Let $h := \max\limits_{T \in \mathcal{T}_{h}}h_{T}$ be the mesh size of $\mathcal{T}_{h}$. Let $P_l(T)$ denote the space of polynomials on $T$ of the total degree no more than~$l$. Similarly, we define the space $P_l(e)$ on an edge~$e$. We assume that $\{\mathcal{T}_{h}\}$ is a family of regular subdivisions, i.e.,
\begin{equation}\label{eq:regularity}
\max_{T\in \mathcal{T}_{h}}\frac{h_{T}}{\rho_{T}} \leq \gamma_{0},
\end{equation}
where $\gamma_{0}$ is a generic constant independent of $h$.

\section{A new $P_2-P_1$ finite element pair}
\label{sec:new element}
\subsection{Construction of a new finite element pair}
Let $T$ be a triangle with nodes $\{a_{i},a_{j},a_{k}\}$, and $e_{i}$ be an edge of $T$ opposite to the $i$-th vertex $a_{i}$; see Figure~\ref{fig:triangle}. Denote a unit vector normal to $e_{i}$ and a unit tangential vector of $e_{i}$ as  $\mathbf{n}_{T,e_{i}}$ and $\mathbf{t}_{T,e_{i}}$, respectively. 
\begin{figure}[htbp]
\centerline{\includegraphics[width=2in]{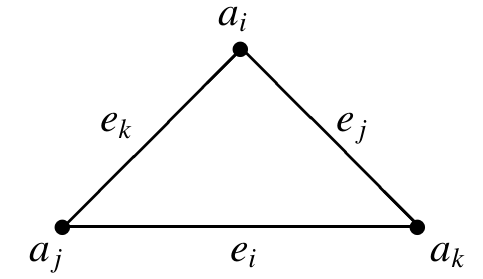}}
\caption{Illustration of a triangle and its nodes and edges.}\label{fig:triangle}
\end{figure}

The new vector $P_{2}$ element is defined by the triple
 $(T,P_{T},D_{T})$:
\begin{itemize}
\item[(1)] $T$ is a triangle;
\item[(2)] $P_{T} := (P_{2}(T))^{2}$;
\item[(3)] for any $\uv\in (H^{1}(T))^{2}$, the degrees of freedom on $T$, denoted by $D_{T}$, are
$$\Big\{  \  \fint_{e_{i}}\uv\cdot  \mathbf{n}_{T,e_{i}} \ud s, \ \fint_{e_{i}}\uv\cdot  \mathbf{n}_{T,e_{i}} (\lambda_{j}-\lambda_{k})\ud s,
\ \fint_{e_{i}}\uv\cdot  \mathbf{n}_{T,e_{i}} (-\lambda_{j}\lambda_{k}+\frac{1}{6})\ud s, \ \fint_{e_{i}}\uv\cdot  \mathbf{t}_{T,e_{i}} \ud s \ \Big\}_{i=1:3},$$
\end{itemize}
where $\{\lambda_{i}, \lambda_{j}\ \lambda_{k}\}$ represent the barycentric coordinates on $T$. 
The above triple is $P_{T}$-unisolvent. Particularly, we use $\uphi{}_{\mathbf{n}_{T,e_{i}},0}$, $\uphi{}_{\mathbf{n}_{T,e_{i}},1}$, $\uphi{}_{\mathbf{n}_{T,e_{i}},2}$, and $\uphi{}_{\mathbf{t}_{T,e_{i}},0}$ to represent the nodal basis functions associated with DOFs on $e_{i}$, and then
\begin{equation}\label{eq:basis on T}
\left\{ \quad
\begin{split}
\uphi{}_{\mathbf{n}_{T,e_{i}},0}= \ \ & \lambda_{j}(3\lambda_{j}-2)\frac{\mathbf{t}_{k}}{(\mathbf{n}_{i},\mathbf{t}_{k})} + \lambda_{k}(3\lambda_{k}-2)\frac{\mathbf{t}_{j}}{(\mathbf{n}_{i},\mathbf{t}_{j})} + 6\lambda_{j}\lambda_{k}\mathbf{n}_{i}; \\
\uphi{}_{\mathbf{n}_{T,e_{i}},1}=  \ \ & 3\lambda_{j}(3\lambda_{j}-2)\frac{\mathbf{t}_{k}}{(\mathbf{n}_{i},\mathbf{t}_{k})}  - 3\lambda_{k}(3\lambda_{k}-2)\frac{\mathbf{t}_{j}}{(\mathbf{n}_{i},\mathbf{t}_{j})}; \\
\uphi{}_{\mathbf{n}_{T,e_{i}},2}= \ \  & 30\lambda_{j}(3\lambda_{j}-2)\frac{\mathbf{t}_{k}}{(\mathbf{n}_{i},\mathbf{t}_{k})} + 30 \lambda_{k}(3\lambda_{k}-2)\frac{\mathbf{t}_{j}}{(\mathbf{n}_{i},\mathbf{t}_{j})}; \\
\uphi{}_{\mathbf{t}_{T,e_{i}},0}= \ \ & 6\lambda_{j}\lambda_{k}\mathbf{t}_{i}.
\end{split}
\quad
\right.
\end{equation}
The corresponding finite element space is defined by
\begin{multline*}\uV{}_{h} := \Big\{\uv{}_{h} \in \undertilde{L}{}^{2}(\Omega):\ \uv{}_{h}|_{T} \in (P_{2}(T))^{2}, \ \forall T \in \mathcal{T}_{h};\ \uv{}_{h}\cdot \mathbf{n}_{e} 
\mbox{ and } \fint_{e} \uv{}_{h}\cdot \mathbf{t}_{e} \ud s 
\mbox{ are continuous } \forall e\in \mathcal{E}_{h}^{i}  \Big\}.
\end{multline*}
Note that $\uV{}_{h}\subset \uH(\dv,\Omega)$ but $\uV{}_{h}\not\subset \uH{}^{1}(\Omega)$. 

Define a nodal interpolation operator $\Pi_{h}: \ \uH{}^{1}(\Omega) \rightarrow \uV{}_{h}$ such that for any $e\in \mathcal{E}_{h}$,
\begin{align*}
& \fint_{e}\Pi_{h}\uv\cdot \mathbf{n}_{e} \ w\ud s = \fint_{e}\uv\cdot \mathbf{n}_{e} \ w \ud s, \quad \forall w \in P_{2}(e), \\
&\fint_{e}\Pi_{h}\uv\cdot \mathbf{t}_{e} \ud s = \fint_{e}\uv\cdot \mathbf{t}_{e} \ud s.
\end{align*}

The operator $\Pi_{h}$ is locally defined, and the local space $\uV{}_{h}(T)$ restricted on $T$ is invariant under the Piola's transformation, i.e., it maps $\uV{}_{h}(T)$ onto $\uV{}_{h}(\hat{T})$, where $\hat{T}$ represents a reference triangle. Moreover, $\Pi_{h}$ preserves quadratic functions locally. Therefore, a combination of Lemmas~1.6 and~1.7~in~\cite{Brezzi;Fortin1991}, standard scaling arguments, and the Bramble-Hilbert lemma leads to the following approximation property of $\Pi_{h}$. 
\begin{proposition}\label{pro:interpolation}
If $\,0\leqslant k \leqslant 1 \leqslant s \leqslant 3$, then 
\begin{align}\label{eq:interpolation error}
|\uv - \Pi_{h} \uv|_{k,h}\lesssim h^{s-k}|\uv|_{s,\Omega}, \quad \forall \uv \in \uH^{s}(\Omega). 
\end{align}
Moreover, the following low order estimate is valid
\begin{align}\label{eq:interpolation lower error}
\|\uv - \Pi_{h}\uv \|_{0,\Omega} \lesssim h^{1/2}\|\uv\|_{0,\Omega}\|\uv\|_{1,\Omega}.
\end{align}
\end{proposition}

Assume $\Gamma_{D}$ to be a part of the boundary $\partial\Omega$. Define 
\begin{align*}
\uV{}_{hD} := \Big\{\uv{}_{h} \in \undertilde{L}{}^{2}(\Omega): \ \uv{}_{h}|_{T} \in (P_{2}(T))^{2}, \ \forall T \in \mathcal{T}_{h}; \ \uv{}_{h}\cdot \mathbf{n}_{e}  \  \mbox{and } \fint_{e} \uv{}_{h}\cdot \mathbf{t}_{e} \ud s
\\
 \mbox{ are continuous  } \mbox{for any } e\in \mathcal{E}_{h}^{i} \mbox{ and vanish for any } e\subset \Gamma_{D}  \Big\}.
\end{align*}
Specially, if $\Gamma_{D} = \partial\Omega$, $\uV{}_{hD}$ is written as $\uV{}_{h0}$. 
Define
\begin{align*}
Q_{h} := \{q\in L^{2}(\Omega): q|_{T} \in P_{1}(T), \forall T \in \mathcal{T}_{h}\}, \ \mbox{ and } \
Q_{h*}: = Q_{h}\cap L_{0}^{2}(\Omega).
\end{align*}
Evidently, $\dv\,\uV{}_{h}\subset Q_{h}$. Therefore, $\uV{}_{hD}\times Q_h$ and $\uV{}_{h0}\times Q_{h*}$ each forms a conservative pair. The stability and discrete Korn's inequality also hold. We firstly introduce an assumption on the triangulations. 

\medskip

\noindent{\bf Assumption A.}\label{ass:A}
Every triangle in $\mathcal{T}_{h}$ has at least one vertex in the interior of $\Omega$. 

The theorems below, which will be proved in the sequel subsections, hold on triangulations that satisfy {\bf Assumption A}. 

\begin{theorem}[Inf-sup conditions]\label{thm:inf-sup mixed boundary}
Let $\{\mathcal{T}_{h}\}$ be a family of  triangulations satisfying~{\bf Assumption~A}. Then 
\begin{align}
&\sup_{ \uv{}_{h} \in \uV{}_{hD}} 
\frac{\int_{\Omega} \dv\,\uv{}_{h}\,q_{h}\ud \Omega}{\|\uv{}_{h}\|_{1,h}} \gtrsim \|q\|_{0,\Omega}, \quad \forall q \in Q_{h},\quad\mbox{if }\ \Gamma_D\neq\partial\Omega,\label{eq:disc inf-sup}\\
&\sup_{ \uv{}_{h} \in \uV{}_{h0}} \frac{\int_{\Omega} \dv\, \uv{}_{h}\,q_{h}\ud \Omega}{\|\uv{}_{h}\|_{1,h}} \gtrsim \|q\|_{0,\Omega}, \quad \forall q \in Q_{h*}.\label{eq:disc inf-sup Vh0}
\end{align}
\end{theorem}

\begin{theorem}[Discrete Korn's inequality]\label{thm:Korn's ineq}
Let $\{\mathcal{T}_{h}\}$ be a family of  triangulations satisfying {\bf Assumption~A}. Let $\epsilon (\uv) := \frac{1}{2}[\nabla v + (\nabla v)^{T}]$. Then
\begin{align}\label{eq:Korn's ineq}
\sum_{T \in \mathcal{T}_{h}}\int_{T}|\epsilon (\uv)|^{2} \ud T \gtrsim |\uv|_{1,h}^{2}, \quad \forall \uv \in \uV{}_{hD}.
\end{align}
\end{theorem}

\subsection{Proof of inf-sup conditions} 
Note that the the commutativity $\dv\,\Pi_h\uw=P_{Q_{h*}}\dv\,\uw$ does not hold for all $\uw\in \uH{}^1_0(\Omega)$, where $P_{Q_{h*}}$ represents the $L^2$ projection onto $Q_{h*}$. To prove the inf-sup conditions~\eqref{eq:disc inf-sup} and \eqref{eq:disc inf-sup Vh0}, we adopt the macroelement technique by Stenberg~\cite{Stenberg1990technique} We postpone the proof of Theorem \ref{thm:inf-sup mixed boundary} after some technical preparations.

\subsubsection{Stenberg's macroelement technique}

A macroelement is a connected set of at least two cells in $\mathcal{T}_{h}$. And a macroelement partition of $\mathcal{T}_{h}$, denoted by $\mathcal{M}_{h}$, is a set of macroelements such that each triangle in 
of $\mathcal{T}_{h}$ is covered by at least one macroelement in $\mathcal{M}_{h}$.
\begin{definition}
{\rm
Two macroelements $M_{1}$ and $M_{2}$ are said to be equivalent if there exists a continuous one-to-one mapping $G: \ M_{1} \rightarrow M_{2}$, such that
\begin{itemize}
\item[(a)] $G(M_{1}) = M_{2};$
\item[(b)] if $M_{1} = \cup_{i = 1}^{m}T_{i}^{1}$, then $T_{i}^{2} = G(T_{i}^{1})$ with $i = 1:m$ are the cells of $M_{2}$.
\item[(c)] $G|_{T_{i}^{1}} = F_{T_{i}^{2}}\circ F_{T_{i}^{1}}^{-1}$, $i = 1:m$, where  $F_{T_{i}^{1}}$ and $F_{T_{i}^{2}}$ are the mappings from a reference element $\hat{T}$ onto $T_{i}^{1}$ and $T_{i}^{2}$, respectively.
\end{itemize}
}
\end{definition}
A class of equivalent macroelements is a set of all the macroelements which are equivalent to each other. Given a macroelement $M$, we denote
\begin{align*}%\label{eq:spaces on M}
\uV{}_{h0,M}: = \uV{}_{h0}(M), \quad Q_{h,M}: = Q_{h}(M), \quad \mbox{and} \quad Q_{h*,M}: = Q_{h*}(M).
\end{align*}
And we denote
\begin{align}\label{eq:spurious pressure}
N_{M} := \Big\{q_{h}\in Q_{h,M}: \ \int_{M} \dv\,\uv{}_{h} \ q_{h} \ud M = 0,\ \forall \uv{}_{h} \in \uV{}_{h0,M}\Big\}.
\end{align}
Stenberg's macroelement technique can be summarized as the following proposition.

\begin{proposition}{\cite[Theorem 3.1]{Stenberg1990technique}}\label{pro:macro stable tech}
Suppose there exist a macroelement partitioning $\mathcal{M}_{h}$ with a fixed set of equivalence classes $\mathbb{E}_{i}$ of macroelements, $i = 1,\ 2,\ \ldots,\ n$, a positive integer $N$ ($n$ and $N$ are independent of $h$), and an operator $\Pi: \uH{}^1_0(\Omega)\rightarrow \uV{}_{h0}$, such that
\begin{itemize}
\item[$(C_{1})$] for each $M\in \mathbb{E}_{i}$, $i = 1,\ 2,\ \ldots,\ n$, the space $N_{M}$ defined in~\eqref{eq:spurious pressure} is one-dimensional, which consists of functions that are constant on $M$;
\item[$(C_{2})$] each $M\in \mathcal{M}_{h}$ belongs to one of the classes $\mathbb{E}_{i}$, $i = 1,\ 2,\ \ldots,\ n$;
\item[$(C_{3})$] each $e\in \mathcal{E}_{h}^{i}$ is an interior edge of at least one and no more than $N$ macroelements;
\item[$(C_{4})$] for any $\uw\in \uH{}^1_0(\Omega)$, it holds that 
\begin{align*}
\sum\limits_{T\in \mathcal{T}_{h}} h_{T}^{-2}\|\uw-\Pi \uw\|_{0,T}^{2} + \sum\limits_{e\in \mathcal{E}_{h}^{i}} h_{e}^{-1}\|\uw-\Pi \uw\|_{0,e}^{2} \lesssim \|\uw\|_{1,\Omega}^{2}  \quad {and } \quad \|\Pi \uw\|_{1,h} \lesssim  \|\uw\|_{1,\Omega}.
\end{align*}
\end{itemize}
Then the stability~\eqref{eq:disc inf-sup Vh0} is valid.
\end{proposition}

\subsubsection{Technical lemmas}
In general, the main difficulty to design a stable mixed element stems from $(C_{1})$. We use the specific type of macroelements as below.
\begin{definition}\label{def:macroelement}{\rm
A macroelement, denoted by $M$, being a union of the $m$ cells that share exactly one common vertex in the interior of the macroelement, is called an {\bf ${\mathbf m}$-cell vertex-centred macroelement}, and {\bf${\mathbf m}$-macroelement} for short.
}\end{definition}

The set of interior edges and cells of $M$ are denoted by $\{e_{i}\}_{i = 1:m}$ and $\{T_{i}\}_{i = 1:m}$, respectively. Denote the lengths of interior edges by $\{d_{i}\}_{i=1:m}$ and the areas of cells by $\{S_{i}\}_{i = 1:m}$. Figure~\ref{fig:6 cell macroelement} gives an illustration of a {\bf${\mathbf 6}$-macroelement}. 

\begin{figure}[htbp]
\centerline{\includegraphics[width=2.35in]
{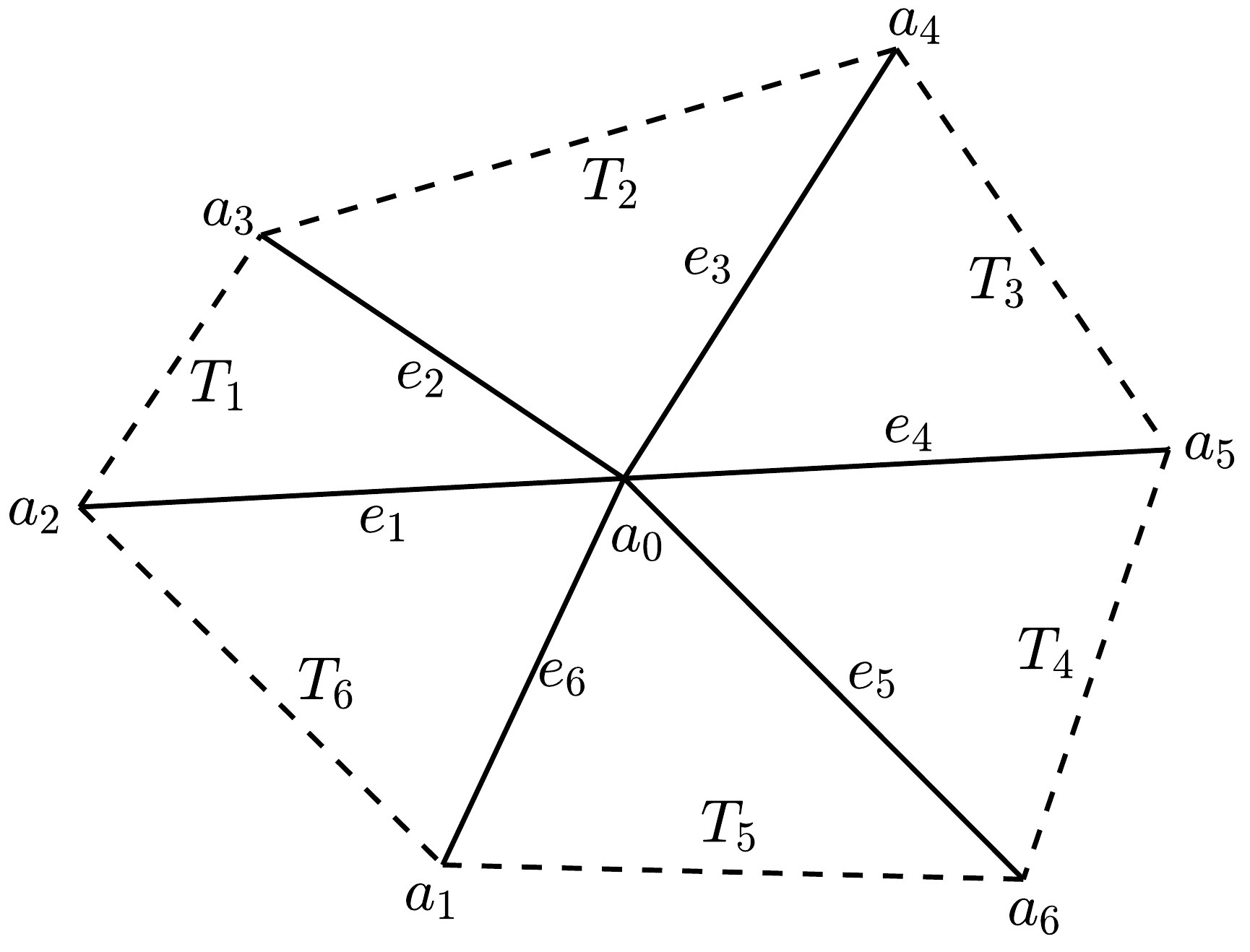}}
\caption{A macroelement composed of six cells with one interior vertex}\label{fig:6 cell macroelement}
\end{figure}

Below are concrete definitions of some local defined spaces on $M$ introduced in the previous context.
\begin{align*}%\label{eq:div free on M}
\uV{}_{h0,M} := \Big\{\uv{}_{h} \in \uL{}^{2}(M):\ \uv{}_{h}|_{T}&  \in (P_{2}(T))^{2},   \forall T \subset M, \  \uv{}_{h}\cdot \mathbf{n}_{e}  \mbox{ and }  \fint_{e} \uv{}_{h}\cdot \mathbf{t}_{e}  \ud s   \mbox{ are }\\
& \mbox{continuous across  interior edges and vanish on } \partial M \Big\},
\end{align*}
\begin{align*}
& Q_{h,M} = \Big\{q_{h}\in L^{2}(M): q_{h}|_{T} \in P_{1}(T), \forall T \in M\Big\},\\
& Q_{h*,M} = \Big\{q_{h}\in Q_{h,M}: \int_{M} q_{h} \ud M = 0\Big\}.
\end{align*}

For an {\bf${\mathbf m}$-macroelement} $M$, denote by ${\rm ker}(\dv, \uV{}_{h0,M}) := \Big\{\uv{}_{h} \in \uV{}_{h0,M}: \dv\,\uv{}_{h}  = 0\Big\}$ and ${\rm Im}\big(\dv, \uV{}_{h0,M}\big)= \dv(\uV{}_{h0,M})$. The main technical issue is the following lemma.
\begin{lemma}\label{lem:localkernel}
It holds that ${\rm dim} \big({\rm ker}(\dv, \uV{}_{h0,M})\big) \leqslant m+1$.
\end{lemma}
We postpone the technical proof of Lemma \ref{lem:localkernel} to Appendix~\ref{sec:app}.
\begin{lemma}\label{lem:conds C1}
Let $M$ be an {\bf${\mathbf m}$-macroelement}. Then $N_{M}$ is a one-dimensional space consisting of constant functions on $M$. 
\end{lemma}
\begin{proof}
First we show ${\rm Im}\big(\dv, \uV{}_{h0,M}\big)= Q_{h*,M}$.  As 
${\rm dim}\uV{}_{h0,M} = {\rm dim} \big({\rm ker}(\dv, \uV{}_{h0,M})\big) + {\rm dim} \big({\rm Im}(\dv, \uV{}_{h0,M})\big),$ we obtain ${\rm dim}\Big({\rm Im}\big(\dv, \uV{}_{h0,M}\big)\Big) \geqslant 3m-1$, where we have utilized ${\rm dim}\uV{}_{h0,M} = 4m$ by definition and ${\rm dim} \big({\rm ker}(\dv, \uV{}_{h0,M})\big) \leqslant m+1$ by Lemma~\ref{lem:localkernel}. From ${\rm dim} \big(Q_{h*,M}\big) = 3m-1$ and ${\rm Im}\big(\dv, \uV{}_{h0,M}\big) \subset  Q_{h*,M}$, we derive ${\rm Im}\big(\dv, \uV{}_{h0,M}\big)= Q_{h*,M}$. Then, for any $q_{h}\in N_{M}$, it holds that $\int_{M}p^{*}_{h}\ q_{h} \ud M = 0,$ for any $p^{*}_{h}\in Q_{h*,M}. $ Let $q_{h} = \overline{q}_{h} + q_{h}^{*}$, where $\overline{q}_{h} = \fint_{M}q_{h} \ud M$ and $q_{h}^{*}\in Q_{h*,M}$. Then $\int_{M}p^{*}_{h}\ q_{h}^{*} \ud M = 0, \ \forall p^{*}_{h}\in Q_{M,*}$, which yields $q_{h}^{*} = 0$ and hence $q_{h}= \overline{q}_{h}$. Therefore, $N_{M}$ is a one-dimensional space consisting of constant functions on $M$.  
\end{proof}

\begin{lemma}\label{lem:conds C2 C3 C4}
Let $\{\mathcal{T}_{h}\}$ be a family of  triangulations satisfying {\bf Assumption~A}. Each macroelement in $\mathcal{M}_{h}$ has one interior vertex. Then conditions $C_{2}$,  $C_{3}$, and $C_{4}$ in Theorem~\ref{pro:macro stable tech} are satisfied.
\end{lemma}
\begin{proof}
From the {\bf Assumption A} and the regularity~\eqref{eq:regularity} of $\mathcal{T}_{h}$, there exists a generic constant $n$, independent of $h$, such that condition $C_{2}$ holds.  If $e\in \mathcal{E}_{h}^{i}$, then at least one endpoint of $e$ is an interior vertex. Hence $e$ is an interior edge of at least one macroelement of $\mathcal{M}_{h}$. On the other hand, $e$ is interior to at most two macroelements, which occurs if both endpoints of $e$ are interior in $\Omega$. Therefore, condition $C_{3}$ also holds. By Proposition~\ref{pro:interpolation} and the well-known trace theorem~(see, e.g.,~\cite[Theorem 1.6.6]{Brenner;Scott2002}),  condition $C_{4}$ can be obtained. \end{proof}

\subsubsection{Proof of Theorem \ref{thm:inf-sup mixed boundary}}
By Lemma~\ref{lem:conds C1}, Lemma~\ref{lem:conds C2 C3 C4}, and Theorem~\ref{pro:macro stable tech}, it holds that $\uV{}_{h0}\times Q_{h*}$ satisfies the inf-sup condition~\eqref{eq:disc inf-sup Vh0}. 
The inf-sup stability of $\uV{}_{hD}\times Q_{h}$ is proved utilizing the technique introduced in~\cite{Kouhia;Stenberg1995}  by the following four steps.

\noindent{\bf Step 1.} 
Given $q_{h} \in Q_{h}$, let $\overline{q}_{h} = \frac{1}{|\Omega|}\int_{\Omega}q_{h}\ud \Omega$ and $q^{*}_{h} = q_{h}-\overline{q}_{h}$. Then $q^{*}_{h}\in L_{0}^{2}(\Omega)$ and 
\begin{align}\label{eq:orth decom of q}
\|q_{h}\|_{0,\Omega}^{2} = \|\overline{q}_{h}\|_{0,\Omega}^{2}+ \|q^{*}_{h}\|_{0,\Omega}^{2}.
\end{align}

\noindent{\bf Step 2.}
By~\eqref{eq:disc inf-sup Vh0}, there exists some $\uv^{*} \in \uV{}_{h0}$, such that 
\begin{align}\label{eq:step2}
(\dv\,\uv{}_{h}^{*}, q_{h}^{*})  \geqslant C_{1} \|q_{h}^{*}\|_{0,\Omega}^{2} \quad \mbox{and} \quad |\uv{}_{h}^{*}|_{1,h} = \|q_{h}^{*}\|_{0,\Omega},
\end{align}
where $C_{1}>0$ is a generic constant independent of $h$.

\noindent{\bf Step 3.}
Let $\Gamma_{N} : = \partial\Omega \backslash \Gamma_{D}$. Notice that $\overline{q}_{h}$ is constant in $\Omega$. Let $\overline{\uv}{}_{h} \in \uV{}_{hD}$ satisfy that $\overline{\uv}{}_{h} \cdot \mathbf{n_{e}} = C_{0}\overline{q}_{h}|_{e}$ for any $e \subset \Gamma_{N}$, and other degrees of freedom vanish. The value of $C_{0}$ is chosen such that $|\overline{\uv}{}_{h}|_{1,h} = \|\overline{q}_{h}\|_{0,\Omega}$. Then it holds with $C_{2}:=C_{0}\frac{|\Gamma_{N}|}{|\Omega|}$ that
\begin{equation}\label{eq:step3}
(\dv\,\overline{\uv}{}_{h}, \overline{q}_{h}) = \sum_{T \in \mathcal{T}_{h}}\int_{T}\dv\,\overline{\uv}{}_{h} \,\overline{q}_{h} \ud T   = \sum_{e\in \Gamma_{N}} \int_{e} \overline{\uv}{}_{h} \cdot \mathbf{n} \, \overline{q}_{h} \ud s   = C_{2}\|\overline{q}_{h}\|_{0,\Omega}^{2}.
\end{equation}

\noindent{\bf Step 4.} Let $\uv{}_{h} = \uv{}_{h}^{*}+\kappa\overline{\uv}{}_{h}$ with $\kappa = \frac{2C_{1}C_{2}}{(C_{2})^{2}+2}$. By the Schwarz inequality,  the elementary inequality, and~\eqref{eq:orth decom of q}--\eqref{eq:step3}, we obtain
\begin{equation*}
\begin{split}
(\dv\,\uv{}_{h}, q_{h}) &= (\dv\,\uv{}_{h}^{*}, q_{h}^{*}) + \kappa(\dv\,\overline{\uv}{}_{h}, \overline{q}_{h}) + \kappa(\dv\,\overline{\uv}{}_{h}, q_{h}^{*})  \\
& \geqslant C_{1}\|q_{h}^{*}\|_{0,\Omega}^{2} + C_{2}\kappa\|\overline{q}_{h}\|_{0,\Omega}^{2} + \kappa(\dv\,\overline{\uv}{}_{h}, q_{h}^{*}) \\
& \geqslant C_{1}\|q_{h}^{*}\|_{0,\Omega}^{2} + C_{2}\kappa\|\overline{q}_{h}\|_{0,\Omega}^{2} -\sqrt{2} \kappa\Big( \frac{C_{2}}{2\sqrt{2}}|\overline{\uv}{}_{h}|_{1,h}^{2} + \frac{\sqrt{2}}{2C_{2}}\|q_{h}^{*}\|_{0,\Omega}^{2} \Big)\\
&=(C_{1} - \frac{\kappa}{C_{2}})\|q_{h}^{*}\|_{0,\Omega}^{2} + \frac{C_{2}\kappa}{2}\|\overline{q}_{h}\|_{0,\Omega}^{2} = \frac{C_{1}(C_{2})^{2}}{(C_{2})^{2}+2}\|q_{h}\|_{0,\Omega}^{2}. 
\end{split}
\end{equation*}
From $|\uv{}_{h}^{*}|_{1,h} = \|q_{h}^{*}\|_{0,\Omega}$,  $|\overline{\uv}{}_{h}|_{1,h} = \|\overline{q}_{h}\|_{0,\Omega}$, and the Poincar\'{e} inequality, we have
$
\|\uv{}_{h}\|_{1,h} \lesssim \|q_{h}\|_{0,\Omega}.
$
This completes the proof of \eqref{eq:disc inf-sup} and Theorem \ref{thm:inf-sup mixed boundary}.

\subsection{Proof of discrete Korn's inequality}

To verify the discrete Korn's inequality, we follow the lines of~\cite{Kouhia;Stenberg1995} and firstly introduce an auxiliary element scheme constructed by adding element bubble functions to the the standard Bernardi-Raugel element~\cite{Bernardi;Raugel1985}. Denote 
 $$
 \undertilde{P}{}_{T}: = \big(P_{1}(T)\big)^{2} \oplus {\rm span}\big\{\lambda_{2}\lambda_{3}\mathbf{n}_{1},\ \lambda_{3}\lambda_{1}\mathbf{n}_{2}, \lambda_{1}\lambda_{2}\mathbf{n}_{3}\big\} \oplus \big({\rm span}\big\{\lambda_{1}\lambda_{2}\lambda_{3}\big\}\big)^{2}.
 $$ 
 Define
\begin{align*}
\uC{}_{h}:= \Big\{\uz{}_{h} \in \uH{}^{1}(\Omega): \ \uz{}_{h}|_{T}\in \undertilde{P}{}_{T}, \forall T \in \mathcal{T}_{h},\ \uz{}_{h}(a)  \mbox{ is continuous at any } a\in \mathcal{N}_{h}^{i},  
\\
\mbox{and} \fint_{e}\uz{}_{h}\cdot \mathbf{n}_{e} \ud s \mbox{ is continuos across any } e\in \mathcal{E}_{h}^{i}\Big\},
\end{align*}
 and $\uC{}_{h}^{N}: = \Big\{\uz{}_{h}\in \uC{}_{h}: \ \uz{}_{h}(a) = 0, \ \forall a \in \Gamma_{N} \mbox{ and }  \fint_{e}\uz{}_{h}\cdot \mathbf{n}_{e} \ud s = 0, \ \forall e\subset \Gamma_{N} \Big\}$,
 where $\Gamma_{N} = \partial\Omega\backslash\Gamma_{D}$. Here we are concerned about the case of $\Gamma_{D} \ne \partial \Omega$, in which the discrete Korn's inequality plays a crucial role for outflow conditions.

\subsubsection{Technical lemmas}
\begin{lemma}\label{lem:Fortin operator for BR element}
The element pair $\uC{}_{h}^{N} \times Q_{h}$ satisfies the inf-sup condition
\begin{align}\label{eq:auxiliary spaces stable}
\sup_{ \uz{}_{h} \in \uC{}_{h}^{N}} \frac{\int_{\Omega}\dv\,\uz{}_{h} \, q_{h}\ud \Omega}{|\uz{}_{h}|_{1,h}} \gtrsim \|q_{h}\|_{0,\Omega}, \quad \forall q_{h} \in Q_{h}.
\end{align}
\end{lemma}
\begin{proof}
Let $\uH{}_{N}^{1}(\Omega) : = \big\{\uv\in\uH{}^{1}(\Omega): \ \uv = 0\ \mbox{on} \ \Gamma_{N}\big\}.$
Define $\Pi_{C}: \ \uH{}_{N}^{1}(\Omega) \mapsto \uC{}_{h}^{N}$ by 
\begin{equation*}
\left\{
\begin{split}
& \quad \Pi_{C}\, \uv\, (a)  = R_{h}\,\uv\, (a), \quad \forall a \in \mathcal{N}_{h}, \\
 & \int_{e}(\Pi_{C}\,\uv - \uv)\cdot \mathbf{n}_{e}\ud s = 0, \quad \forall e\in \mathcal{E}_{h},\\
& \int_{T}x\ \dv\,(\Pi_{C}\,\uv - \uv)\ud T = 0, \quad \int_{T}y\ \dv\,(\Pi_{C}\,\uv - \uv)\ud T = 0, \quad \forall T\in \mathcal{T}_{h},
\end{split}
\right.
\end{equation*}
where $R_{h}$ represents the local $L^{2}$-projection given in~\cite[(A.53)--(A.54)]{Girault;Raviart1986}.
It can be verified directly that
$(\dv\,\Pi_{C}\,\uv, q_{h}) = (\dv\,\uv, q_{h})$ for any $q_{h}\in Q_{h}$, and $|\Pi_{C}\,\uv|_{1,h}\lesssim |\uv|_{1,h}$. Hence the stability~\eqref{eq:auxiliary spaces stable} is valid~\cite[Propositions 4.1 -- 4.2]{Fortin1977}.
\end{proof}
\begin{lemma}\label{lem:rhs eqs 0}
For any $\uv\in \uV{}_{hD}$ and $\uz \in \uC{}_{h}^{N}$, it holds that 
\begin{align}\label{eq:orthogonality}
\sum_{T \in \mathcal{T}_{h}}\int_{T}\nabla \uv : \curl\, \uz \ud T = 0.
\end{align}
\end{lemma}
\begin{proof}
Let subscripts $``\cdot_1"$ and $``\cdot_2"$ represent the components of the vector in the $x$ and $y$ directions, respectively. Integration by parts and direct calculation lead to
\begin{multline}\label{eq:rhs zero}
\sum_{T \in \mathcal{T}_{h}}\int_{T}\nabla \uv : \curl\, \uz \ud T = \sum_{T \in \mathcal{T}_{h}}\sum_{e \subset \partial T}\int_{e} (v_{1}\nabla z_{1}\cdot \mathbf{t}_{T,e} + v_{2}\nabla z_{2}\cdot \mathbf{t}_{T,e})\ud s
\\
=\sum_{T \in \mathcal{T}_{h}}\sum_{e \subset \partial T}\int_{e} (\uv \cdot \mathbf{n}_{T,e})(\ug{}_{T,e}\cdot \mathbf{n}_{T,e}) + (\uv \cdot \mathbf{t}_{T,e})(\ug{}_{T,e}\cdot \mathbf{t}_{T,e}) \ud s,
\end{multline}
where $\ug{}_{T,e}: =\nabla\uz \cdot  \mathbf{t}_{T,e}$, $\mathbf{t}_{T,e}$ is the counter-clockwise unit tangent vector of $T$ on $e$ and $\mathbf{n}_{T,e}$ represents the unit outer normal vector. Notice that $\uz|_{e} \in \big(P_{1}\big)^{2}+ {\rm span}\big\{\phi_{T,e}\cdot \mathbf{n}_{T,e}\big\}$, and $\phi_{T,e}$ is the quadratic bubble function associated with $e$ in $T$, we derive that $\ug{}_{T,e} \cdot \mathbf{t}_{T,e}$ is constant on each $e\subset \partial T$. We check the right hand part of \eqref{eq:rhs zero} case by case.

\noindent{\bf Case 1.} For $e\in \mathcal{E}_{h}^{i}$ with $T_{1}\cap T_{2} = e$. Utilizing the continuity of $\uv \cdot \mathbf{n}$, $\mathbf{n}_{T_{1},e} = - \,\mathbf{n}_{T_{2},e}$, $\mathbf{t}_{T_{1},e} = - \, \mathbf{t}_{T_{2},e}$, and $\ug{}_{T_{1},e} = -\, \ug{}_{T_{2},e}$ by $\uz \in \uH{}^{1}(\Omega)$, we obtain
\begin{align*}%\label{eq:interior edge part1}
\int_{e}(\uv \cdot \mathbf{n}_{e,T_{1}})(\ug{}_{T_{1},e}\cdot \mathbf{n}_{e,T_{1}})  +\int_{e}(\uv \cdot \mathbf{n}_{T_{2},e})(\ug{}_{T_{2},e}\cdot \mathbf{n}_{T_{2},e})  
= \int_{e}\uv \cdot \mathbf{n}_{e,T_{1}}(\ug{}_{T_{1},e}\cdot \mathbf{n}_{T_{1},e} -\ug{}_{T_{1},e}\cdot \mathbf{n}_{T_{1},e} ) \ud s =0.
\end{align*}
At the same time, utilizing the continuity of $\int_{e} \uv \cdot \mathbf{t}\ud s $ across interior edges, and noticing that $\ug{}_{T_{1},e} \cdot \mathbf{t}_{T_{1},e} =  \ug{}_{T_{2},e} \cdot \mathbf{t}_{T_{2},e} = c$, where 
$c$ represent a constant on $e$, we have
\begin{align*}%\label{eq:interior edge part2}
\int_{e}(\uv \cdot \mathbf{t}_{T_{1},e})(\ug{}_{T_{1},e}\cdot \mathbf{t}_{T_{1},e}) +\int_{e}(\uv \cdot \mathbf{t}_{T_{2},e})(\ug{}_{T_{2},e}\cdot \mathbf{t}_{T_{2},e})
= \big(\ug{}_{T_{1},e}\cdot \mathbf{t}_{T_{1},e}\big) \Big( \int_{e}\uv \cdot \mathbf{t}_{T_{1},e} \ud s + \int_{e}\uv \cdot \mathbf{t}_{T_{2},e} \ud s\Big) \ud s= 0.
\end{align*}

\noindent{\bf Case 2.} For $e\subset \Gamma_{D}$, $\uv \cdot \mathbf{n}_{e} = \int_{e}\uv \cdot \mathbf{t}_{e} \ud s = 0$, and $\ug{}_{T,e} \cdot \mathbf{t}_{e}$ is constant on $e\subset T$. Therefore,
\begin{align*}%\label{eq:Dirichlet egde}
\int_{e} (\uv \cdot \mathbf{n}_{T,e})(\ug{}_{T,e}\cdot \mathbf{n}_{T,e}) + (\uv \cdot \mathbf{t}_{T,e})(\ug{}_{T,e}\cdot \mathbf{t}_{T,e}) \ud s = 0, \quad \forall e \subset \Gamma_{D}.
\end{align*}

\noindent{\bf Case 3.} For $e\subset \Gamma_{N}$, we have $\uz|_{e}= \undertilde{0}$ by definition. Hence $\ug{}_{T,e} = \undertilde{0}$ for $e\subset T$, and
\begin{align*}%\label{eq:Neumann egde}
\int_{e} (\uv \cdot \mathbf{n}_{T,e})(\ug{}_{T,e}\cdot \mathbf{n}_{T,e}) + (\uv \cdot \mathbf{t}_{T,e})(\ug{}_{T,e}\cdot \mathbf{t}_{T,e}) \ud s = 0, \quad \forall e \subset \Gamma_{N}.
\end{align*}
Namely, the right-hand-side of~\eqref{eq:rhs zero} equals to zero. The proof is completed. 
 \end{proof}

\subsubsection{Proof of Theorem \ref{thm:Korn's ineq}}
For any $\uv\in \uV{}_{hD}$, 
$\epsilon(\uv)|_{T} = (\nabla \uv - \frac{1}{2}\rot\,\uv \ \chi)|_{T}$,
where $\chi = \bigl(\begin{smallmatrix}
0 & -1\\
1 & 0
\end{smallmatrix} \bigr)$. From~\eqref{eq:auxiliary spaces stable} and $\rot\,\uv \in Q_{h}$, there exists some $\uz \in \uC{}_{h}^{N}$, such that
$$ 
\int_{\Omega} \dv\,\uz \, q \ud \Omega  = \sum_{T\in\mathcal{T}_{h}}\int_{T} \rot\,\uv \, q \ud T, \quad \forall q \in Q_{h} \quad\mbox{and}\quad
 |\uz|_{1,h}  \lesssim \|\rot\,\uv\|_{0,\Omega} \lesssim |\uv|_{1,h}. 
$$
Therefore, 
$
\|\nabla \uv - \curl\, \uz\|_{0,\Omega} \leqslant |\uv|_{1,h} + |\uz|_{1,h} \lesssim |v|_{1,h},
$
and 
\begin{multline*}%\label{eq:epsilon tau inner product}%{split}
\sum_{T\in \mathcal{T}_{h}} \int_{T}\epsilon (\uv)
\,: \, (\nabla \uv - \curl\,\uz) \ud T  =  \sum_{T\in \mathcal{T}_{h}} \int_{T}(\nabla \uv - \frac{1}{2}\rot\,\uv \ \chi) \,: \, (\nabla \uv - \curl\,\uz) \ud T \\
 = \sum_{T\in \mathcal{T}_{h}} \int_{T}|\nabla \uv|^{2}\ud T - \frac{1}{2}\sum_{T\in \mathcal{T}_{h}} \int_{T} \rot\,\uv \,(\rot\,\uv - \dv\,\uz) \ud T  = \sum_{T\in \mathcal{T}_{h}} \int_{T}|\nabla \uv|^{2}\ud T = |\uv|_{1,h}^{2}.
\end{multline*}%{split}
Finally 
$\displaystyle%\begin{align*}
\big( \sum_{T\in \mathcal{T}_{h}} \int_{T} |\epsilon(\uv)|^{2} \ud T \big)^{\frac{1}{2}} \geqslant  \frac{\sum\limits_{T\in \mathcal{T}_{h}} \int_{T} \epsilon(\uv)\,:\, (\nabla \uv - \curl\, \uz) \ud T}{\|\nabla \uv - \curl\, \uz\|_{0,\Omega}} \gtrsim |v|_{1,h},
$ %\end{align*}
and the proof is completed.

\section{Application to Stokes problems}
\label{sec:Model problems}
\subsection{Application to the Stokes equations}
Consider the stationary Stokes system:
\begin{equation}\label{eq:Stokes eq}
\left\{
\begin{split}
-\varepsilon^{2} \Delta\,\uu  + \nabla\, p & = \uf  \quad \mbox{in} \ \Omega, \\
\dv\, \uu & =  g \quad \mbox{in} \ \Omega,  \\
\uu  & = 0, \quad \mbox{on} \ \Gamma_{D}.
\end{split}
\right.
\end{equation}
 For simplicity of presentation, we only consider the case of $\partial \Omega = \Gamma_{D}$ herein. Extensions to other boundary conditions follows directly.

 The discretization scheme of~\eqref{eq:Stokes eq} reads:
Find $(\uu{}_{h},p_{h})\in \uV{}_{h0}\times Q_{h*}$, 
such that
\begin{equation}\label{eq:discre Stokes eq}
\left\{
\begin{split} 
\varepsilon^{2}\big(\nabla_{h}\,\uu{}_{h}, \nabla_{h}\,\uv{}_{h}\big)  -( \dv\,\uv{}_{h}, p_{h}) & = \langle \uf,\uv{}_{h} \rangle \quad \forall \uv{}_{h}\in \uV{}_{h}, \\
(\dv\,\uu{}_{h}, q_{h} )& = \langle g,q_{h} \rangle \quad \forall q_{h}\in Q_{h}.
\end{split}
\right.
\end{equation}
Based on the discussions in Section \ref{sec:new element}, Brezzi's conditions can be easily verified, and \eqref{eq:discre Stokes eq} is uniformly well-posed with respect to $\varepsilon$ and $h$.

\begin{theorem}\label{thm:error estimates Stokes}
 Let $(\uu,p)$ and $(\uu{}_{h},p_{h})$ be the solutions of~\eqref{eq:Stokes eq} and \eqref{eq:discre Stokes eq}, respectively. The following estimates hold with $0 < r \leqslant 2$
\begin{align*}
 |\uu - \uu{}_{h}|_{1,h} & \lesssim h^{r}|\uu|_{r+1,\Omega} + h|\uu|_{2,\Omega}, \\
  \|p-p_{h}\|_{0,\Omega} & \lesssim h^{r}|p|_{r,\Omega} + \varepsilon^{2}\big(h^{r}|\uu|_{r+1,\Omega} + h|\uu|_{2,\Omega}\big). 
 \end{align*}
\end{theorem}
\begin{proof}
Since the mixed element is inf-sup stable, divergence-free, and $\uv{}_{h}\cdot \mathbf{n}$ is continuous, the following estimates are standard~\cite{Brezzi1974,Brezzi;Fortin1991,Crouzeix1973}:
\begin{align*}
&|\uu - \uu{}_{h}|_{1,h}  \lesssim \inf_{\uw{}_{h}\in \uV{}_{h}} |\uu-\uw{}_{h}|_{1,h} + \sup_{\uv{}_{h}\in  \uZ{}_{h}(0)}\frac{\big|\sum\limits_{e \in \mathcal{E}_{h}} \varepsilon^{2} \int_{e} (\nabla \uu \cdot \mathbf{n}_{e} )\cdot \jump{\uv{}_{h}} \ud s\big|}{\varepsilon^{2}|\uv{}_{h}|_{1,h}}, %\label{eq:u-uh Stokes}
\\
& \|p-p_{h}\|_{0,\Omega}  \lesssim \varepsilon^{2}|\uu-\uu{}_{h}|_{1,h} + \inf_{q_{h}\in Q_{h}}\|p-q_{h}\|_{0,\Omega} + \sup_{\uv{}_{h}\in  \uV{}_{h}}\frac{\big|\sum\limits_{e \in \mathcal{E}_{h}} \varepsilon^{2} \int_{e} (\nabla \uu \cdot \mathbf{n}_{e} )\cdot \jump{\uv{}_{h}} \ud s\big|}{|\uv{}_{h}|_{1,h}}. %\label{eq:p-ph Stokes}
\end{align*}
The term $\inf_{\uw{}_{h}\in \uV{}_{h}} |\uu-\uw{}_{h}|_{1,h}$ is bounded by the interpolation error.  Since $\int_{e}\uv{}_{h}\ud s$ is continuous across interior edges and vanish on $\partial\Omega$, a standard estimate similar to that of the Crouzeix and Raviart element~\cite[Lemma~3]{Ciarlet1978} leads to 
\begin{align}\label{eq:consis error}
\big|\sum_{e \in \mathcal{E}_{h}} \varepsilon^{2} \int_{e} (\nabla \uu \cdot \mathbf{n}_{e} )\cdot \jump{\uv{}_{h}} \ud s\big| \lesssim \varepsilon^{2} h |\uu|_{2,\Omega}|\uv{}_{h}|_{1,h}.
\end{align}
Hence we derive 
$$
 |\uu - \uu{}_{h}|_{1,h}  \lesssim h^{r}|\uu|_{r+1,\Omega} + h|\uu|_{2,\Omega} \ \mbox{ with } \ 0 < r \leqslant 2. 
$$
The above estimates together with $ \inf\limits_{q_{h}\in Q_{h}}\|p-q_{h}\|_{0,\Omega}  \lesssim h^{r}|p|_{r,\Omega}$ lead to that
$$\|p-p_{h}\|_{0,\Omega}  \lesssim h^{r}|p|_{r,\Omega} + \varepsilon^{2}\big(h^{r}|\uu|_{r+1,\Omega} + h|\uu|_{2,\Omega}\big) \ \mbox{ with } \ 0<r\leqslant 2. $$
The proof is completed. 
\end{proof}

\subsection{Application to the Darcy--Stokes--Brinkman equations}
Consider the Darcy--Stokes--Brinkman equations: 
\begin{equation}\label{eq:BrinkmanEqs}
\left\{
\begin{split}
-\varepsilon^{2}\Delta\,\uu + \uu + \nabla p & = \uf \quad \mbox{in} \ \Omega, \\
\dv\,\uu & = g \quad \mbox{in} \ \Omega,\\
\uu\cdot \mathbf{n} = 0, \ \varepsilon\, \uu \cdot \mathbf{t} & = 0 \quad \mbox{on} \ \partial \Omega,
\end{split}
\right.
\end{equation}
where $\varepsilon \in (0,1]$ is a parameter. When $\varepsilon $ is not too small and $g = 0$, it is a Stokes problem with an additional lower order term. When $\varepsilon  = 0$, the first equation becomes the Darcy's law for porous medium flow. Most classic mixed elements fail to converge uniformly with respect to $\varepsilon$ when applied to~\eqref{eq:BrinkmanEqs}~\cite{Mardal;Tai;Winther2002}.

The discretization scheme of~\eqref{eq:BrinkmanEqs} reads:
Find $(\uu{}_{h},p_{h})\in \uV{}_{h0}\times Q_{h*}$, 
such that
\begin{equation}\label{eq:discre Brinkman eq}
\left\{
\begin{split} 
\varepsilon^{2}\big(\nabla_{h}\,\uu{}_{h}, \nabla_{h}\,\uv{}_{h}\big) + (\uu{}_{h}, \uv{}_{h}) -( \dv\,\uv{}_{h}, p_{h}) & = \langle \uf,\uv{}_{h} \rangle \quad \forall \uv{}_{h}\in \uV{}_{h}, \\
(\dv\,\uu{}_{h}, q_{h} )& = \langle g,q_{h} \rangle \quad \forall q_{h}\in Q_{h}.
\end{split}
\right.
\end{equation}

Since the finite element pair is stable and conservative, Brezzi's conditions can be easily verified for \eqref{eq:discre Brinkman eq}, and it is uniformly well-posed with respect to $\varepsilon$ and $h$, provided $\int_{\Omega}g \ud \Omega = 0$. Robust convergence can be obtained both for smooth continuous solutions and for the case that the effect of the $\varepsilon$-dependent boundary layers is taken into account later.  
\begin{theorem}\label{thm:error Brinkman smooth}
If $\uu \in \uH{}^{r+1}(\Omega)\cap \uH{}_{0}^{1}(\Omega)$ and $p\in H^{r}(\Omega)\cap L_{0}^{2}(\Omega)$ with $0<r \leqslant 2$, then
\begin{align}
& \|\dv\,\uu - \dv\,\uu{}_{h}\|_{0,\Omega}   \lesssim  h^{r}|\uu|_{r+1,\Omega}, \label{eq:error div vel Brinkman}
 \\
& \|\uu - \uu{}_{h}\|_{0,\Omega} + \varepsilon |\uu-\uu{}_{h}|_{1,h} \lesssim h^{r}(\varepsilon + h)|\uu|_{r+1,\Omega} + \varepsilon h|\uu|_{2,\Omega},\label{eq:error vel Brinkman} 
\\
& \|p - p_{h}\|_{0,\Omega}   \lesssim    h^{r}|p|_{r,\Omega}  + h^{r}(\varepsilon + h)|\uu|_{r+1,\Omega} + \varepsilon h|\uu|_{2,\Omega}.\label{eq:error pre Brinkman}
\end{align}
\end{theorem}
\begin{proof}
Evidently, $\dv\,\uu{}_{h} = {\rm P}_{Q_{h*}}(\dv\,\uu)$, where ${\rm P}_{Q_{h*}}$ represents the $L^{2}$-projection into $Q_{h*}$. Therefore, the first inequality~\eqref{eq:error div vel Brinkman} follows from the estimation of the $L^{2}$-projection. For this conservative pair, the following estimates are standard (see, e.g.,~\cite{Brezzi;Fortin1991} and~\cite{Mardal;Tai;Winther2002}), 
\begin{multline}
\|\uu - \uu{}_{h}\|_{0,\Omega} + \varepsilon |\uu-\uu{}_{h}|_{1,h} \lesssim 
\\
\inf_{\uw{}_{h}\in \uV{}_{h}} (\|\uu - \uw{}_{h}\|_{0,\Omega} + \varepsilon |\uu-\uw{}_{h}|_{1,h} ) 
+ \sup_{\uv{}_{h}\in  \uZ{}_{h}(0)}\frac{\big|\sum\limits_{e \in \mathcal{E}_{h}} \varepsilon^{2} \int_{e} (\nabla \uu \cdot \mathbf{n}_{e} )\cdot \jump{\uv{}_{h}} \ud s\big|}{\varepsilon |\uv{}_{h}|_{1,h}},
\label{eq:u-uh Brinkman}
\end{multline}

\vskip -0.3cm

\begin{multline}
\|p-p_{h}\|_{0,\Omega}  \lesssim  \|\uu - \uu{}_{h}\|_{0,\Omega} + \varepsilon |\uu-\uu{}_{h}|_{1,h} 
\\
+\inf_{q_{h}\in Q_{h}}\|p-q_{h}\|_{0,\Omega} + \sup_{\uv{}_{h}\in  \uV{}_{h}}\frac{\big|\sum\limits_{e \in \mathcal{E}_{h}} \varepsilon^{2} \int_{e} (\nabla \uu \cdot \mathbf{n}_{e} )\cdot \jump{\uv{}_{h}} \ud s \big|}{\varepsilon |\uv{}_{h}|_{1,h}}.
\label{eq:p-ph Brinkman}
\end{multline}
Hence~\eqref{eq:error vel Brinkman} and~\eqref{eq:error pre Brinkman} are derived in a similar way as those in Theorem~\ref{thm:error estimates Stokes}.
\end{proof}

As is mentioned in~\cite{Mardal;Tai;Winther2002}, it may happen that $|\uu|_{2,\Omega}$ and $|\uu|_{3,\Omega}$ blow up as $\varepsilon$ tends to 0. In this case, the convergence
estimates given in Theorem~\ref{thm:error Brinkman smooth} will deteriorate, especially when
the solution of~\eqref{eq:BrinkmanEqs} has boundary layers. To derive a uniform convergence analysis of the discrete solutions, we assume that $\Omega$ is a convex polygon . Let $\big\{a_{j} : = (x_{j},y_{j})\big\}$ denote the set of corner nodes of $\Omega$. Define 
\begin{align*}
H_{+}^{1}(\Omega) : = \Big\{ g \in H^{1}(\Omega) \cap L_{0}^{2}(\Omega): \ \int_{\Omega}\frac{|g(x,y)|}{(x-x_{j})^{2} + (y-y_{j})^{2}} \ud \Omega < \infty, \ j = 1, 2, \ldots, l \,\Big\}, 
\end{align*}
with associated norm 
\begin{align*}
\|g\|_{1,+}^{2} : = \|g\|_{1,\Omega}^{2} + \sum_{j = 1}^{l} \int_{\Omega}\frac{|g(x,y)|}{(x-x_{j})^{2} + (y-y_{j})^{2}} \ud \Omega.
\end{align*}
Let $(\uu{}^{0},p^{0})$ solves~\eqref{eq:BrinkmanEqs} in the case of $\varepsilon = 0$. Then it is proved in~\cite{Mardal;Tai;Winther2002} that
\begin{equation}\label{eq:regularity Brinkman}
\begin{split}
\varepsilon^{2}\|\uu\|_{2,\Omega} + \varepsilon\|\uu\|_{1,\Omega} + \|\uu - \uu{}^{0}\|_{0,\Omega} + \|p - p^{0}\|_{1,\Omega}  + \varepsilon^{\frac{1}{2}}\|\uu{}^{0}\|_{1,\Omega} +  \varepsilon^{\frac{1}{2}}\|p^{0}\|_{1,\Omega}  \\
 \lesssim \varepsilon^{\frac{1}{2}} \big(\|\uf\|_{\rot} + \|g\|_{1,+}\big),
 \end{split}
\end{equation}
where $\|\cdot\|_{\rot} : = \|\cdot\|_{0,\Omega} + \|\rot(\cdot)\|_{0,\Omega}$ is the norm defined in $\uH{}(\rot,\Omega)$.
Following the technique in~\cite{Mardal;Tai;Winther2002}, we can obtain the following uniform convergence estimate.
\begin{theorem}\label{thm:uniform error Brinkman}
Let $(\uu,p)$ be the exact solution of~\eqref{eq:BrinkmanEqs} and $(\uu{}_{h},p_{h})$ be its approximation in $\uV{}_{h0}\times Q_{h*}$.
If $\uf \in \uH{}(\rot,\Omega)$ and $g \in H_{+}^{1}(\Omega)$, then 
\begin{align}
& \|\dv\,\uu - \dv\,\uu{}_{h}\| \lesssim h \|g\|_{1,\Omega}, \label{eq:uniform div Brinkman}\\ 
& \|\uu - \uu{}_{h}\|_{0,\Omega} + \varepsilon |\uu-\uu{}_{h}|_{1,h} \lesssim h^{\frac{1}{2}}(\|\uf\|_{\rot} + \|g\|_{1,+}), \label{eq:vel uniform Brinkman}
\\
&\|p-p_{h}\|_{0,\Omega}  \lesssim h^{\frac{1}{2}}(\|\uf\|_{\rot} + \|g\|_{1,+}).\label{eq:pre uniform Brinkman}
\end{align}
\end{theorem}
\begin{proof}
The first estimate is direct since $\dv\,\uu = g$.
To obtain the second inequality, we first analyze the interpolation error. By~\eqref{eq:interpolation error}, \eqref{eq:interpolation lower error}, and \eqref{eq:regularity Brinkman}, we have
\begin{equation}\label{eq:uniform Pi_h L2 norm}
\begin{split}
\|\uu-\Pi_{h}\uu\|_{0,\Omega} & \leqslant \|({\rm I} - \Pi_{h})(\uu-\uu{}^{0}\|_{0,\Omega} + \|\uu{}^{0}-\Pi_{h}\uu{}^{0}\|_{0,\Omega} \\
& \lesssim h^{\frac{1}{2}}\big(\|\uu - \uu{}^{0}\|_{0,\Omega}^{\frac{1}{2}}\|\uu - \uu{}^{0}\|_{1,\Omega}^{\frac{1}{2}} + h^{\frac{1}{2}}\|\uu{}^{0}\|_{1,\Omega} \big) \lesssim h^{\frac{1}{2}}(\|\uf\|_{\rot} + \|g\|_{1,+}).
\end{split}
\end{equation}
At the same time, 
\begin{equation}\label{eq:uniform Pi_h e norm}
\begin{split}
\varepsilon|\uu - \Pi{}_{h}\uu|_{1,h} &\lesssim \varepsilon |\uu|_{1,\Omega}^{\frac{1}{2}} |\uu - \Pi{}_{h}\uu|_{1,h}^{\frac{1}{2}} \lesssim  \varepsilon h^{\frac{1}{2}}|\uu|_{1,\Omega}^{\frac{1}{2}}|\uu|_{2,\Omega}^{\frac{1}{2}} \lesssim h^{\frac{1}{2}}(\|\uf\|_{\rot} + \|g\|_{1,+}),
\end{split}
\end{equation}
where we utilize $ \varepsilon|\uu|_{1,\Omega}^{\frac{1}{2}}|\uu|_{2,\Omega}^{\frac{1}{2}}  \lesssim  \varepsilon^{\frac{1}{2}}|\uu|_{1,\Omega} +  \varepsilon^{\frac{3}{2}}|\uu|_{2,\Omega} \lesssim \|\uf\|_{\rot} + \|g\|_{1,+}$. 

By the continuity of $\uv{}_{h}\cdot \mathbf{n}_{e}$ and $\fint_{e}\uv{}_{h}\cdot \mathbf{t}_{e}\ud s$, a standard estimate (see, e.g.,~\cite[Lemma 5.1]{Mardal;Tai;Winther2002}) yields
$\sum\limits_{e \in \mathcal{E}_{h}} \varepsilon^{2} \int_{e} (\nabla \uu \cdot \mathbf{n}_{e} )\cdot \jump{\uv{}_{h}} \ud s \lesssim \varepsilon^{2}h^{\frac{1}{2}}|\uu{}|_{1,\Omega}^{\frac{1}{2}} |\uu{}|_{2,\Omega}^{\frac{1}{2}} |\uv{}_{h}|_{1,h}.
$ Then we derive
\begin{equation}\label{eq:uniform consis error}
\frac{\big|\sum\limits_{e \in \mathcal{E}_{h}} \varepsilon^{2} \int_{e} (\nabla \uu \cdot \mathbf{n}_{e} )\cdot \jump{\uv{}_{h}} \ud s\big|}{\varepsilon |\uv{}_{h}|_{1,h}} \lesssim \varepsilon h^{\frac{1}{2}}|\uu|_{1,\Omega}^{\frac{1}{2}}|\uu|_{2,\Omega}^{\frac{1}{2}}  \lesssim h^{\frac{1}{2}}(\|\uf\|_{\rot} + \|g\|_{1,+}).
\end{equation}
A combination of\eqref{eq:u-uh Brinkman}, \eqref{eq:uniform  Pi_h L2 norm}, \eqref{eq:uniform Pi_h e norm}, and~\eqref{eq:uniform consis error} leads to 
\begin{align}\label{eq:uniform vel error}
 \|\uu - \uu{}_{h}\|_{0,\Omega} + \varepsilon|\uu-\uu{}_{h}|_{1,h} \lesssim h^{\frac{1}{2}}(\|\uf\|_{\rot} + \|g\|_{1,+}).
\end{align}
 Again from~\eqref{eq:regularity Brinkman} and notice that $\varepsilon <1$, we have
 \begin{equation}\label{eq:uniform Pi pressure}
 \|p-\Pi_{Q_{h*}}p\|_{0,\Omega} \lesssim h|p|_{1,\Omega} \lesssim h|p-p^{0}|_{1,\Omega} + h|p^{0}|_{1,\Omega} \lesssim h(\|\uf\|_{\rot} + \|g\|_{1,+}).
 \end{equation}
Hence, by~\eqref{eq:p-ph Brinkman}, \eqref{eq:uniform vel error}, and~\eqref{eq:uniform Pi pressure}, the last estimate~\eqref{eq:pre uniform Brinkman} is derived. 
\end{proof}

\section{Numerical Experiments}
\label{sec:numerical experiments}
In this section, we carry out numerical experiments to validate the theory and illustrate the capacity of the newly proposed element pair. Examples are given as illustrations from different perspectives.
\begin{itemize}
\item  Examples 1 and 2 test the method with the Stokes problem, especially its robustness with respect to the Reynolds' number and to the triangulations; 
\item  Examples 3 and 4 test the method with the Darcy--Stokes--Brinkman equation, especially the robustness with respect to the the small parameter, for smooth solutions as well as solutions with sharp layers;
\item Examples 5 and 6 test the method with the incompressible Navier--Stokes equation, regarding evolutionary and steady states. 
\end{itemize}

Three kinds of $P_2-P_1$ pairs are involved in the experiments, namely,
\begin{itemize}
\item[{TH:}] the Taylor-Hood element pair with continuous $P_{2}$ functions for the velocity space and continuous $P_{1}$ functions for the pressure space;
\item[{SV:}] the Scott-Vogelius element pair with continuous $P_{2}$ functions for the velocity space and discontinuous $P_{1}$ functions for the pressure space;
\item[{NPP:}] the newly proposed $P_{2}-P_{1}$ element pair.
\end{itemize}

All simulations are performed on uniformly refined grids. For the SV pair, an additional barycentric refinement is applied on each grid to guarantee the stability.

\medskip

\paragraph{\bf Example 1} This example was suggested in~\cite{Neilan2017} to illustrate the non-pressure-robustness of classical elements. Let $\Omega =(0,1)^{2}$. Consider the Stokes equations~\eqref{eq:Stokes eq} with $\varepsilon^{2} = 1$, $g=0$, and $\uf = (0,Ra(1-y+3y^{2}))^{T}$, where $Ra>0$ represents a parameter. No-slip boundary conditions are imposed on $\partial\Omega$. 
 The exact solution pair is
$\uu = \undertilde{0}$ and $ p = Ra(y^{3}-\frac{y^{2}}{2} + y -\frac{7}{12}).$ 
 
\begin{figure}[H]
\begin{minipage}[t]{0.5\linewidth}
\centerline{\includegraphics[width=2.75in]{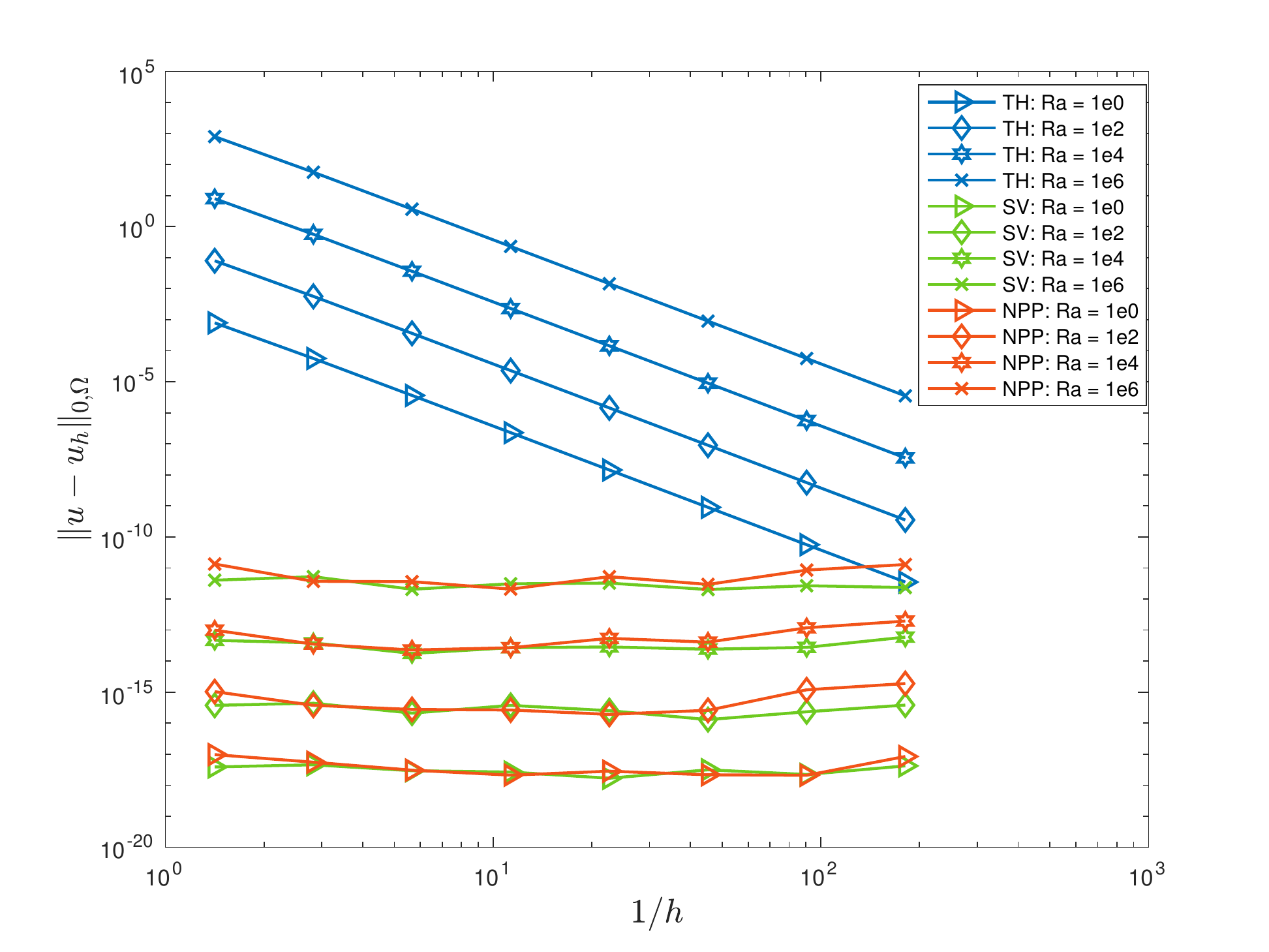}}
\end{minipage}%
\begin{minipage}[t]{0.5\linewidth}
\centerline{\includegraphics[width=2.75in]{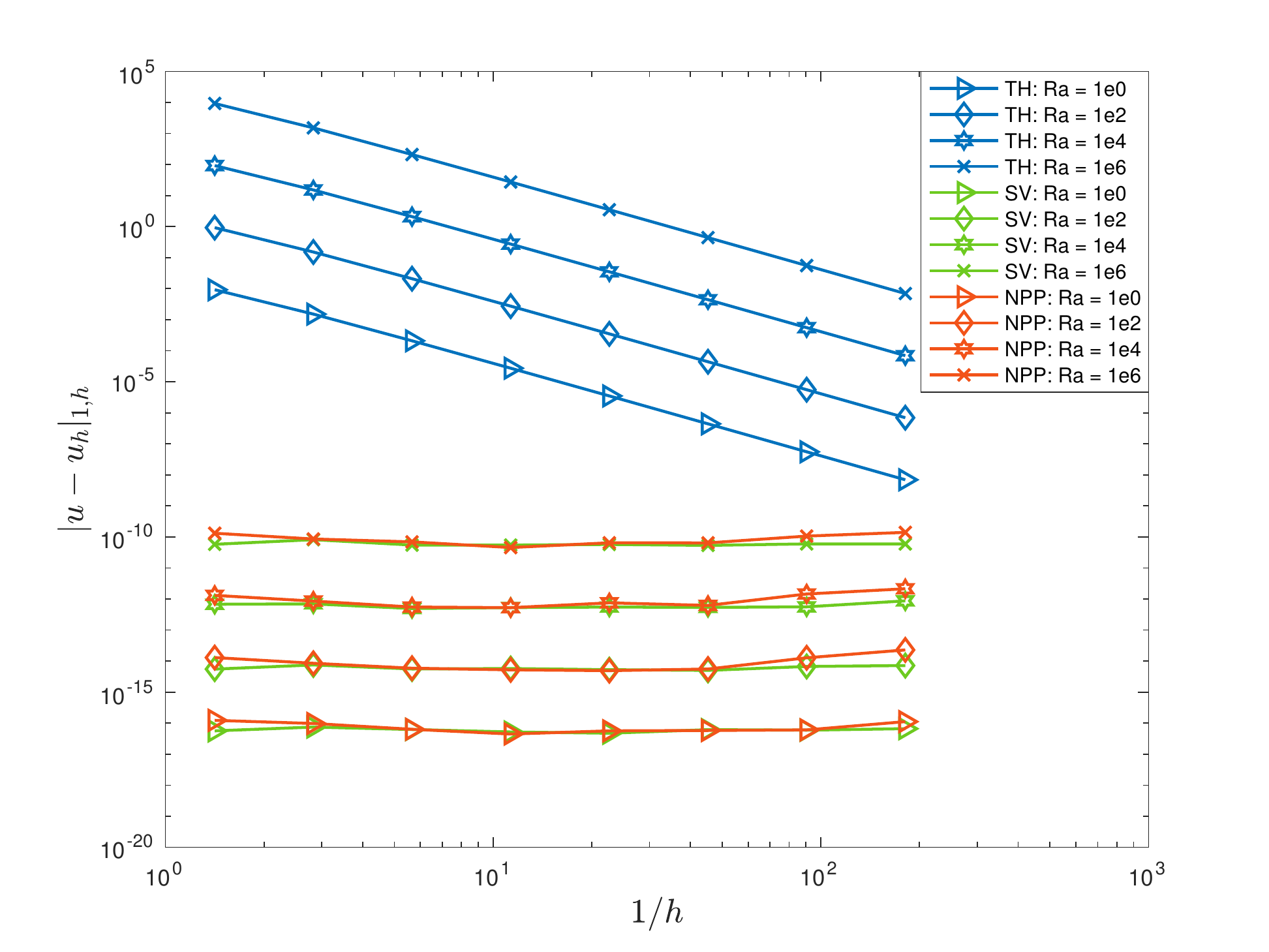}}
\end{minipage}%
\caption{Example 1: Velocity errors in the no-flow  Stokes equations by the TH, SV, and NPP pairs.}\label{fig:no flow err}
\end{figure}

For the continuous problem, different values of $Ra$  result in different exact pressures and the same exact velocity vector. As is shown in Figure~\ref{fig:no flow err}, for both the SV and NPP pairs, the numerical velocities are very close to zero for different values of $Ra$. However, for the TH pair, the discrete velocity is far from zero, even when $Ra = 1$. It demonstrates the advantage of pressure-robust pairs especially for problems with large pressures.

\medskip

\paragraph{\bf Example 2} This example was also introduced in~\cite{Neilan2017}.  Let $\Omega = (0,4)\times (0,2)\backslash [2,4]\times[0,1]$. Consider a flow with Coriolis forces with the following form 
\begin{equation*}
\left\{
\begin{split}
-\varepsilon^{2} \Delta\,\uu + \nabla p + 2\, \uw \times \uu = \uf  \quad \mbox{in } \Omega,  \\
\dv\,\uu = 0 \quad \mbox{in } \Omega,
\end{split}
\right.
\end{equation*}
where $\uw = (0,0,w)^{T}$ is a constant angular velocity vector. Changing the magnitude $w$ will change only the exact pressure, and not the true velocity solution. Dirichlet boundary conditions are imposed on $\partial \Omega$; see Figure~\ref{fig:grids Ex2} (Left). The computed domain and initial unstructured grid are depicted in Figure~\ref{fig:grids Ex2}. Simulations were performed with $\varepsilon^{2} = 0.01$, while $w= 100$ or $w= 1000$. 

\begin{figure}[htbp]
\centering
\begin{minipage}[t]{0.33\linewidth}
\centerline{\includegraphics[width=2in]{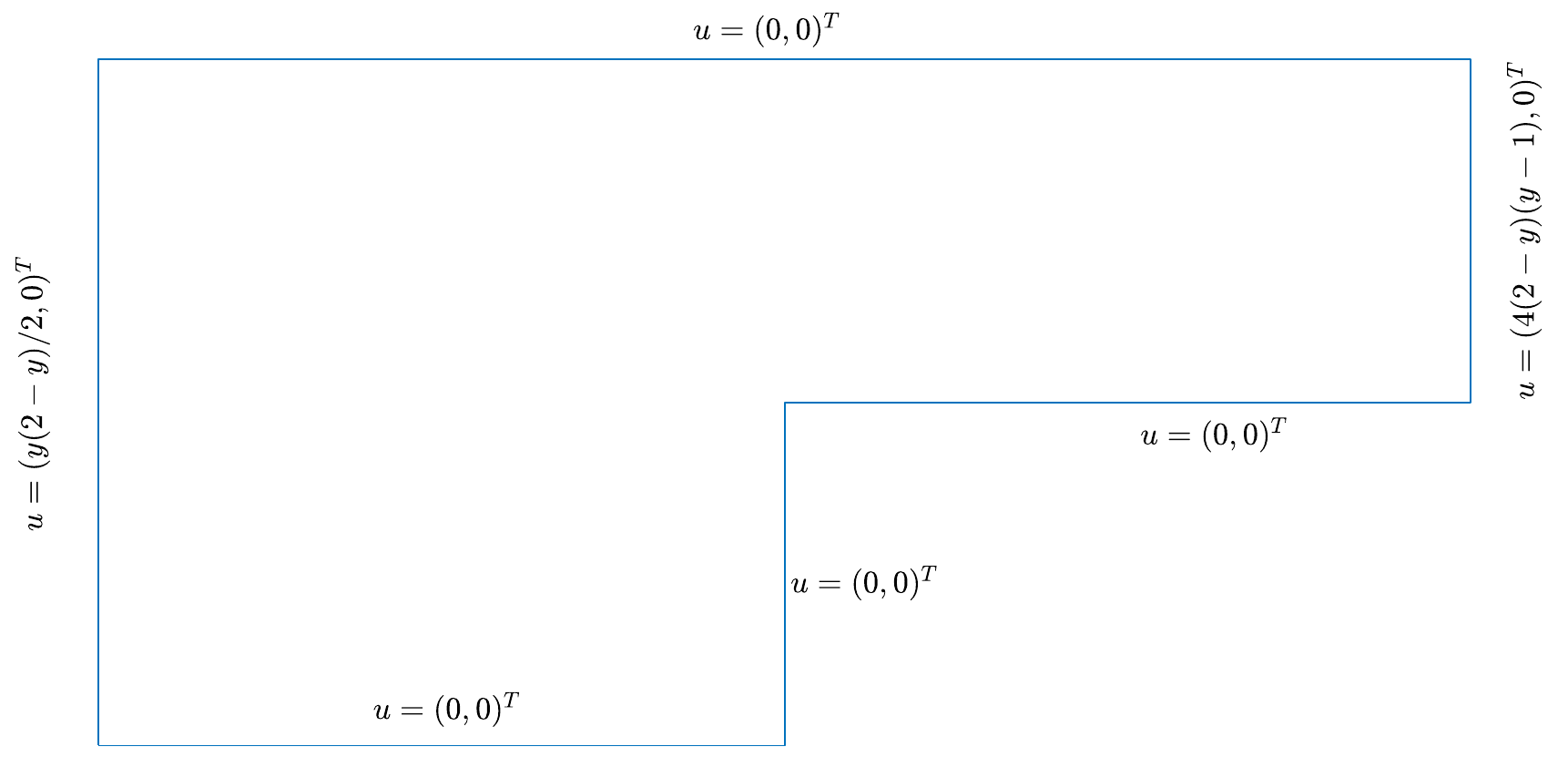}}
\end{minipage}%
\begin{minipage}[t]{0.33\linewidth}
\centerline{\includegraphics[width=1.9in]{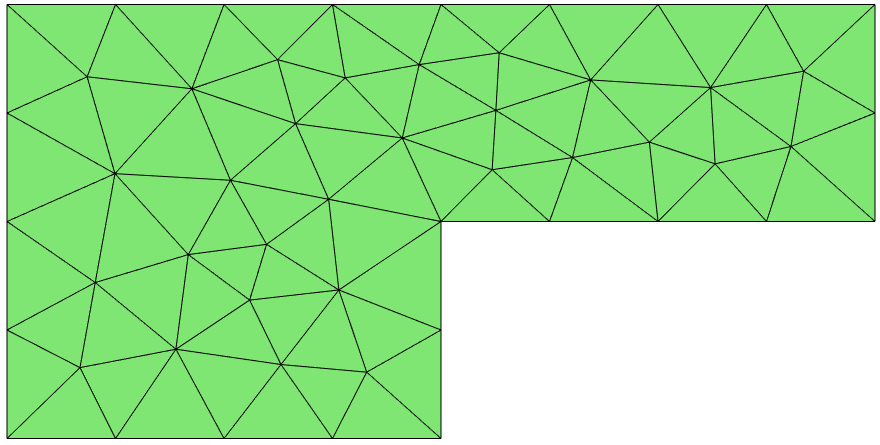}}
\end{minipage}%
\begin{minipage}[t]{0.33\linewidth}
\centerline{\includegraphics[width=1.9in]{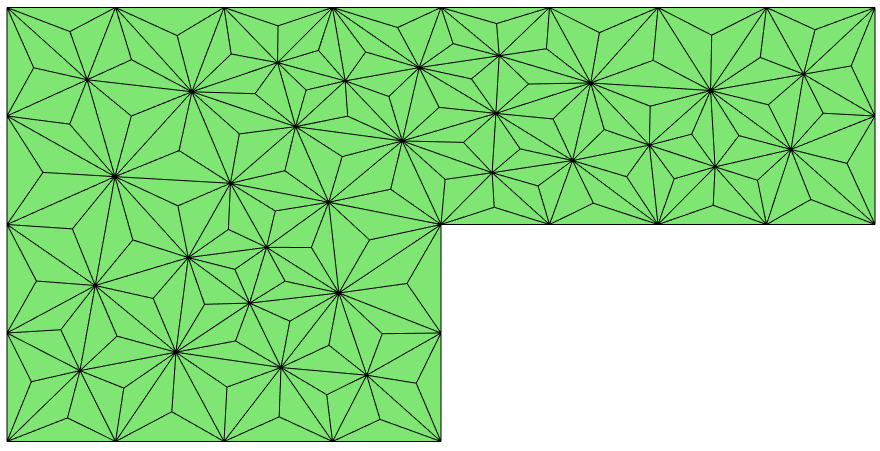}}
\end{minipage}%
\caption{Example~2: Forward facing step domain, unstructured mesh (level 1), and unstructured barycentric mesh for the SV pair (level 1).}\label{fig:grids Ex2}
\end{figure}

Computed velocities (speed) with $w= 100$ and $w= 1000$ are depicted in Figures~\ref{fig:w100 Ex2} and~\ref{fig:w1000 Ex2}, respectively. The solutions computed with the SV and NPP pairs are considerably more accurate compared with the TH pair.
Moreover, when $w = 1000$, the advantages of divergence-free elements are more obvious on the third level grid.

\begin{figure}[htbp]
\begin{minipage}[t]{0.33\linewidth}
\centerline{\includegraphics[width=1.9in]{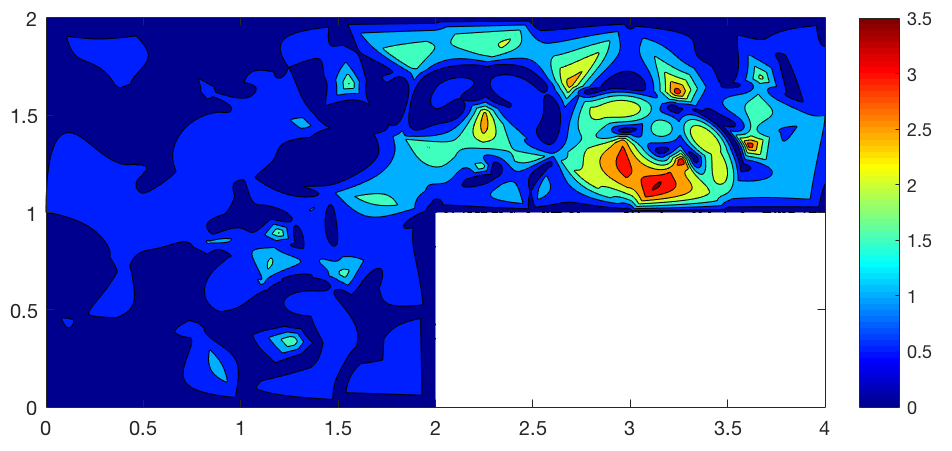}}
\end{minipage}%
\begin{minipage}[t]{0.33\linewidth}
\centerline{\includegraphics[width=1.9in]{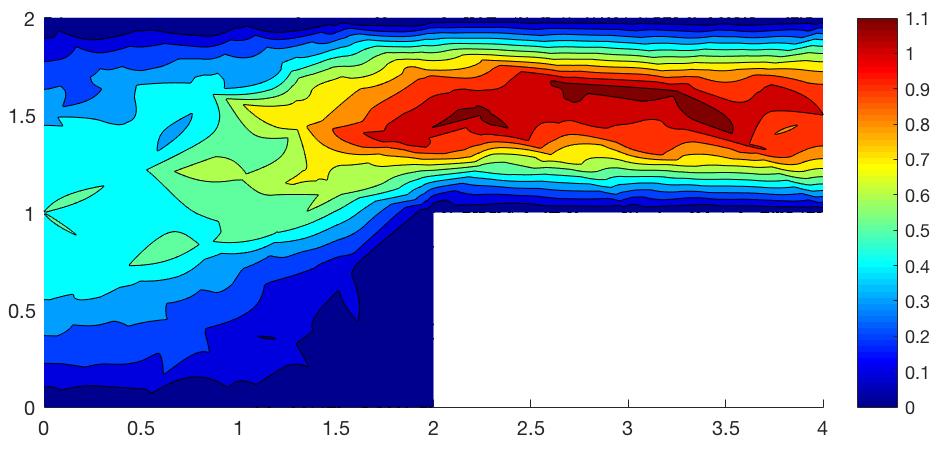}}
\end{minipage}%
\begin{minipage}[t]{0.33\linewidth}
\centerline{\includegraphics[width=1.9in]{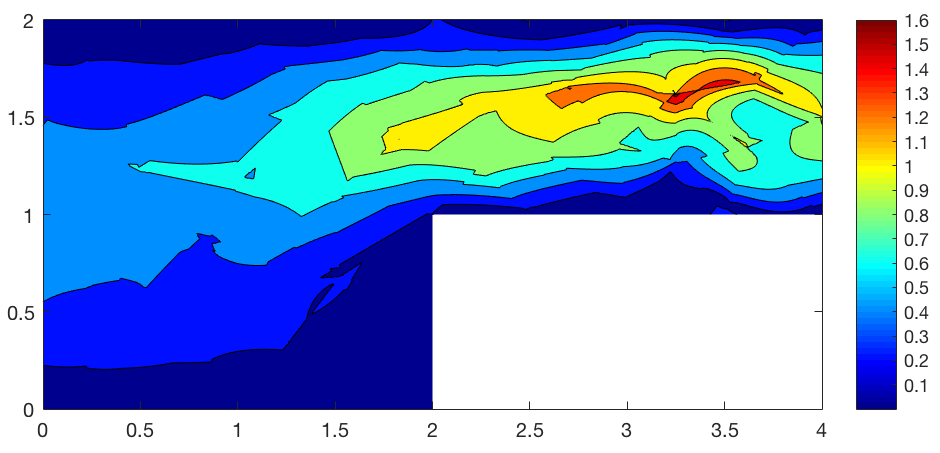}}
\end{minipage}%
 	
\vspace*{8pt}
	
\begin{minipage}[t]{0.33\linewidth}
\centerline{\includegraphics[width=1.9in]{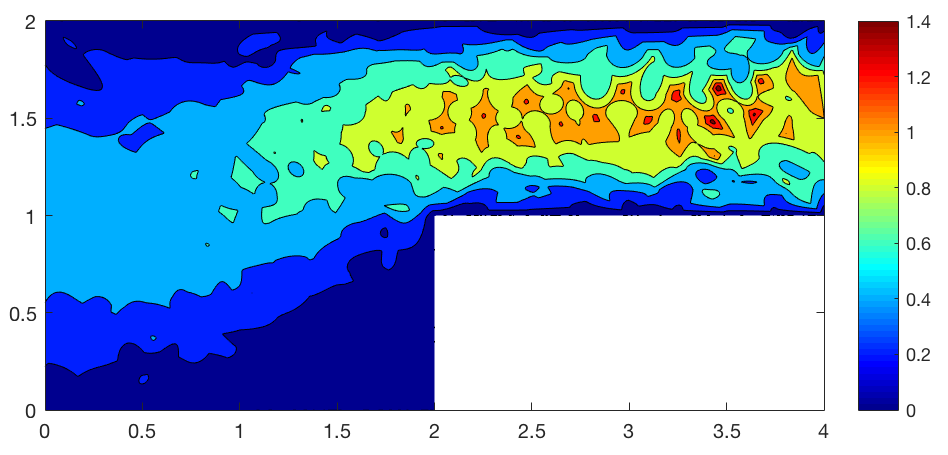}}
\end{minipage}%
\begin{minipage}[t]{0.33\linewidth}
\centerline{\includegraphics[width=1.9in]{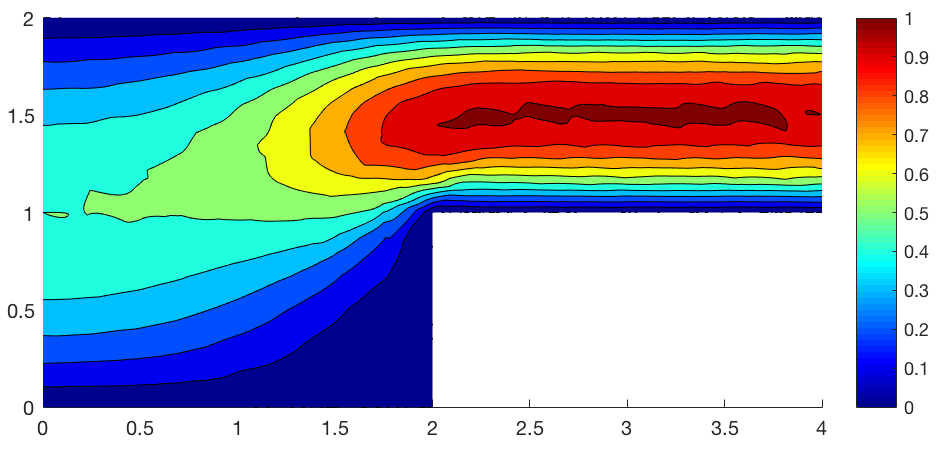}}
\end{minipage}%
\begin{minipage}[t]{0.33\linewidth}
\centerline{\includegraphics[width=1.9in]{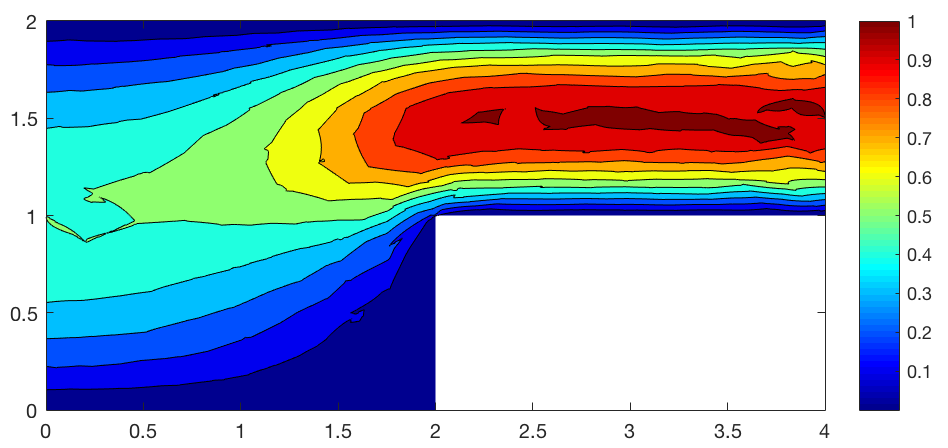}}
\end{minipage}%
	
\vspace*{8pt}

\begin{minipage}[t]{0.33\linewidth}
\centerline{\includegraphics[width=1.9in]{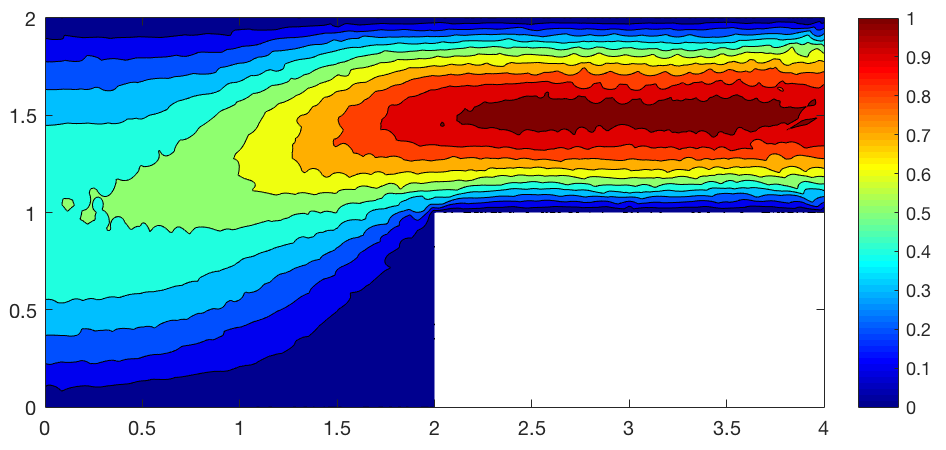}}
\end{minipage}%
\begin{minipage}[t]{0.33\linewidth}
\centerline{\includegraphics[width=1.9in]{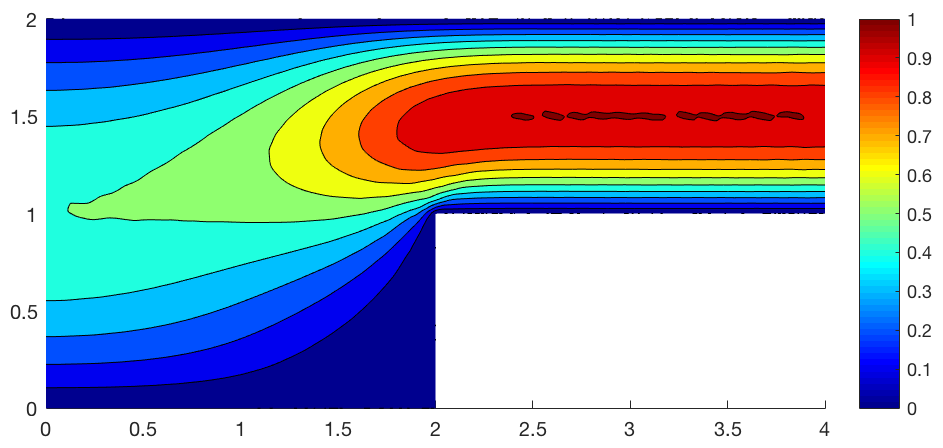}}
\end{minipage}%
\begin{minipage}[t]{0.33\linewidth}
\centerline{\includegraphics[width=1.9in]{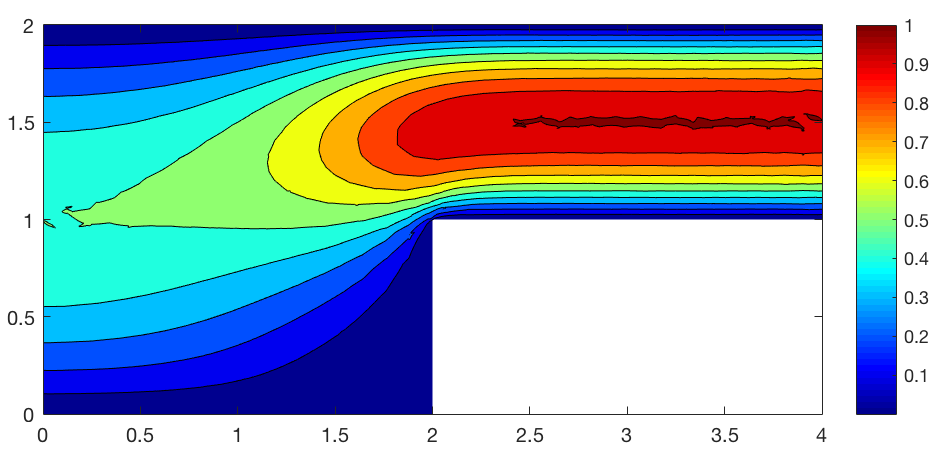}}
\end{minipage}	%
\caption{Example~2 ($w =100$ and $\varepsilon^{2}= 0.01$): Speed obtained by the TH pair (left column), the SV pair (middle column), and the NPP pair (right column); rows~1~--~3 are results on meshes 1--  3.}\label{fig:w100 Ex2}
\end{figure}

{
\begin{figure}[H]
\begin{minipage}[t]{0.33\linewidth}
\centerline{\includegraphics[width=1.9in]{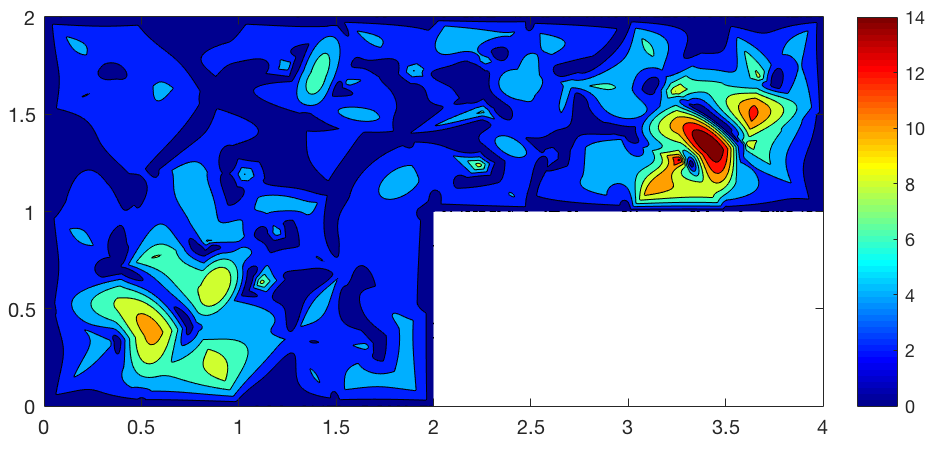}}
\end{minipage}%
\begin{minipage}[t]{0.33\linewidth}
\centerline{\includegraphics[width=1.9in]{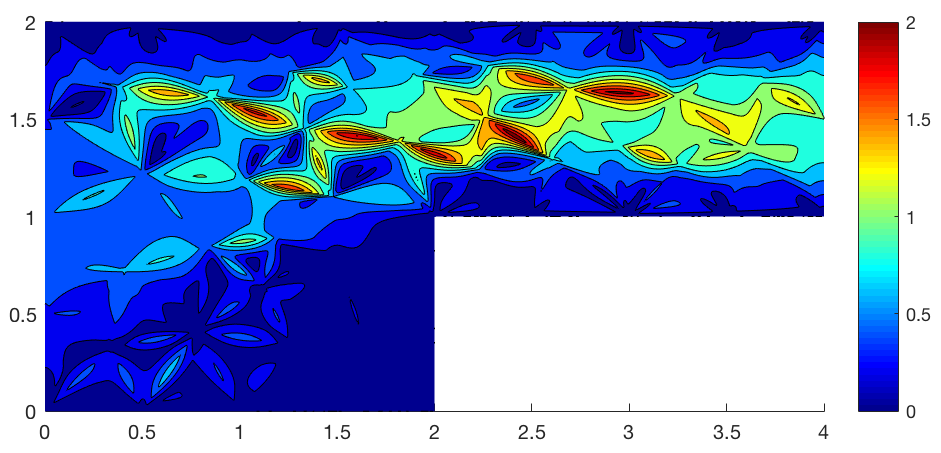}}
\end{minipage}%
\begin{minipage}[t]{0.33\linewidth}
\centerline{\includegraphics[width=1.9in]{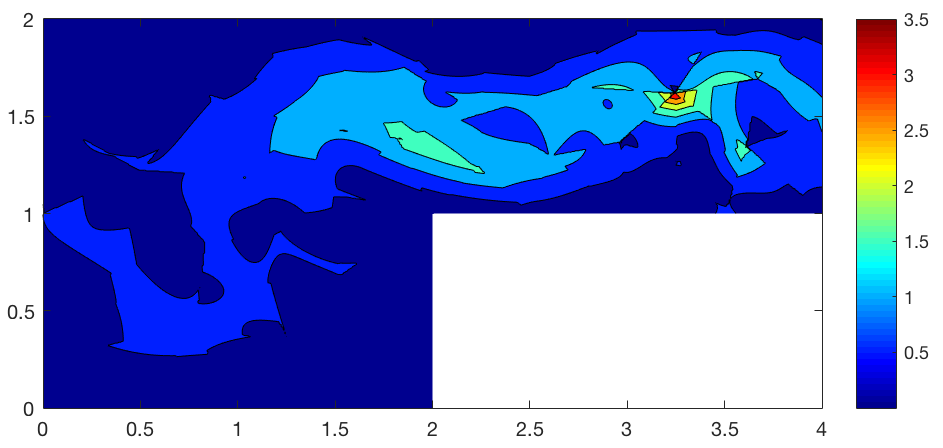}}
\end{minipage}%

\begin{minipage}[t]{0.33\linewidth}
\centerline{\includegraphics[width=1.9in]{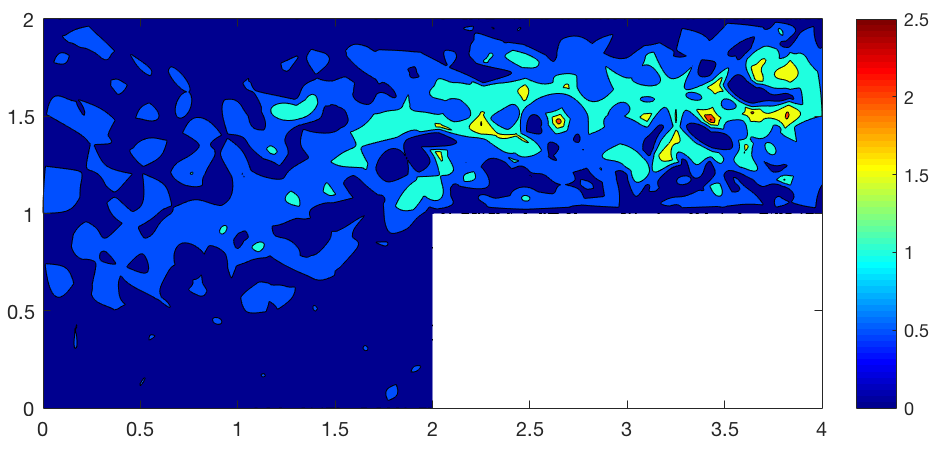}}
\end{minipage}%
\begin{minipage}[t]{0.33\linewidth}
\centerline{\includegraphics[width=1.9in]{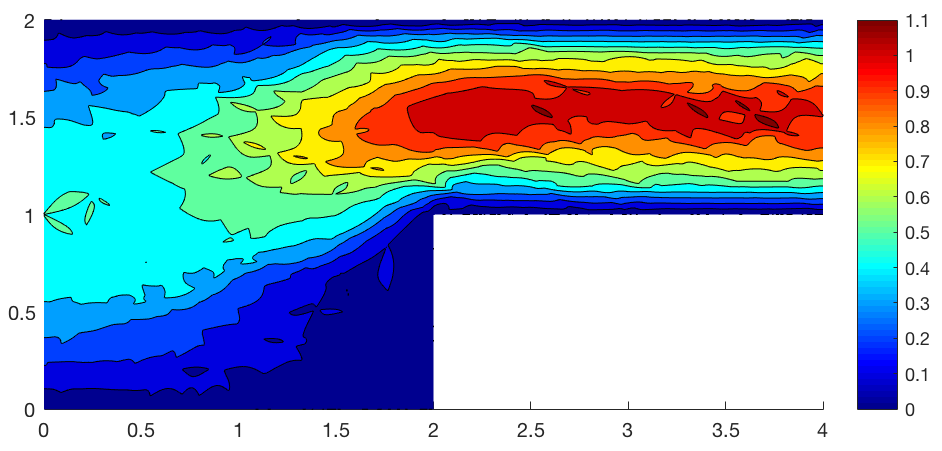}}
\end{minipage}%
\begin{minipage}[t]{0.33\linewidth}
\centerline{\includegraphics[width=1.9in]{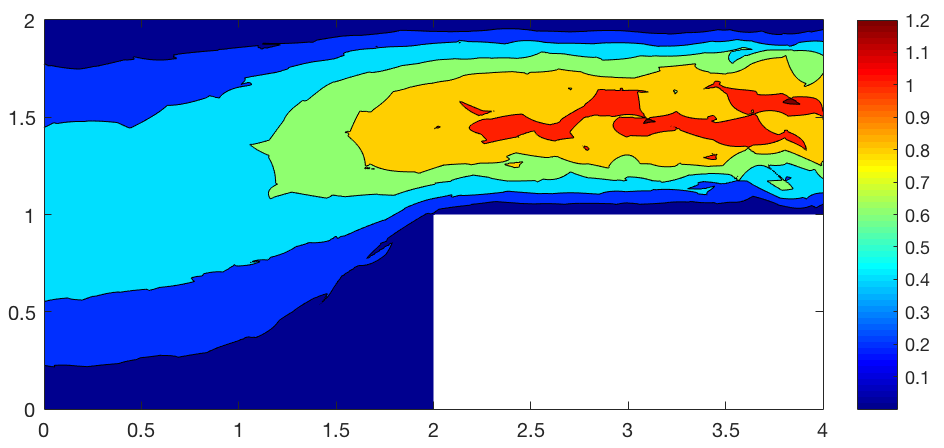}}
\end{minipage}%

\vspace*{8pt}
	
\begin{minipage}[t]{0.33\linewidth}
\centerline{\includegraphics[width=1.9in]{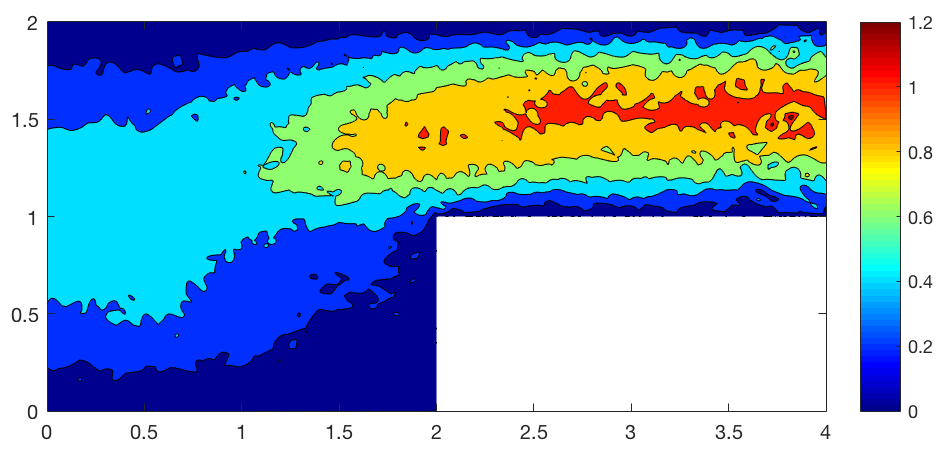}}
\end{minipage}%
\begin{minipage}[t]{0.33\linewidth}
\centerline{\includegraphics[width=1.9in]{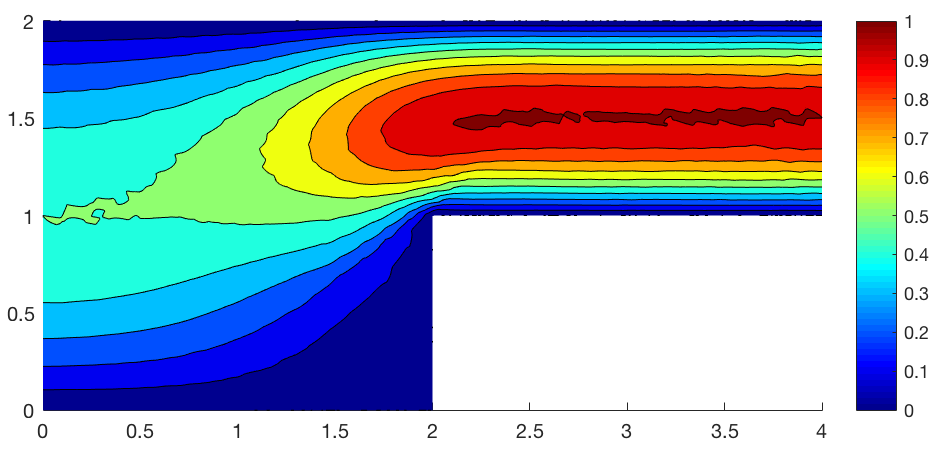}}
\end{minipage}%
\begin{minipage}[t]{0.33\linewidth}
\centerline{\includegraphics[width=1.9in]{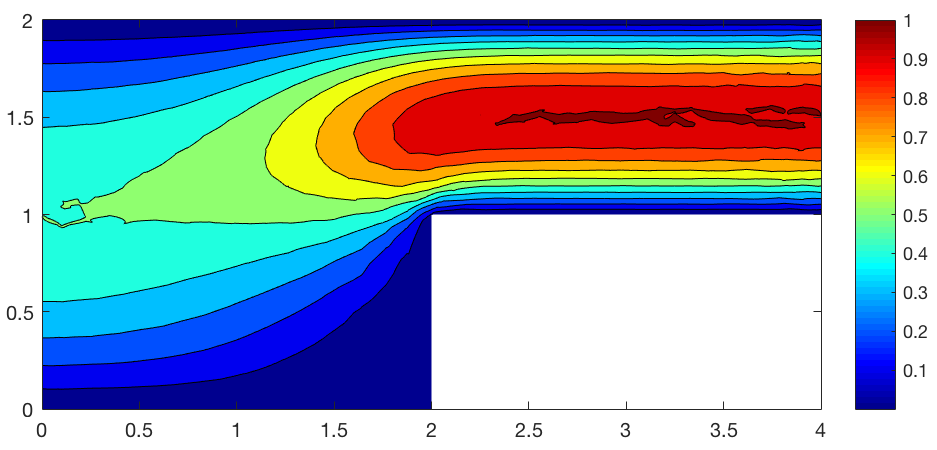}}
\end{minipage}%

\vspace*{8pt}

\caption{Example~2 ($w =1000$ and $\varepsilon^{2} = 0.01$): Speed obtained by the TH pair (left column), the SV (middle column), and the NPP pair (right column); rows~1~-- ~3 are results on meshes~1~--~3.}\label{fig:w1000 Ex2}
\end{figure}
}

To compare the mesh dependence of  these three pairs, we apply them on a structured mesh without additional barycentric refinement (Figure~\ref{fig:structured mesh ex2}). This type of meshes are generally considered of good quality and commonly used. As is shown in Figure~\ref{fig:mesh-dependence ex2}, the simulation by the SV pair turns out to be not reliable on the grid, while the NPP pair plays fine.

{
\begin{figure}[htbp]
\centering
\begin{minipage}[t]{0.5\linewidth}
\centerline{\includegraphics[width=1.95in]
{facing_step_mesh.pdf}}
\end{minipage}%
\begin{minipage}[t]{0.5\linewidth}
\centerline{\includegraphics[width=1.9in]
{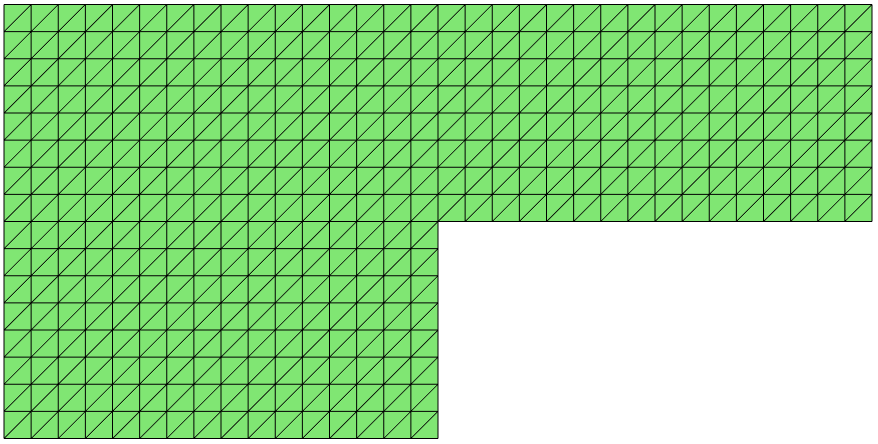}}
\end{minipage}%
\caption{Forward facing step domain and structured non-barycentric mesh.}
\label{fig:structured mesh ex2}
\end{figure}
}

\begin{figure}[H]
\begin{minipage}[t]{0.33\linewidth}
\centerline{\includegraphics[width=1.9in]{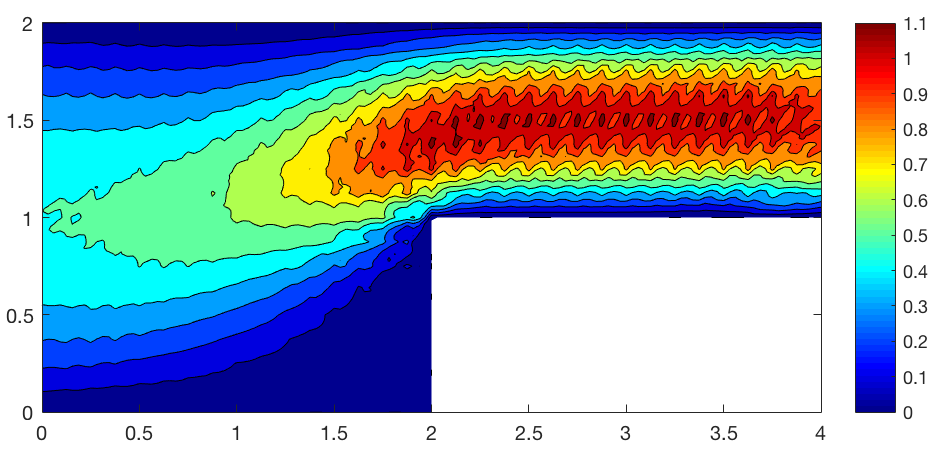}}
\end{minipage}%
\begin{minipage}[t]{0.33\linewidth}
\centerline{\includegraphics[width=1.9in]{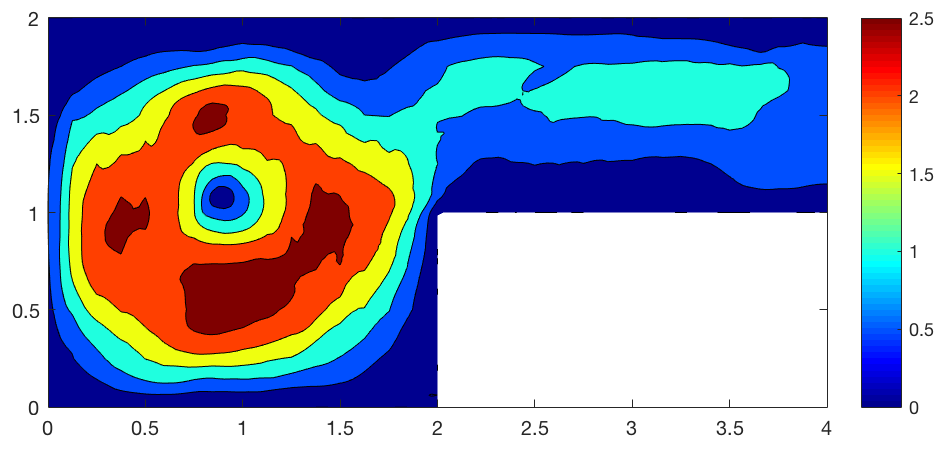}}
\end{minipage}%
\begin{minipage}[t]{0.33\linewidth}
\centerline{\includegraphics[width=1.9in]{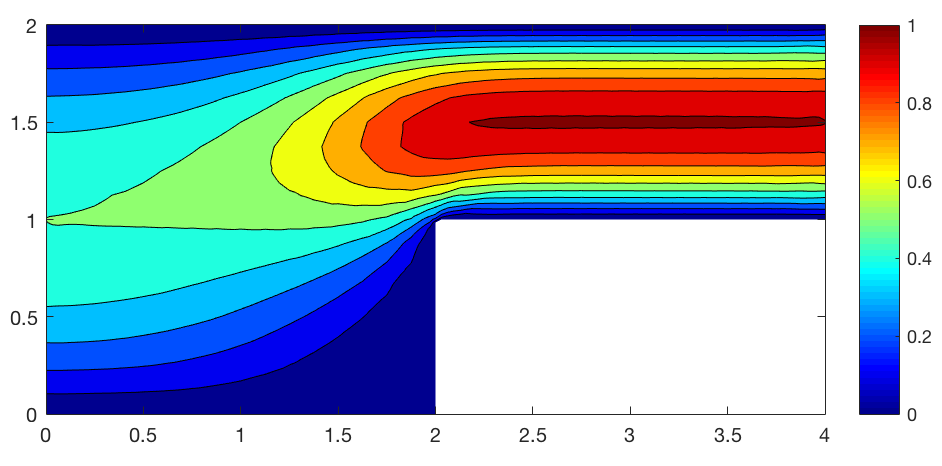}}
\end{minipage}%

\vspace*{8pt}

\caption{Example~2 ($w =100$ and $\varepsilon^{2} = 0.01$): Simulation by the SV pair (middle) is not as good as the TH pair (left) nor the NPP pair (right) on non-barycentric mesh.}
\label{fig:mesh-dependence ex2}
\end{figure}

\medskip

\paragraph{\bf Example 3} Let $\Omega = (0,1)^{2}$. We consider the Darcy--Stokes--Brinkman problem with
\begin{align*}
 \uu = \curl\,\big(sin^{2}(\pi x)sin^{2}(\pi y)\big) = \pi 
\left(
\begin{array}{c}
 sin^{2}(\pi x)sin(2\pi y) \\ -sin^{2}(\pi y)sin(2\pi x)
\end{array}
\right);
\quad
 p = \frac{2}{\pi} - sin(\pi x).
\end{align*}
The force $\uf$ is computed by $\uf = -\varepsilon^{2}\Delta \uu + \uu + \nabla p$, and $g = \dv\,\uu = 0$. The solution is smooth and independent of $\varepsilon$. We start from an unstructured initial triangular grid, and it is successively refined to maintain the quality of grids. 

In Figure~\ref{fig:smooth Brinkman P2P1 err}, we draw convergence curves of the NPP pair with different values of $\varepsilon$, where curves represents actual error declines while triangles illustrates theoretic convergence rates correspondingly. As is shown, when $0<\varepsilon<1 $, errors  in $L^{2}$-norm are of $\mathcal{O}(h^2)$ order and errors in $H^{1}$-norm are of $\mathcal{O}(h)$ order. In the limiting case of $\varepsilon=0$, the $L^{2}$-norm error reaches $\mathcal{O}(h^3)$ order and $H^{1}$-norm error reaches $\mathcal{O}(h^2)$ order, which is due to the fact that $\uV{}_{h0}$ is a conforming subspace of $\uH{}(\dv,\Omega)$. 

\begin{figure}[htbp]
\begin{minipage}[t]{0.5\linewidth}
\centerline{\includegraphics[width=2.5in]{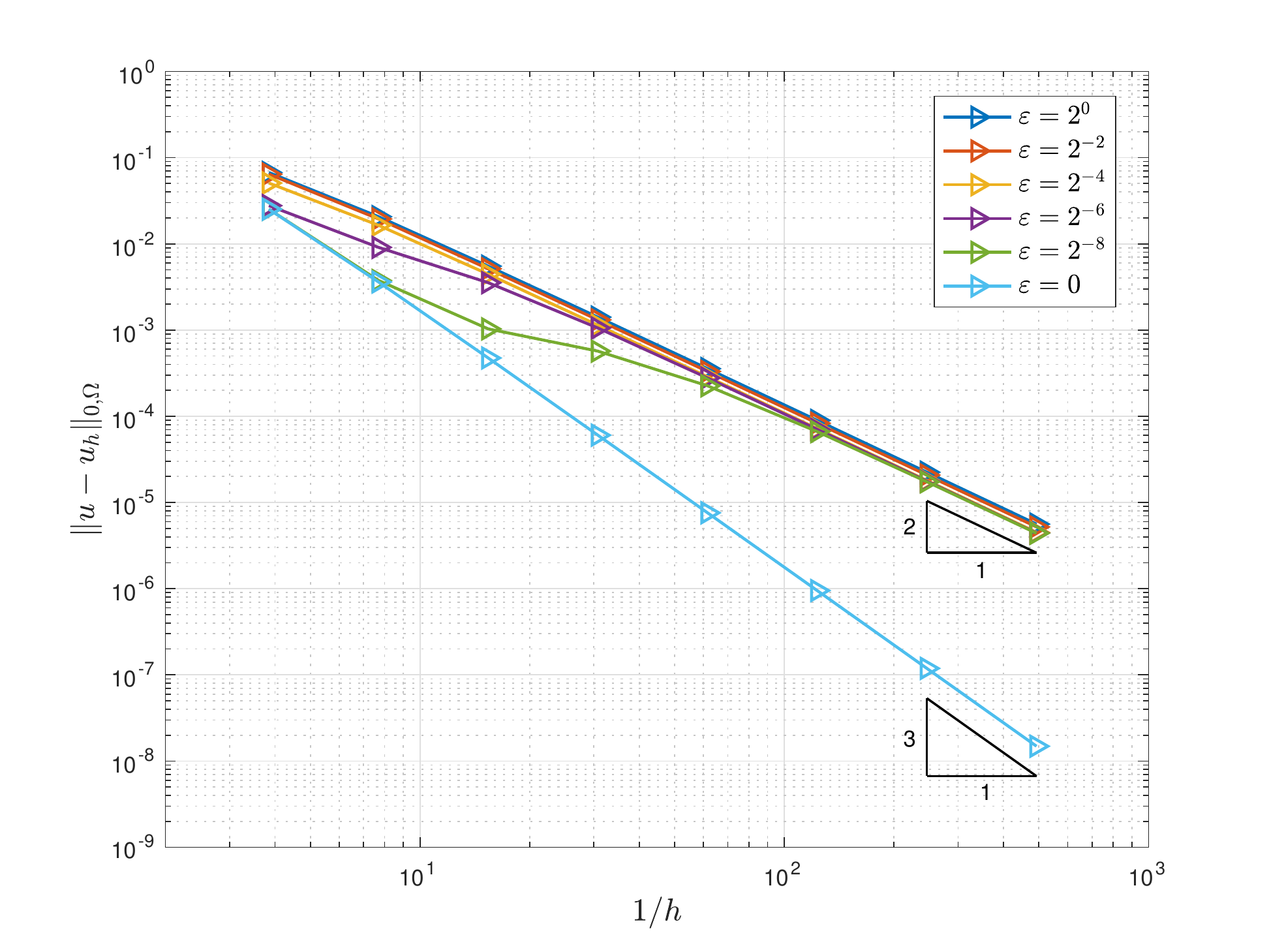}}
\end{minipage}%
\begin{minipage}[t]{0.5\linewidth}
\centerline{\includegraphics[width=2.5in]{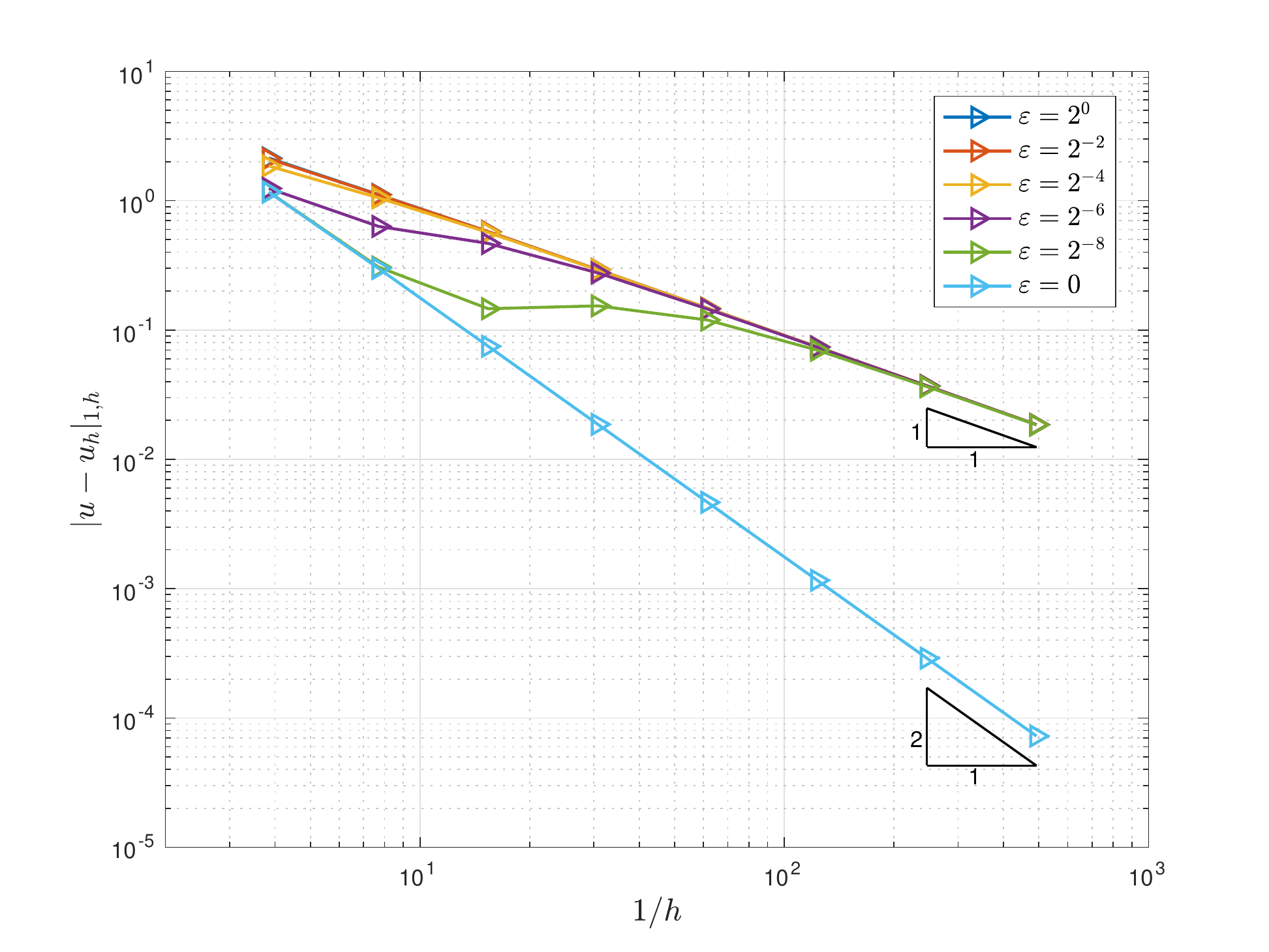}}
\end{minipage}%
\caption{Example~3: Velocity errors in the $L^{2}$-norm (left) and in the $H^{1}$-norm (right) by the NPP method.}\label{fig:smooth Brinkman P2P1 err}
\end{figure}

In~Figure~\ref{fig:smooth Brinkman P2P1vsTH}, we present the errors in the norm $|||\cdot|||_{\varepsilon,h} $ by the TH pair and the NPP pair when $\varepsilon = 2^{-8}$. Here $|||\uv|||_{\varepsilon,h} := \varepsilon^{2}|\uv|_{1,h}^{2}+\|\uv\|_{0,\Omega}+\|\dv\,\uv\|_{0,\Omega}^{2}$ is the commonly used norm which combines the Stokes and Darcy problems. Although the convergence rate of the NPP pair is one order lower than that of the TH pair, the error of the former is smaller (several magnitudes) than that of the latter in the figure where millions of DOFs have been used on the finest grid. For the NPP pair, the associated energy error of velocity is close to $10^{-3}$, while for the TH pair, it does not reach an error of $10^{-3}$ even on the eighth level mesh. However, as remarked in the figure, the degrees of freedom of the TH pair (on the eighth level mesh) is over 500 times more than the NPP pair (on the third level mesh). 
\begin{figure}[htbp]
\centerline{\includegraphics[width=2.5in]{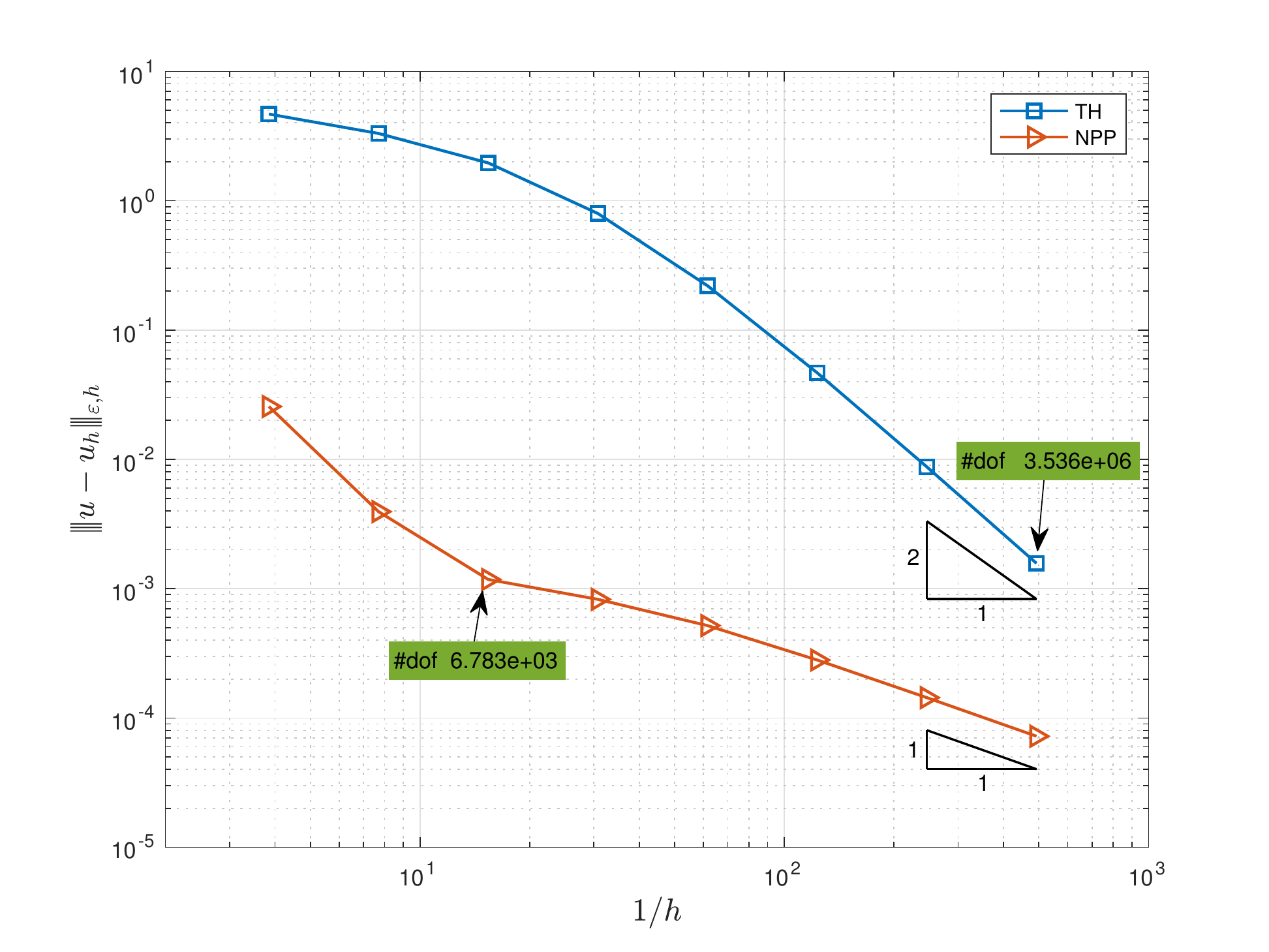}}
\caption{Example~3: Errors of velocity in the energy norm by the NPP pair and the TH pair when $\varepsilon = 2^{-8}$.}\label{fig:smooth Brinkman P2P1vsTH}
\end{figure}

\medskip

\paragraph{\bf Example 4} Let $\Omega = (0,1)^{2}$. Consider the Darcy--Stokes--Brinkman problem with 
$$
\uu = \varepsilon \curl\,\big(e^{-\frac{xy}{\varepsilon}}\big) = \left(
\begin{array}{c}
 -xe^{-\frac{xy}{\varepsilon}} \\ ye^{-\frac{xy}{\varepsilon}} 
\end{array}
\right); \quad p = -\varepsilon e^{-\frac{x}{\varepsilon}}.$$
The boundary layers of the exact velocity $\uu$ are shown in Figure~\ref{fig: exact velocity bdlayer}. 

\begin{figure}[htbp]
\begin{minipage}[t]{0.5\linewidth}\centerline{\includegraphics[width=2.2in]{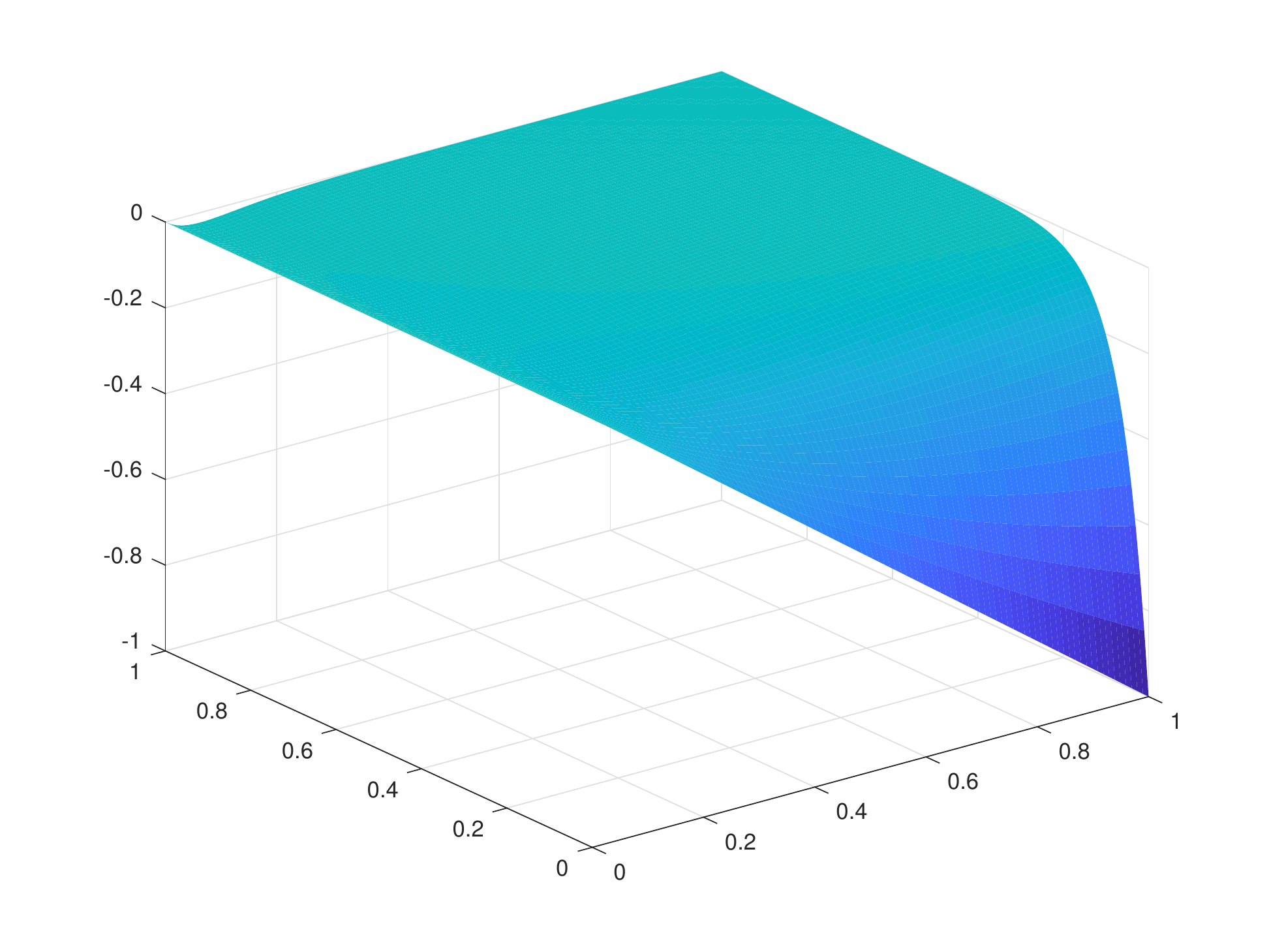}}
\end{minipage}%
\begin{minipage}[t]{0.5\linewidth}
\centerline{\includegraphics[width=2.2in]{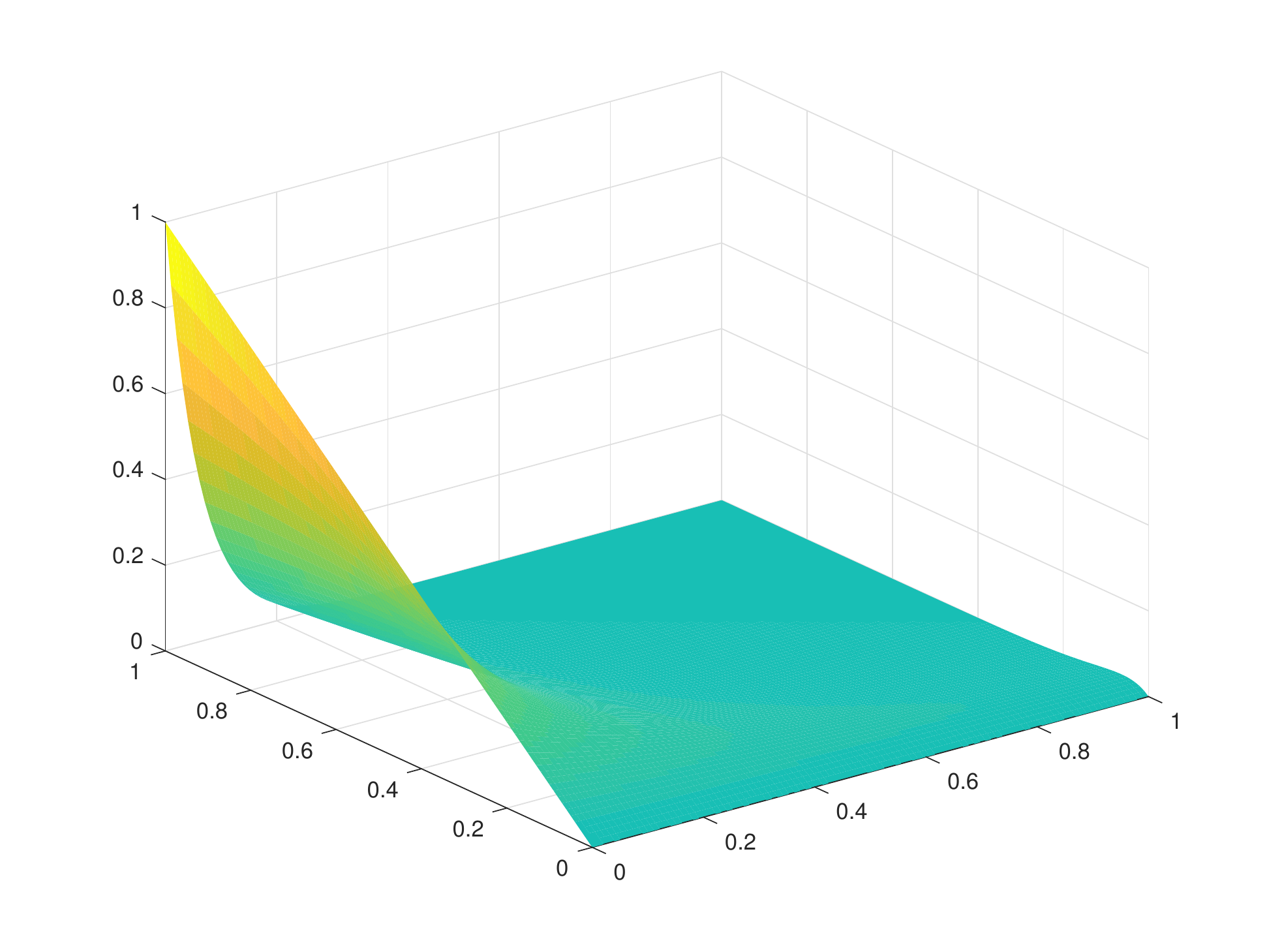}}
\end{minipage}%
\caption{Example~4: $x$-component (left) and $y$-component (right) of the exact velocity with boundary layers when $\varepsilon = 2^{-4}$.}\label{fig: exact velocity bdlayer}
\end{figure}

From Table~\ref{tab:velocity error bdlayer}, the convergence rate of velocity is approximately
one if $\varepsilon$ is sufficiently large, and it decreases to half an order as $\varepsilon$ approaches zero, which is consistent with the analysis in Theorem~\ref{thm:uniform error Brinkman}. From~Table~\ref{tab:pressure error bdlayer}, the discrete pressure exhibits $\mathcal{O}(h)$ order of convergence, higher than the theoretical estimation $\mathcal{O}(h^{1/2})$ order.

\begin{table}[h!!]
	\caption{Example~4 (with boundary layers): Errors of velocity in the energy norm by the NPP element.}
	\begin{tabular}{lllllll}
		\hline\noalign{\smallskip}	
		$\varepsilon \setminus h$ & 2.599E-1 & 1.300E-01 & 6.498E-02 &  3.249E-02  & 1.625E-02 & Rate \\
		\noalign{\smallskip}\hline\noalign{\smallskip}
		$2^{-4}$ & 2.998E-02  & 1.147E-02 & 4.958E-03 & 2.430E-03& 1.228E-03 & 1.15  \\
		$2^{-6}$ & 6.589E-02 & 3.000E-02 & 1.159E-02	 & 4.228E-03	& 1.753E-03 & 1.33  \\
		$2^{-8}$ &1.061E-01	   & 6.246E-02 & 3.238E-02	 & 1.438E-02 & 5.504E-03 & 1.07 \\
		$2^{-10}$ & 1.171E-01 & 7.906E-02 & 5.147E-02  & 3.061E-02 & 1.601E-02 & 0.71  \\
		$2^{-12}$ & 1.234E-01 & 8.455E-02 & 5.688E-02	 & 3.848E-02 & 2.529E-02	 & 0.57  \\
		\noalign{\smallskip}\hline
	\end{tabular}\label{tab:velocity error bdlayer}
\end{table}
\begin{table}[h!!]
	\caption{Example~4 (with boundary layers): Errors of pressure in the $L^{2}$-norm by the NPP element.}
	\begin{tabular}{lllllll}
		\hline\noalign{\smallskip}	
		$\varepsilon \setminus h$ & 2.599E-1 & 1.300E-01 & 6.498E-02 &  3.249E-02  & 1.625E-02 & Rate \\
		\noalign{\smallskip}\hline\noalign{\smallskip}
		$2^{-4}$ & 2.260E-03  & 8.080E-04 & 2.884E-04	 & 1.211E-04 & 5.702E-05 & 1.34   \\
		$2^{-6}$ & 2.779E-03 & 7.880E-04 & 2.696E-04	 & 9.042E-05 & 2.938E-05 & 1.63  \\
		$2^{-8}$ & 6.283E-03  & 2.044E-03 & 5.366E-04	 & 1.273E-04 & 3.448E-05 & 1.90 \\
		$2^{-10}$ & 6.730E-03 & 3.056E-03 & 1.339E-03	  & 4.607E-04 &  1.235E-04 & 1.43  \\
		$2^{-12}$ & 6.710E-03 & 3.044E-03 & 1.440E-03	 & 7.024E-04 & 3.216E-04	 & 1.09  \\
		\noalign{\smallskip}\hline
	\end{tabular}\label{tab:pressure error bdlayer}
\end{table}

\medskip

\paragraph{\bf Example 5} Let $\Omega = (0,1)^{2}$. Consider the incompressible Navier--Stokes equations
\begin{equation}\label{eq:NS eqs}
\left\{	
\begin{split}
\partial_{t}\,\uu - \varepsilon^{2} \Delta\,\uu  + (\uu\cdot \nabla)\,\uu + \nabla p &  = \uf \quad \mbox{in } \Omega, \\ 
 \dv\,\uu  & = 0  \quad  \mbox{in } \Omega, 
\end{split}
\right.
\end{equation}
with prescribed solution 
\begin{align*}
\uu(x,y,t) = 
\left(
\begin{array}{c}
\sin(1-x)\sin(y+t)\\ 
-\cos(1-x)\cos(y+t)
\end{array}
\right);
\quad 
 p = -\cos(1-x)\sin(y+t).
\end{align*}

In this example, the Crank--Nicolson scheme is used for time discretization, and the Newton linearization is adopted to handle the nonlinear term. To isolate the spatial error, let the time-step $dt = 10^{-3}$ and the final time be $10^{-2}$. Unstructured subdivisions illustrated in Example~2 are utilized.

As is depicted in Table~\ref{tab:three method for NS equations: 1e6} with $\varepsilon^{2} = 10^{-6}$, solutions by the TH pair  converge with $\mathcal{O}(h^{3/2})$ order in the $L^{2}$-norm, and by the SV pair they converge with $\mathcal{O}(h^2)$ order. It is analyzed in~\cite{Linke;Rebholz2019} that the TH pair loses order mainly because it is not pressure-robust, while the suboptimal result of the SV pair is due to additional error sources arising from the nonlinear term. The NPP pair exhibits a convergence rate of $\mathcal{O}(h^2)$ order which is consistent with its theoretical analysis, and it gives even a more accurate approximation than the SV pair in this case.

\begin{table}[htbp]
\caption{Example~5 ($\varepsilon^{2} = 10^{-6}$): Errors of velocity in the $L^{2}$-norm.}
	\begin{tabular}{lllllll}	
	\\
	 & TH & & SV & & NPP  & \\
        \noalign{\smallskip}\hline\noalign{\smallskip}
		$h$ & $\|(\uu-\uu{}_{h})(T)\|_{0,\Omega}$ & Rate &  $\|(\uu-\uu{}_{h})(T)\|_{0,\Omega}$ & Rate  & $\|(\uu-\uu{}_{h})(T)\|_{0,\Omega}$ & Rate \\
		\noalign{\smallskip}\hline\noalign{\smallskip}
		2.599E-01	& 1.746E-04 & -- & 1.031E-04 & -- & 	7.128E-05 &--	 \\
		1.300E-01	& 6.006E-05 & 1.54 &	1.362E-05	 & 2.92 & 9.407E-06	& 2.92  \\
		6.498E-02	 & 2.158E-05 & 1.48 	& 1.989E-06 & 2.78 & 1.377E-06 & 2.77 \\
		3.249E-02	 & 7.583E-06 & 1.51 & 	3.561E-07 & 2.48 & 2.488E-07 & 2.47 \\
		1.625E-02	 & 2.524E-06 & 1.59 	& 7.790E-08	& 2.19 & 5.455E-08 & 2.19 \\
		\noalign{\smallskip}\hline
	\end{tabular}
	\label{tab:three method for NS equations: 1e6}
\end{table}

\medskip

\paragraph{\bf Example 6} Let $\Omega = (0,1)^{2}$ be a square domain. Consider the Navier--Stokes equations~\eqref{eq:NS eqs}
with boundary conditions $\uu = (-1,0)^{T}$ on the side $y = 1$ and $\uu = (0,0)^{T}$ on the other three sides.
Take the viscosity as $\varepsilon^{2}= 10^{-3}$. 

The backward-Euler time-stepping scheme and the Picard iteration are adopted for this example. Set the time step to be $dt = 0.1$. Consider a long time simulation with the final time equals $90$ to derive a steady solution. Indeed, as the solution is steady, the choice of time scheme has little influence on the accuracy of the final solution. Referenced data in a benchmark work~\cite{Bruneau;Saad2006} are involved to make a reliable comparison, where the solutions are derived on a rather fine mesh, i.e., $1024\times 1024$ rectangular subdivision of domain $\Omega$. We wish to see whether major features of the steady-state flow can be captured on a coarse mesh  with $43\times 43\times 2$ cells. 

Isolines of the streamfunction, vorticity, and pressure fields are displayed in Figures~\ref{fig:streamfunction},~\ref{fig:vorticity}, and~\ref{fig:pressure}, respectively. Compared with the TH pair, the shapes of contour maps derived by the NPP pair are closer to the reference solution. Specially, the colormap of pressure obtained by the NPP pair and the TH pair are quite different; note the difference between the sidebars. By the values given in~\cite[Table 1]{Bruneau;Saad2006}, the NPP pair method gives more accurate approximation of pressure then the TH pair does.

\begin{figure}[htbp]
\begin{minipage}[c]{0.33\linewidth}
\centerline{\includegraphics[width=1.95in]{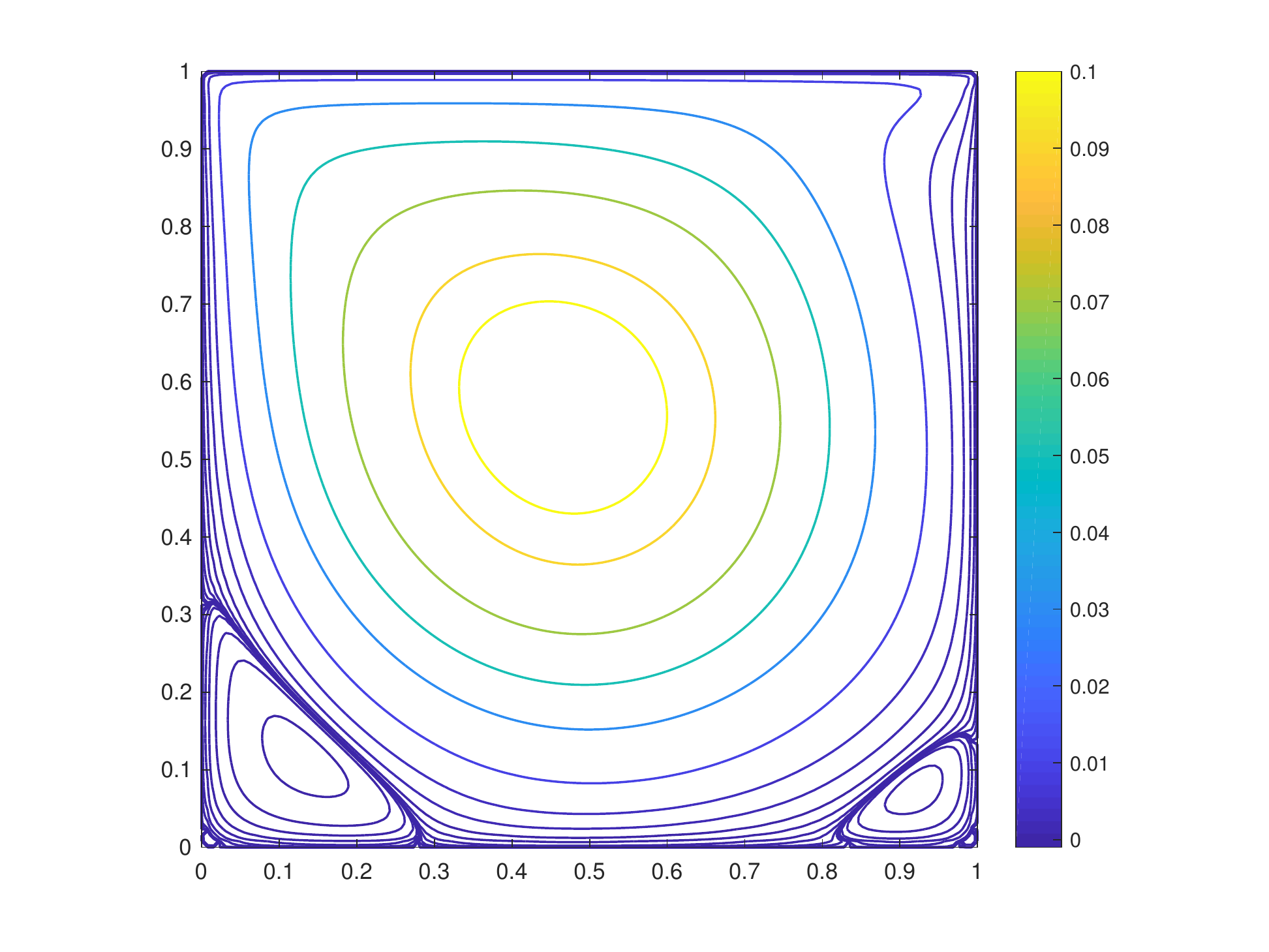}}
\end{minipage}%
\begin{minipage}[c]{0.33\linewidth}
\centerline{\includegraphics[width=1.95in]{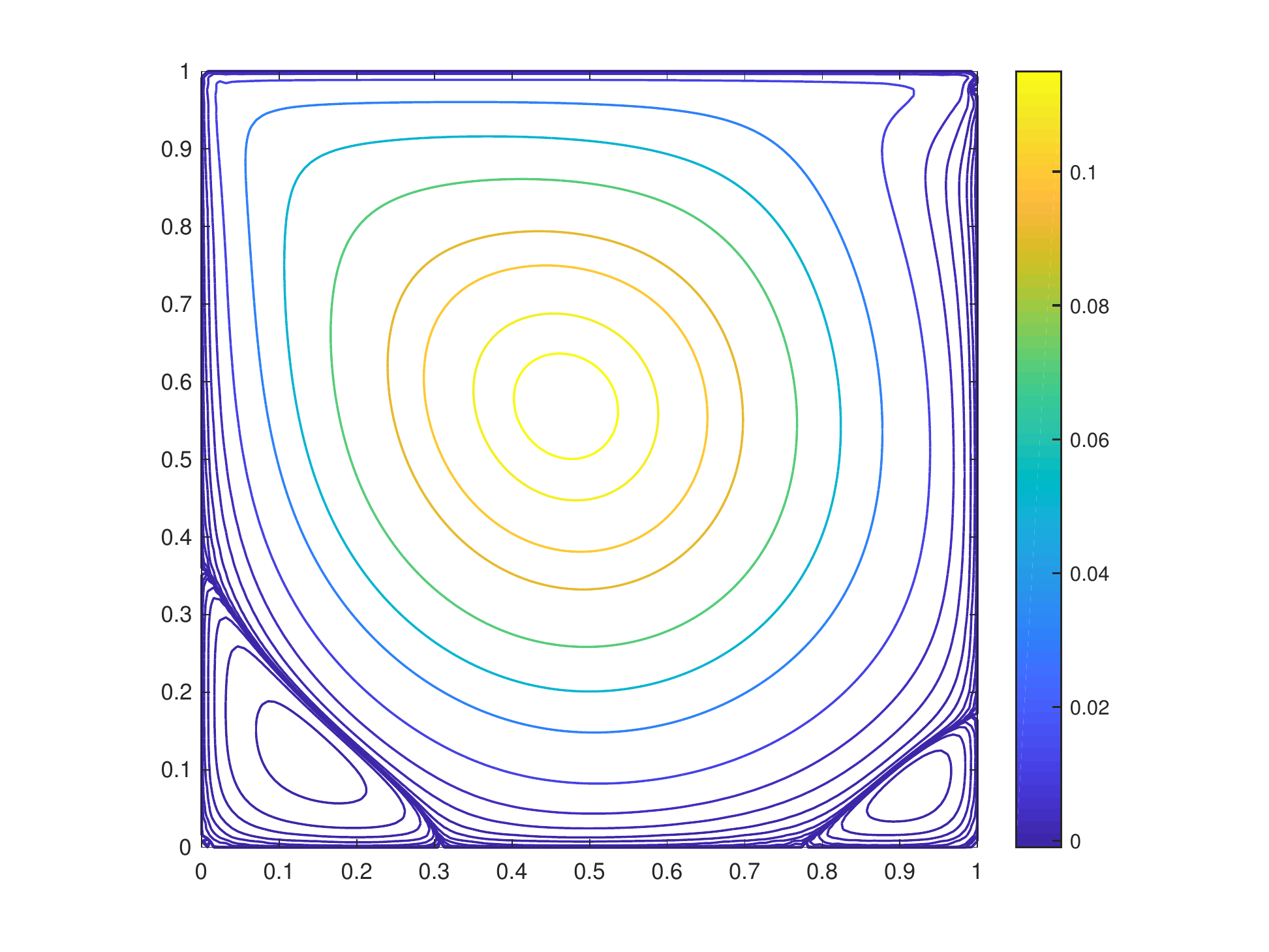}}
\end{minipage}%
\begin{minipage}[c]{0.33\linewidth}
\centerline{\includegraphics[width=1.25in]{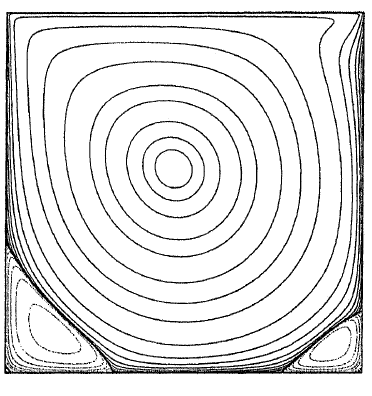}}
\end{minipage}%
\caption{Example~6 (streamfunction): Isolines given by the NPP pair (middle) is closer to the reference solution in~\cite{Bruneau;Saad2006} (right) than the TH pair (left).}\label{fig:streamfunction}
\end{figure}

\begin{figure}[htbp]
\begin{minipage}[c]{0.33\linewidth}
\centerline{\includegraphics[width=1.95in]{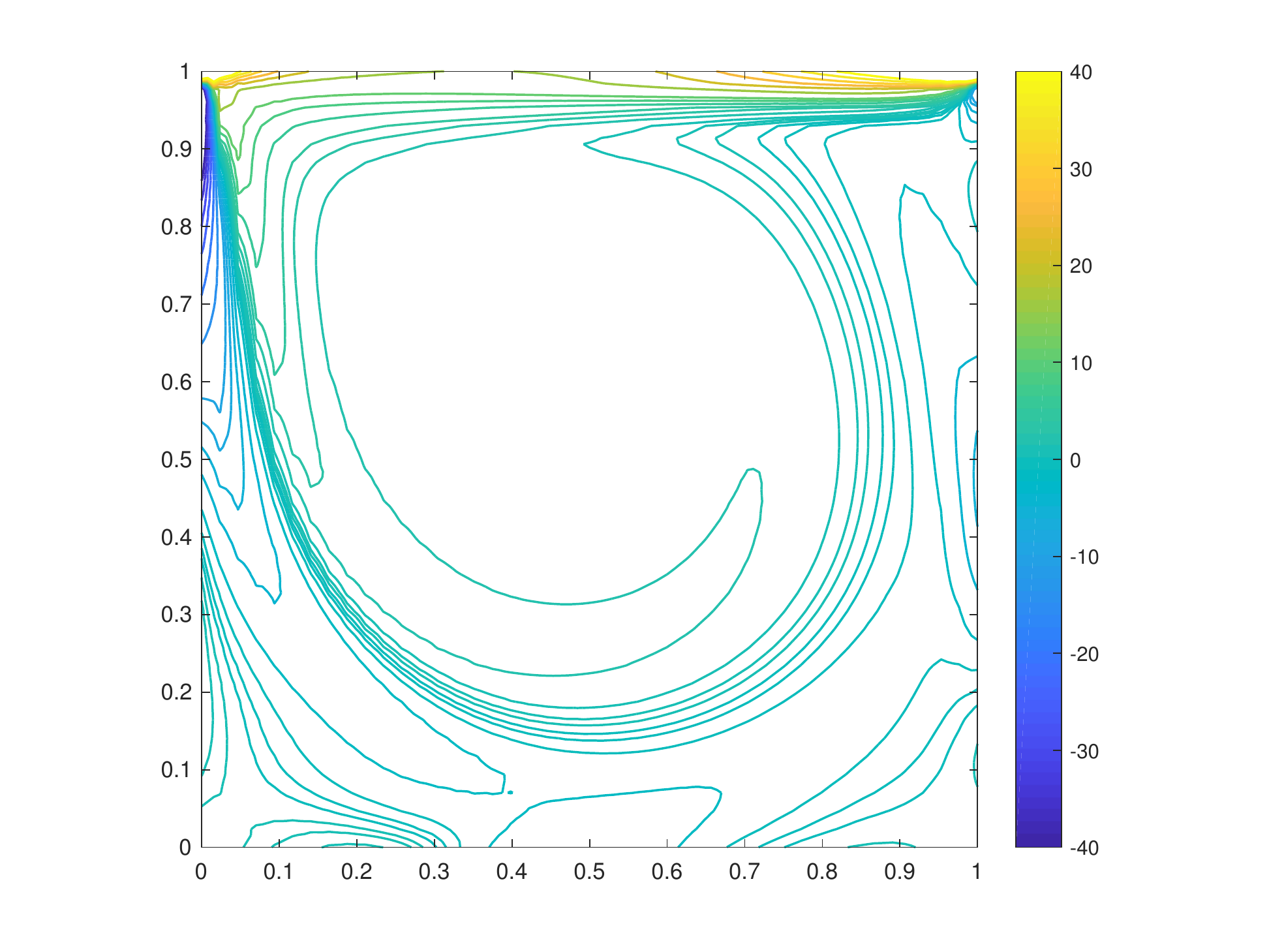}}
\end{minipage}%
\begin{minipage}[c]{0.33\linewidth}
\centerline{\includegraphics[width=1.95in]{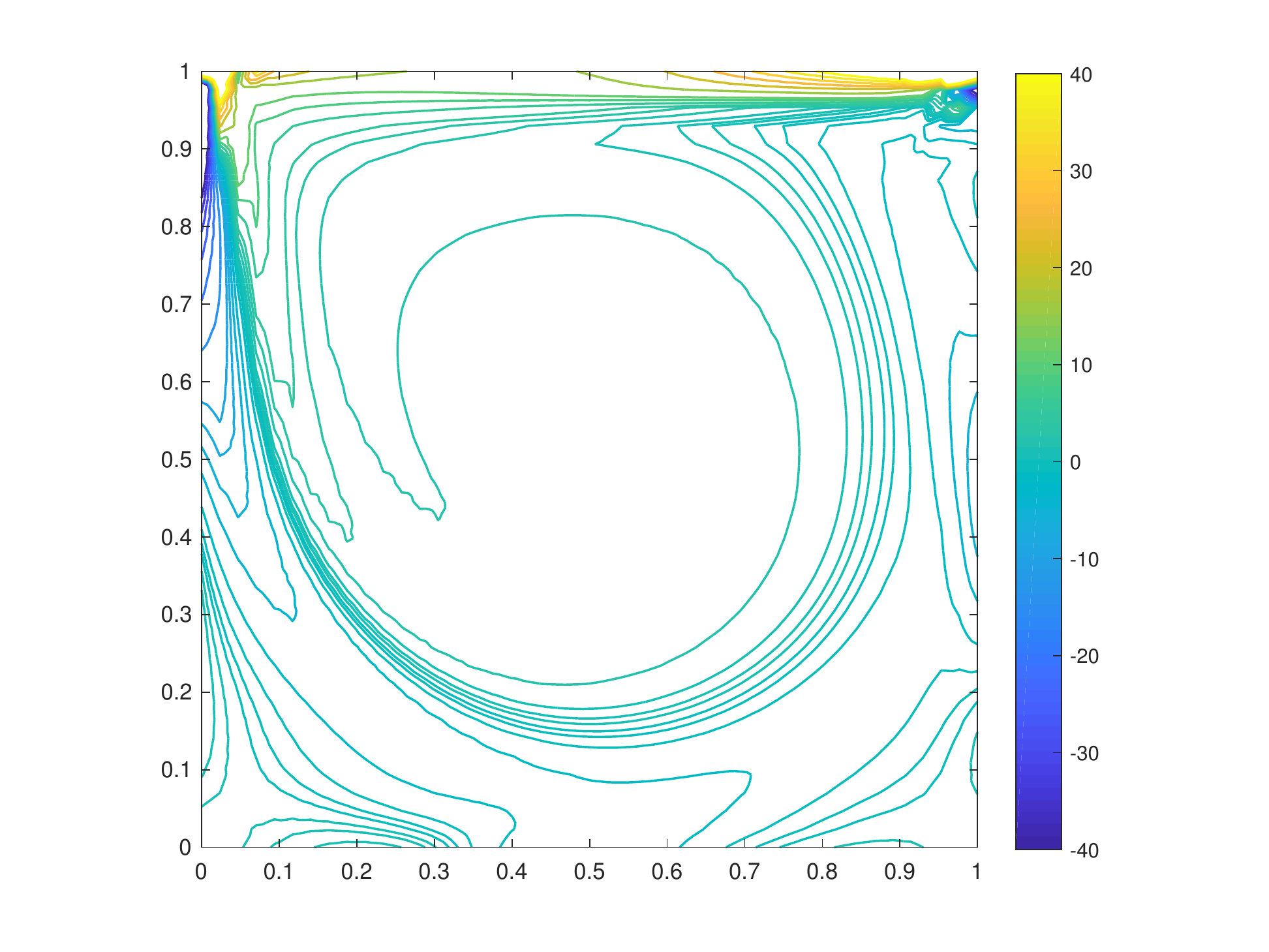}}
\end{minipage}%
\begin{minipage}[c]{0.33\linewidth}
\centerline{\includegraphics[width=1.25in]{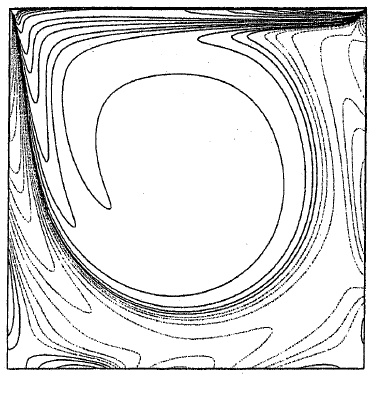}}
\end{minipage}%
\caption{Example~6 (vorticity): Isolines derived by the NPP pair (middle) is closer to the reference solution in~\cite{Bruneau;Saad2006} (right) than the TH pair (left).}\label{fig:vorticity}
\end{figure}
\begin{figure}[h!!]
\begin{minipage}[c]{0.33\linewidth}
\centerline{\includegraphics[width=1.95in]{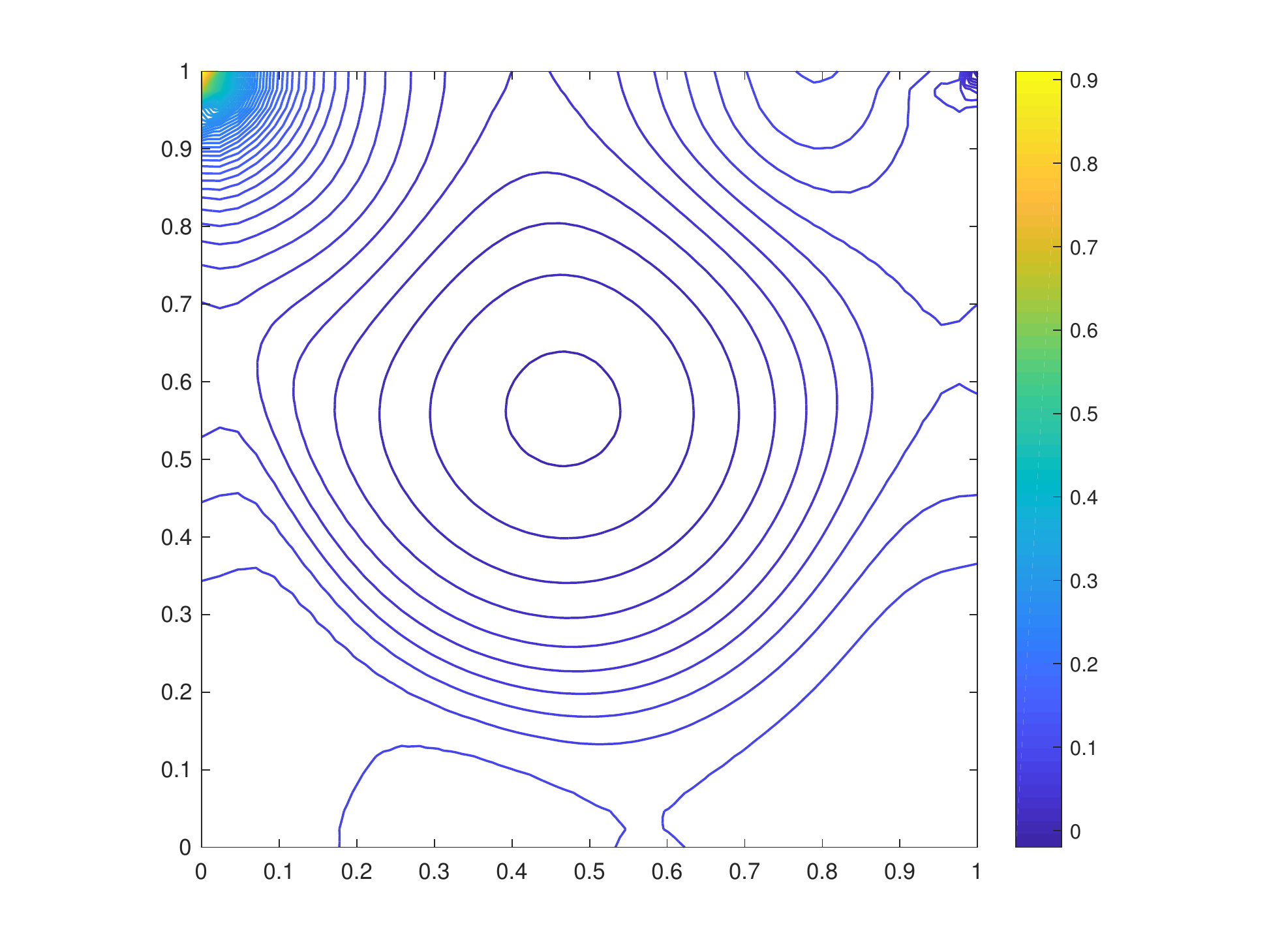}}
\end{minipage}%
\begin{minipage}[c]{0.33\linewidth}
\centerline{\includegraphics[width=1.95in]{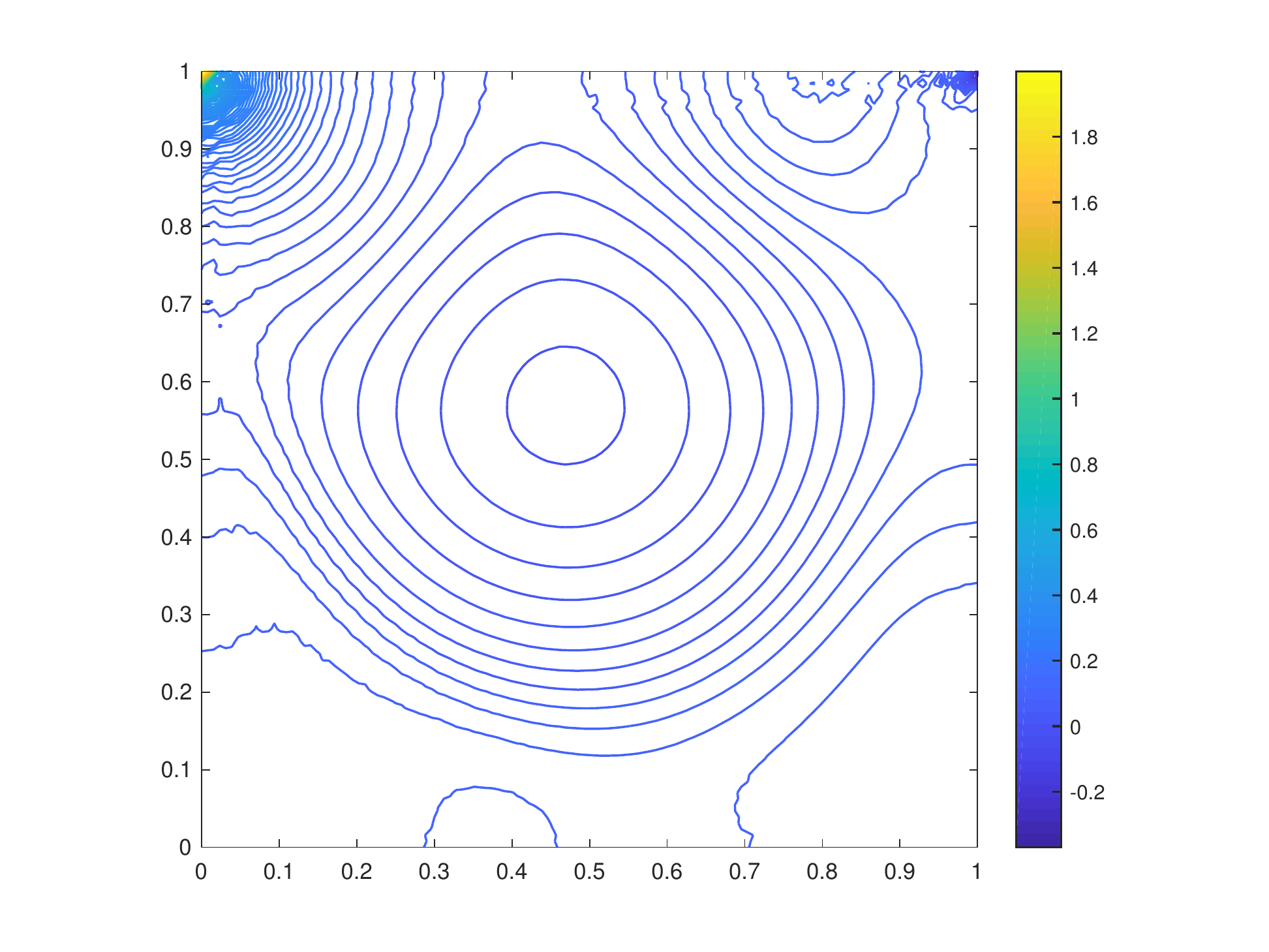}}
\end{minipage}%
\begin{minipage}[c]{0.33\linewidth}
\centerline{\includegraphics[width=1.25in]{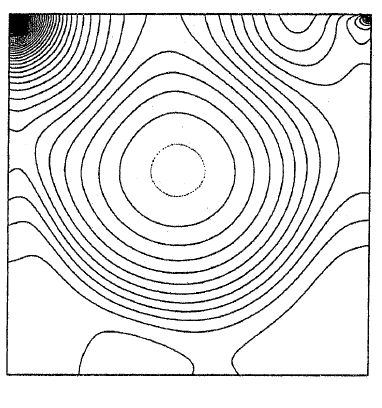}}
\end{minipage}%
\caption{Example~6 (pressure): The extreme values of the pressure by the TH pair (left) is notably different with these by the NPP pair (middle), and the latter is closer to the  reference values given in~\cite{Bruneau;Saad2006}(right).}\label{fig:pressure}
\end{figure}

The velocity along the centerlines of the cavity is also an important quantity of concern. We can see, from Figure~\ref{fig:velocity profile}, that the results computed by the NPP pair are in better agreement with the reference results in Ref.~\cite{Bruneau;Saad2006} than the TH pair.
\begin{figure}[h!!]
\begin{minipage}[t]{0.5\linewidth}
\centerline{\includegraphics[width=2.55in]{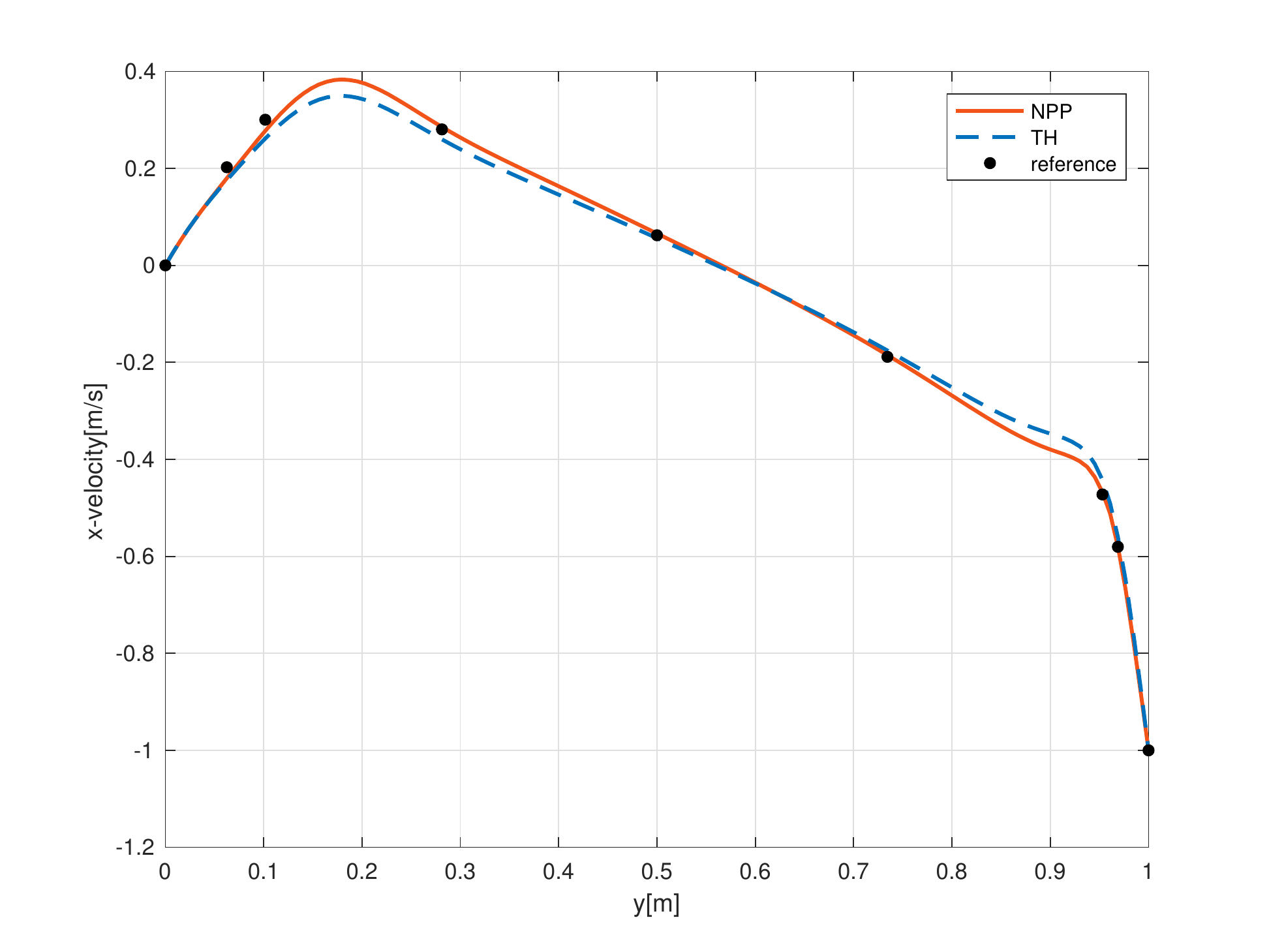}}
\end{minipage}%
\begin{minipage}[t]{0.5\linewidth}
\centerline{\includegraphics[width=2.55in]{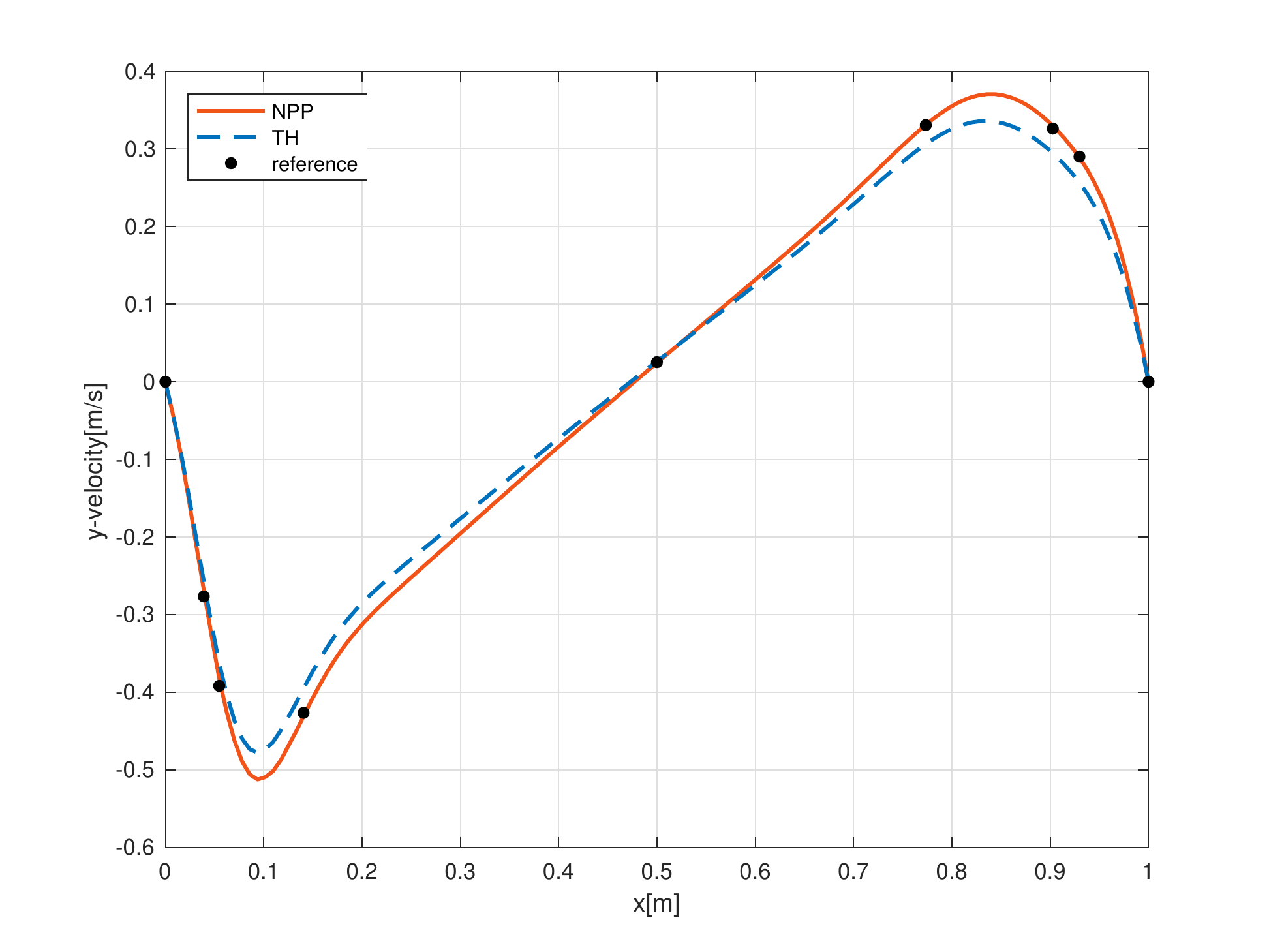}}
\end{minipage}%
\caption{Example~6 (velocity profile): $x$-velocity through the vertical centerline $x=0.5$ (left), and $y$-velocity through the horizontal centerline $y = 0.5$  (right).}\label{fig:velocity profile}
\end{figure}

Moreover, the extremes of the streamfunction and the vorticity are depicted in Tables~\ref{tab:streamfunction} and \ref{tab:vorticity}, respectively. Both of them indicate that the NPP pair gives closer results to the benchmark reference results. 

\begin{table}[h!!]
\caption{Example 6 (streamfunction): Values on the primary and the lower-left secondary vortices.}
\begin{tabular}{llllllll}
\hline\noalign{\smallskip}
Scheme & Mesh &  Primary &  $x$ &  $y$  & Secondary & $x$ &  $y$ \\
\noalign{\smallskip}\hline\noalign{\smallskip}
TH & $43\times 43$&  1.0862E-01 & 0.4688 & 0.5703 & -1.3882E-03 & 0.1328 & 0.1094  \\
NPP & $43\times 43$& 1.1733E-01 & 0.4688 & 0.5703 & -1.6221E-03 & 0.1406 & 0.1094  \\
Ref. &  $1024\times 1024$ &1.1892E-01 & 0.4688 & 0.5654 & -1.7292E-03 & 0.1367 & 0.1123 \\
\noalign{\smallskip}\hline
 \end{tabular}\label{tab:streamfunction}
\end{table}

\begin{table}[h!!]
\caption{Example 6 (vorticity): Values on the primary and the lower-left secondary vortices.}
\begin{tabular}{llllllll}
\hline\noalign{\smallskip}
Scheme & Mesh &  Primary &  $x$ &  $y$ & Secondary & $x$ &  $y$ \\
\noalign{\smallskip}\hline\noalign{\smallskip}
TH & $43\times 43$&1.8976E+00 & 0.4688 & 0.5703 & -9.1294E-01 & 0.1328 & 0.1094  \\
NPP & $43\times 43$& 2.0615E+00 & 0.4688 & 0.5703 & -9.8718E-01 & 0.1406 & 0.1094  \\
Ref. &  $1024\times 1024$ & 2.0674E+00 & 0.4688 & 0.5654 & -1.1120E+00 & 0.1367 & 0.1123 \\
\noalign{\smallskip}\hline
 \end{tabular}\label{tab:vorticity}
\end{table}

\appendix

\section{Dimension of the local space ${\rm ker}(\dv,\protect\uV{}_{h0}(M))$}
\label{sec:app}
This appendix is devoted to analyze the basis functions of ${\rm ker}(\dv,\uV{}_{h0}(M))$ defined on a macroelement $M$. Lemma~\ref{lem:localkernel} is proved at the end of this section.

\subsection{Local structure of divergence-free functions}

%{Local divergence-free functions on an element $T$}

\noindent Let $T$ be a triangle with nodes $\{a_{i},a_{j},a_{k}\}$ and edges $\{e_{i},e_{j},e_{k}\}$. Denote $\mathbf{n}_{T,e_{l}}$ as a unit outward vector normal to $e_{l}$ and $\mathbf{t}_{T,e_{l}}$ as a unit tangential vector of $e_{l}$ such that $\mathbf{n}_{T,e_{l}} \times \mathbf{t}_{T,e_{l}} >0$, where $l \in \{i,\ j,\ k\}$. Denote the lengths of edges by $\{l_{i},l_{j},l_{k}\}$, and the area of $T$ by $S$. 
\begin{figure}[htbp]
\begin{minipage}[t]{0.5\linewidth}
\centerline{\includegraphics[width=1.9in]{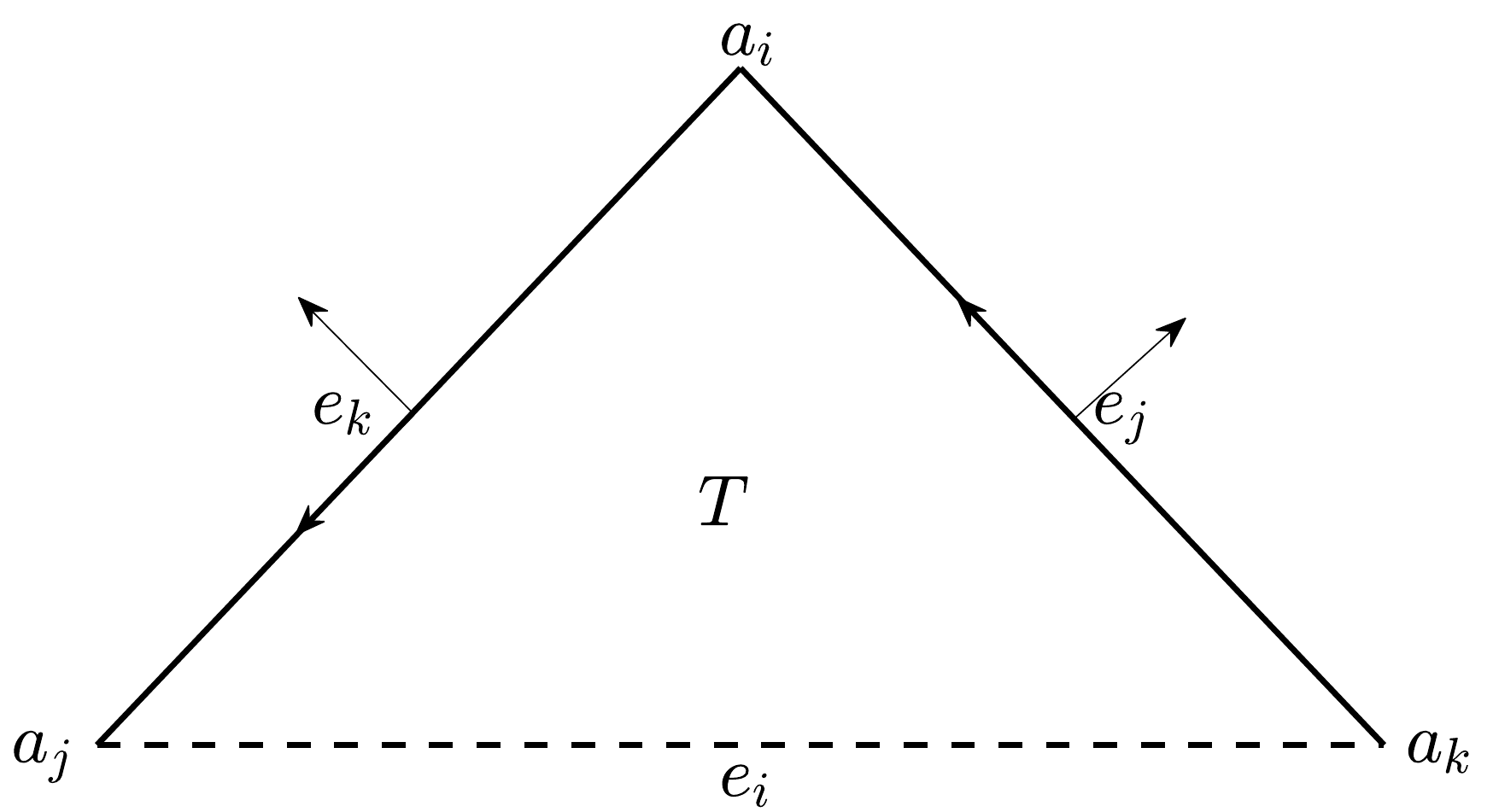}}
\end{minipage}%
\begin{minipage}[t]{0.5\linewidth}
\centerline{\includegraphics[width=1.9in]{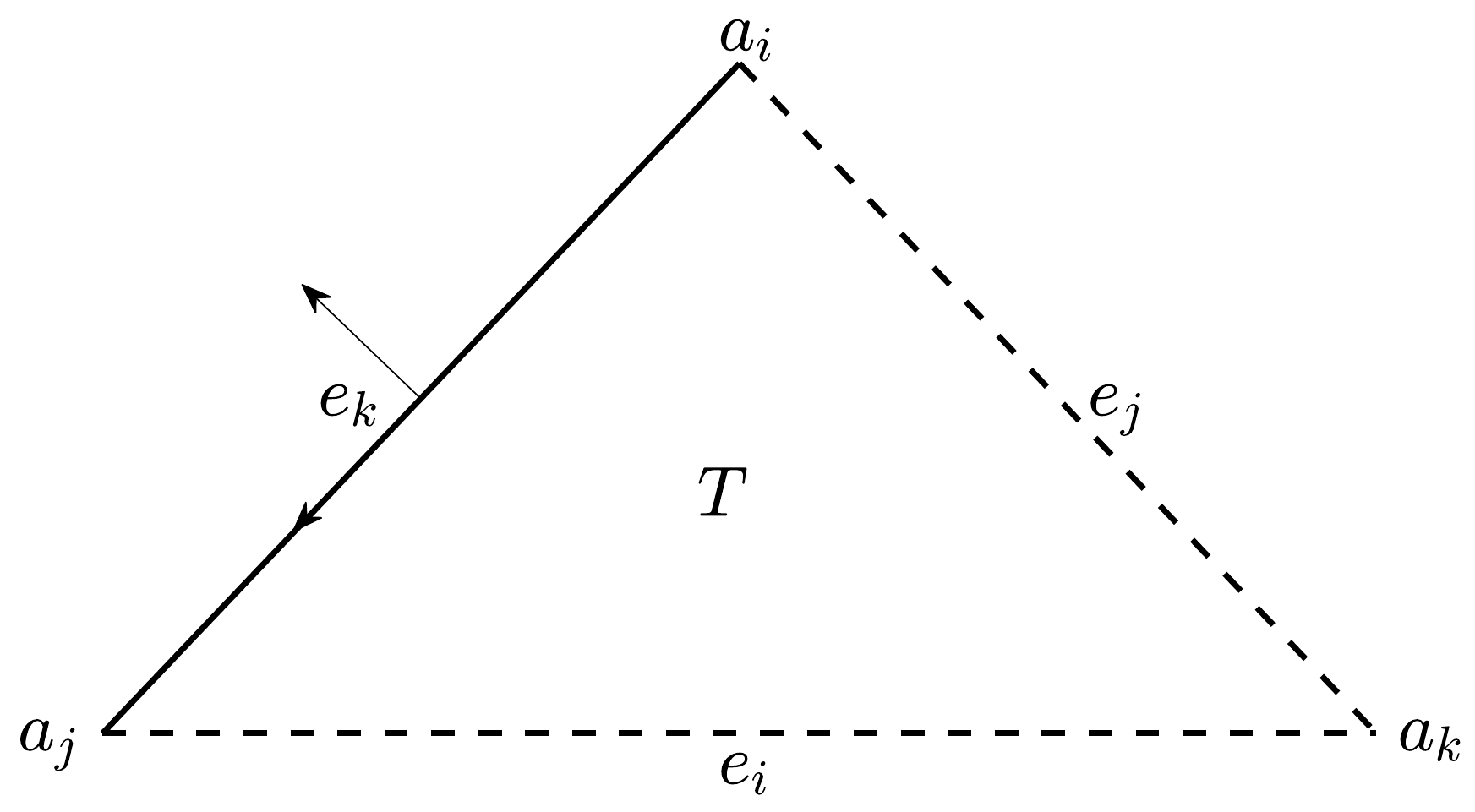}}
\end{minipage}%
\caption{Degrees of freedom vanish on dotted edges.}\label{fig:T vanish on two edge}
\end{figure}

Denote 
\begin{equation}
\uW{}_{T,e_{j}e_{k}}:= \Big\{\uv\in (P_2(T))^2:\dv\,\uv=0,\ \int_{e_i}\uv\cdot\mathbf{t}_{e_i}=0,\int_{e_i}\uv\cdot\mathbf{n}_{e_i}q=0,\forall\,q\in P_2(e_i)\Big\},
\end{equation} namely, $\uW{}_{T,e_{j}e_{k}}$ consists of quadratic polynomials that are divergence-free and all nodal parameters associated with $e_{i}$ equal to zero. Denote 
\begin{equation}
\uW{}_{T,e_{j}e_{k}}^\mathbf{n}:= \Big\{\uv\in \uW{}_{T,e_{j}e_{k}}:\int_{e_{j}}\uv\cdot\mathbf{n}_{e_{j}}=\int_{e_{k}}\uv\cdot\mathbf{n}_{e_k}=0\Big\}
\end{equation}
and
\begin{equation}
\uW{}_{T,e_{k}}:=\uW{}_{T,e_{j}e_{k}}\cap \uW{}_{T,e_{k}e_{i}}.
\end{equation} 

Direct calculation leads to the following results.
\begin{lemma} 
$\dim(\uW{}_{T,e_{k}})=1$, $\dim(\uW{}_{T,e_{j}e_{k}}^\mathbf{n})=4$, and $\dim(\uW{}_{T,e_{j}e_{k}})=5$.
\end{lemma}
A set of basis functions can be constructed explicitly for the local spaces. Recall that $\uphi{}_{\mathbf{n}_{T,e_{l}},0}$, $\uphi{}_{\mathbf{n}_{T,e_{l}},1}$, $\uphi{}_{\mathbf{n}_{T,e_{l}},2}$, and $\uphi{}_{\mathbf{t}_{T,e_{l}},0}$ represent the basis functions on $e_{l}$; $l\in \{i,\ j,\ k\}$. 
Let 
\begin{equation*}
\begin{split}
& \uw{}_{T,e_{j}} = \frac{2S}{3l_{j}^{2}} \uphi{}_{\mathbf{n}_{T,e_{j}},1} - \uphi{}_{\mathbf{t}_{T,e_{j}},0}, \quad \ \
\uw{}_{T,e_{j},e_{k}} = -\frac{1}{3l_{j}}\uphi{}_{\mathbf{n}_{T,e_{j}},1} + \frac{1}{10l_{j}}\uphi{}_{\mathbf{n}_{T,e_{j}},2} - \frac{2}{3l_{k}}\uphi{}_{\mathbf{n}_{T,e_{k}},1}, \\
& \uw{}_{T,e_{k}} = -\frac{2S}{3l_{k}^{2}} \uphi{}_{\mathbf{n}_{T,e_{k}},1} + \uphi{}_{\mathbf{t}_{T,e_{k}},0}, \quad 
\uw{}_{T,e_{k},e_{j}} = -\frac{2}{3l_{j}}\uphi{}_{\mathbf{n}_{T,e_{j}},1} - \frac{1}{3l_{k}}\uphi{}_{\mathbf{n}_{T,e_{k}},1} - \frac{1}{10l_{k}}\uphi{}_{\mathbf{n}_{T,e_{k}},2}, \\
& \uw{}_{T,a_{i}} = -\frac{1}{l_{j}}\uphi{}_{\mathbf{n}_{T,e_{j}},0} + 
\frac{ -l_{i}^{2}l_{j}^{2} + 2l_{i}^{2}l_{k}^{2} + l_{j}^{4} + 3l_{j}^{2}l_{k}^{2} - 2l_{k}^{4} }{ 12l_{j}^{3} l_{k}^{2} }
\uphi{}_{\mathbf{n}_{T,e_{j}},1} + 
\frac{l_{i}^{2}-l_{j}^{2}+3l_{k}^{2}}{ 40l_{j} l_{k}^{2} }
\uphi{}_{\mathbf{n}_{T,e_{j}},2} + \frac{1}{l_{k}}\uphi{}_{\mathbf{n}_{T,e_{k}},0}.
\end{split}
\end{equation*}
Then 
\begin{align*}
& \uW{}_{T,e_{j}e_{k}}^{\mathbf{n}} ={\rm span} \big\{\uw{}_{T,e_{j}},\ \uw{}_{T,e_{j},e_{k}}, \ \uw{}_{T,e_{k},e_{j}},\ \uw{}_{T,e_{k}}\big\}, \quad 
\uW{}_{T,e_{k}} = {\rm span} \big\{\uw{}_{T,e_{k}}\big\},\\
& \mbox{and}\ \uW{}_{T,e_{j}e_{k}} ={\rm span} \big\{\uw{}_{T,e_{j}},\ \uw{}_{T,e_{j},e_{k}}, \ \uw{}_{T,e_{k},e_{j}},\ \uw{}_{T,e_{k}},\ \uw{}_{T,a_{i}}\big\}.
\end{align*}

\subsection{Divergence-free functions on sequentially connected cells}
Let $T_1$ and $T_2$ be two adjacent cells such that $\overline{T_1}\cap \overline{T}_2=e_2$. Denote 
%\begin{align*}
\begin{multline*}
\qquad\quad\uV{}_{h}(T_1\cup T_2) := \Big\{\uv{}_{h} \in \uL{}^{2}(T_1\cup T_2): \  \uv{}_{h}{}|_{T_{l}} \in \big(P_{2}(T_{l})\big)^{2}, \  l = 1,\ 2, \ 
\\ 
\uv{}_{h}\cdot \mathbf{n}_{e}  \mbox{ and }  \fint_{e} \uv{}_{h}\cdot \mathbf{t}_{e} \ud s\ \mbox{ are continuous across }\ e_{2} \Big\},\qquad
\end{multline*}
%\end{align*}
\begin{align*}
& \uV{}_{h0}(T_1\cup T_2):= \Big\{\uv{}_{h} \in \uV{}_h(T_1\cup T_2): \ \uv{}_{h}\cdot \mathbf{n}_{e}  \mbox{ and }  \fint_{e} \uv{}_{h}\cdot \mathbf{t}_{e}  \mbox{ vanish on } \partial(T_{1}\cup T_{2})\Big\},\\
& \uV{}_{h,e_{3}}(T_1\cup T_2):= \Big\{\uv{}_{h} \in \uV{}_h(T_1\cup T_2): \ \uv{}_{h}\cdot \mathbf{n}_{e}  \mbox{ and }  \fint_{e} \uv{}_{h}\cdot \mathbf{t}_{e}  \mbox{ vanish on } \partial(T_{1}\cup T_{2})\big\backslash e_{3}\Big\}.
\end{align*}
\begin{lemma}\label{lem:two cell 2 dim}
It holds that $\dim(\uV{}_{h0}(T_1\cup T_2))=0$ and $\dim(\uV{}_{h,e_3}(T_1\cup T_2))=2$.
\begin{figure}[htbp]
\begin{minipage}[t]{0.5\linewidth}
\centerline{\includegraphics[width=1.95in]{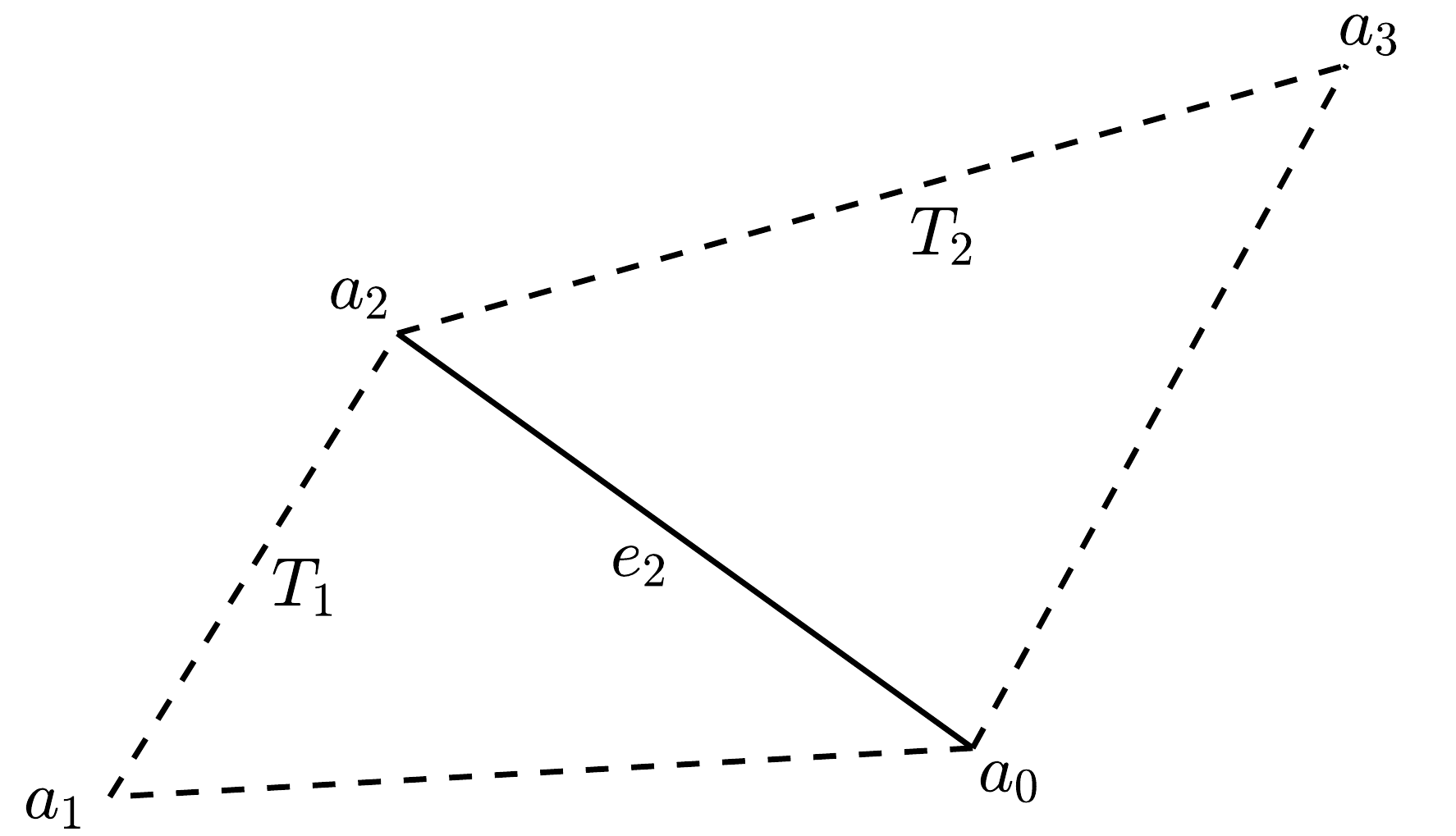}}
\end{minipage}%
\begin{minipage}[t]{0.5\linewidth}
\centerline{\includegraphics[width=1.95in]{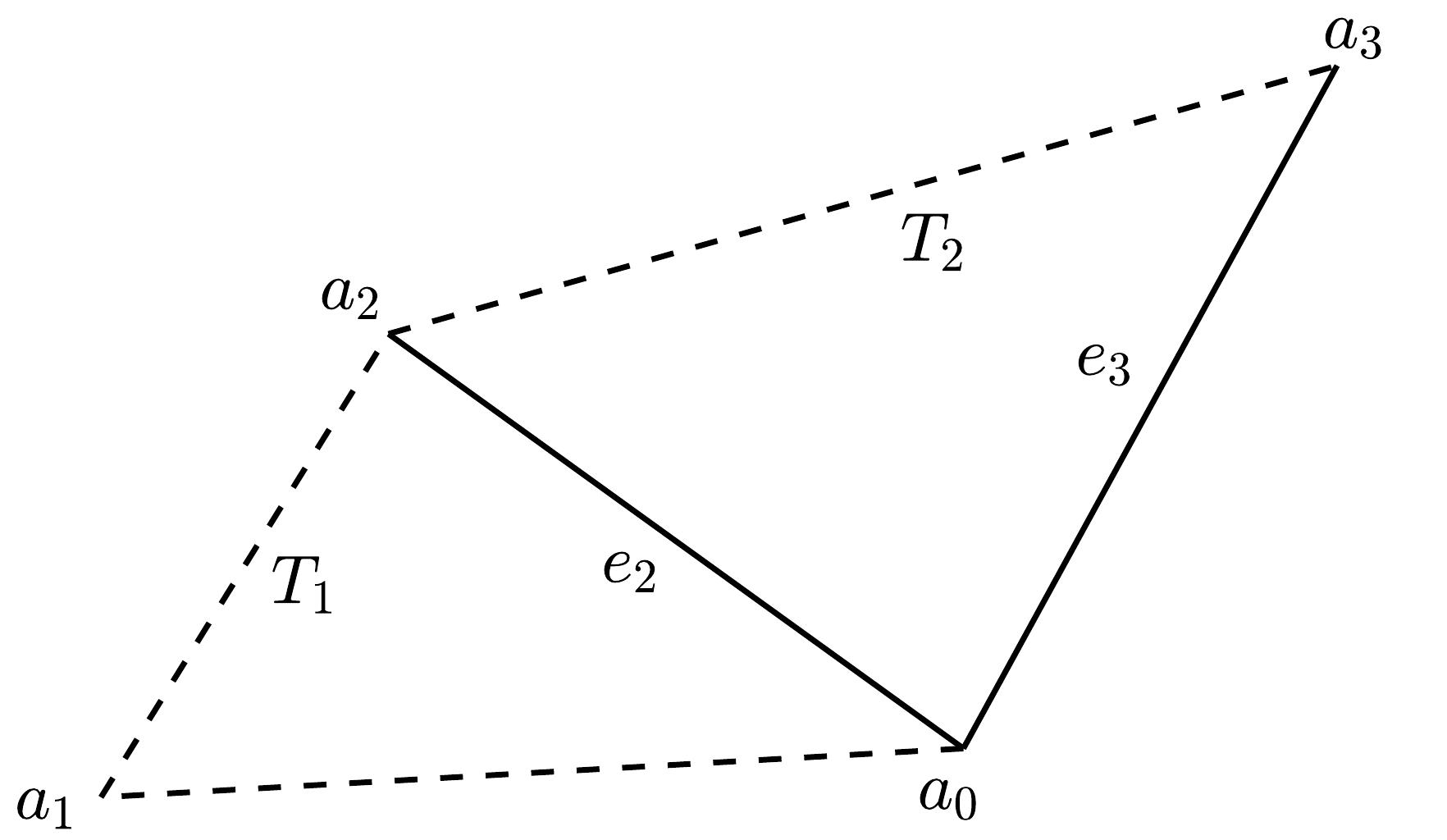}}
\end{minipage}%
\caption{Two adjacent cells; Degrees of freedom vanish on dotted edges.}\label{fig:two_triangles}
\end{figure}
\end{lemma}
\begin{proof}
Given $\upsi{}_h\in \uV{}_{h0}(T_1\cup T_2)$, $\upsi{}_h|_{T_1}=\alpha \uw{}_{T_1,e_2}$ and $\upsi{}_h|_{T_2}=\beta \uw{}_{T_2,e_2}$. It is easy to verify that the weak continuity conditions imposed on $e_2$ makes $\alpha=\beta=0$, namely $\upsi{}_h=\undertilde{0}$. 

Similarly, we can prove $\dim(\uV{}_{h,e_3}(T_1\cup T_2))=2$. Specifically, two basis functions are
$$
\upsi{}_{h}^{1}\ \mbox{ such\ that }\ \upsi{}_{h}^{1}|_{T_{1}}=\uw{}_{T_{1},e_{2}}\ \mbox{ and }\ \upsi{}_{h}^{1}|_{T_2}=\uw{}_{T_{2},e_{2}}+\frac{S_{1}+S_{2}}{d_{2}}
\uw{}_{T_{2},e_{3},e_{2}}
$$
and 
$$
\upsi{}_h^2\ \mbox{ such that }\ \upsi{}_h^2|_{T_1}=\undertilde{0}\ \mbox{and} \ \upsi{}_h^2|_{T_2}=\uw{}_{T_2,e_3}.
$$
The proof is completed. 
\end{proof}

To admit a nontrivial basis function, at least four cells are needed. 
\begin{lemma}\label{lem:three cells dim 0}
Let $\omega$ be a subdomain of three continuous cells, where the first and last cells are not connected; see Figure~\ref{fig:three_triangles}.  If $\uv{}_{h}\in \uV{}_{h}\big(\omega)$, $\dv(\uv{}_{h}) = 0$, and degrees of freedom of $\uv{}_{h}$ vanish on $\partial \omega$, then $\uv{}_{h}=\undertilde{0}$.

\begin{figure}[H]
\centerline{\includegraphics[width=2in]{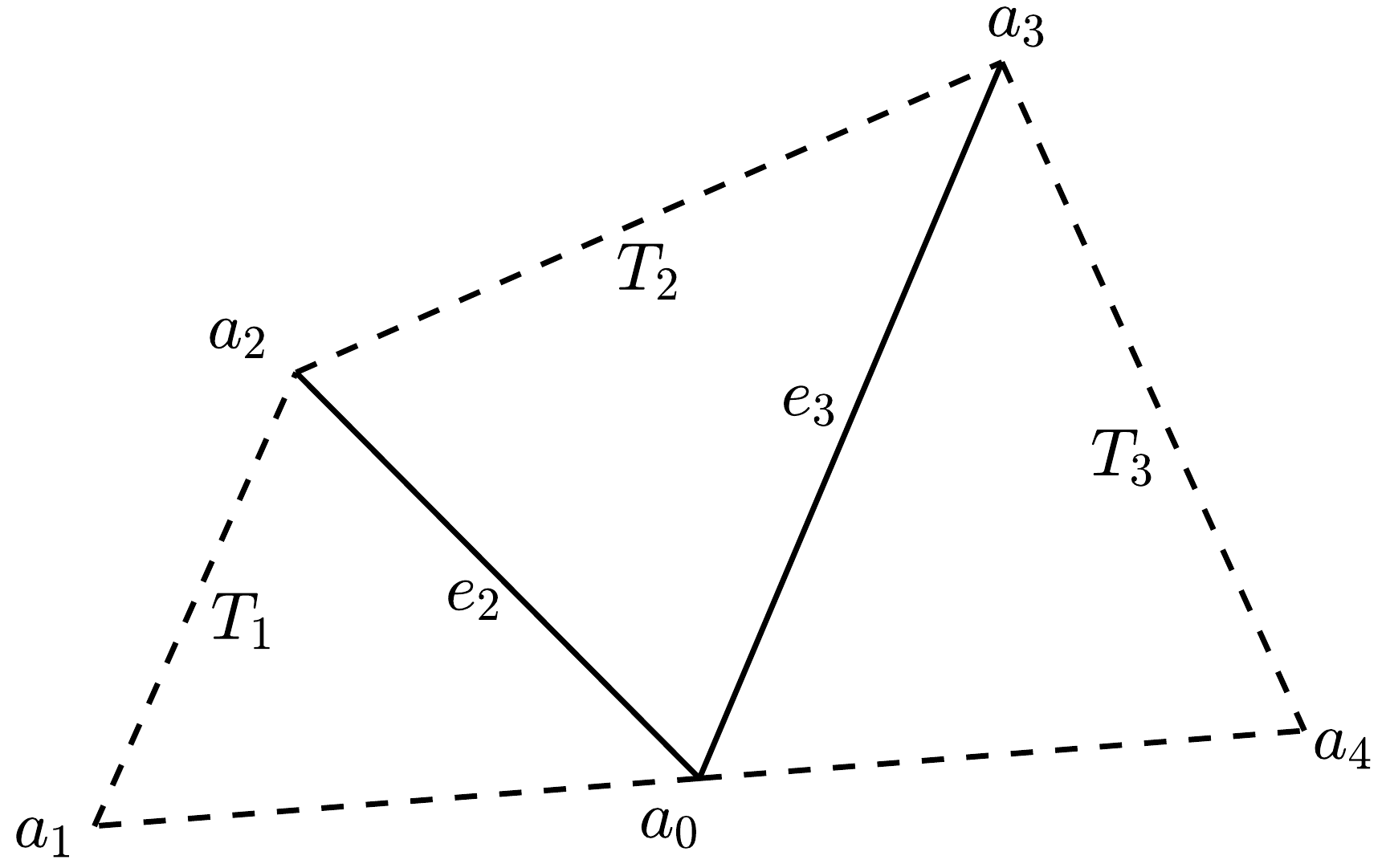}}
\caption{No trivial divergence-free function on three continuous cells.}\label{fig:three_triangles}
\end{figure}
\end{lemma}
\begin{lemma}\label{lem:Four-Cell func}
Let $\omega$ be a subdomain composed of four continuous cells, and its first and last cells are not connected; see Figure~\ref{fig:four cell in M}.
A local space on $\omega$ is define  as
\begin{align*}
\uV{}_{h0,\omega} := \Big\{\uv{}_{h} \in \uL{}^{2}(\omega):\  \uv{}_{h}{}|_{T} \in & (P_{2}(T))^{2}, \  \forall T \in \omega, \ \uv{}_{h}\cdot \mathbf{n}_{e}  \mbox{ and } \fint_{e} \uv{}_{h}\cdot \mathbf{t}_{e} \ud s \mbox{ are continuous} \\
& \mbox{ across interior edges and vanish on edges lying on } \partial \omega \Big\},
\end{align*}
then ${\rm dim}({\rm ker}(\dv, \uV{}_{h0,\omega})) = 1$.

\begin{figure}[H]
\centerline{\includegraphics[width=2.1in]{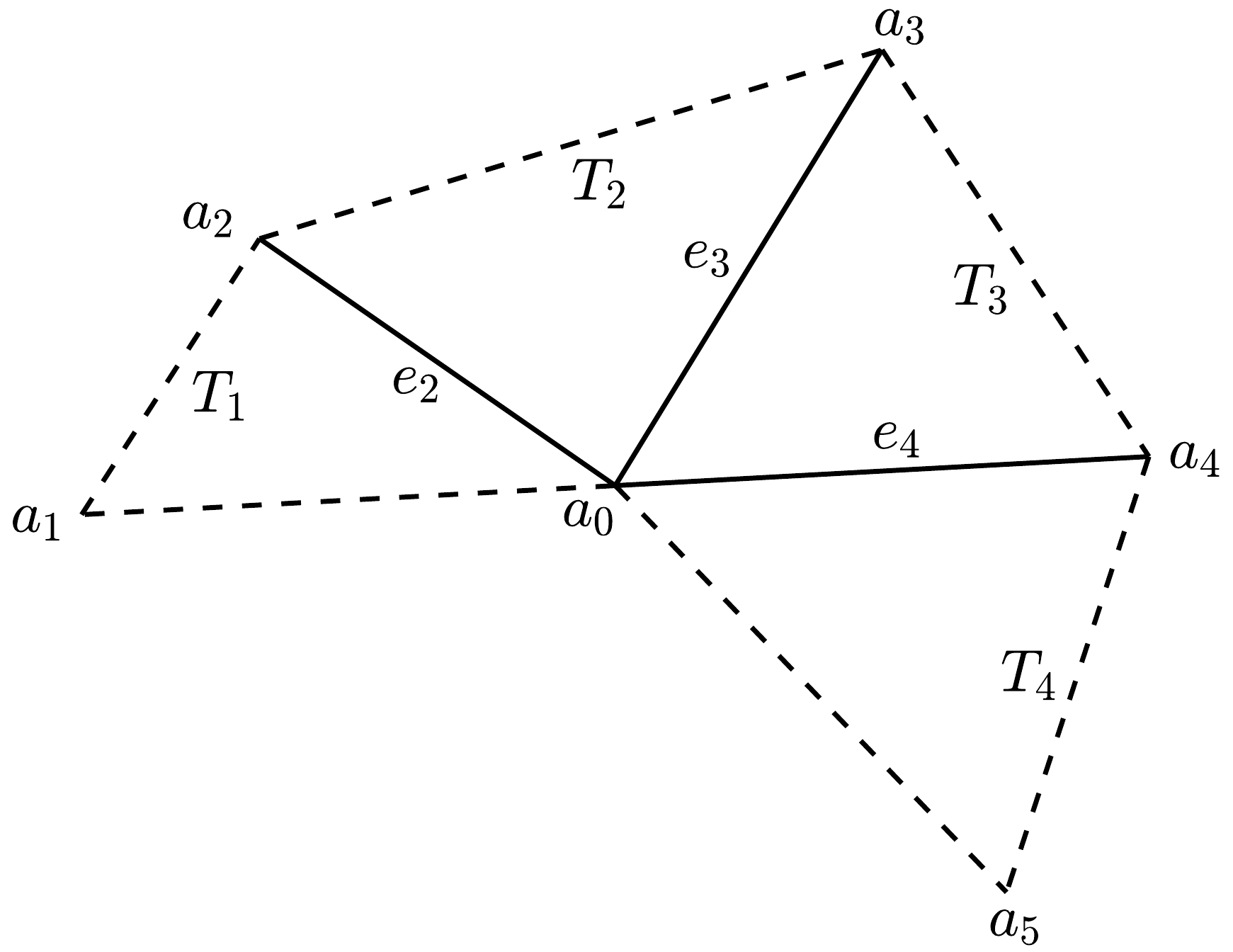}}
\caption{Patch $\omega$ composed of four continuous cells sharing an only vertex $a_{1}$.}\label{fig:four cell in M}
\end{figure}
\end{lemma}
\begin{proof}
We will complete the proof in four steps.

\noindent{\bf Step 1.} Consider $T_{1}$ and $T_{4}$. We have
\begin{align}\label{eq:rep on T1T4}
\uv{}_{h}|_{T_{1}} = r\, \uw{}_{T_{1},e_{2}}, \ \mbox{ and } \uv{}_{h}|_{T_{4}} = s\, \uw{}_{T_{4},e_{4}}.
\end{align}
with constants $r$ and $s$.

\noindent{\bf Step 2.} Consider two adjacent cells $T_{1}\cup T_{2}$. From Lemma~\ref{lem:two cell 2 dim}, it holds that
\begin{align}\label{eq:two cell 2 dim}
 \uv{}_{h}|_{T_{2}} = r\,\uw{}_{T_{2},e_{2}} + \Big(\frac{S_{1}+S_{2}}{d_{2}}r\Big)\,\uw{}_{T_{2},e_{3},e_{2}} + b_{4}\,\uw{}_{T_{2},e_{3}}.
\end{align}
with $b_{4}$ to be determined.

\noindent{\bf Step 3.} Consider two adjacent cells $T_{3}\cup T_{4}$.  Similar to Step 2, we have 
\begin{align}\label{eq:rep on T3}
\uv{}_{h}|_{T_{3}} = c_{1}\uw{}_{T_{3},e_{3}}   -\Big(\frac{S_{4}+S_{3}}{d_{4}}s\Big)\, \uw{}_{T_{3},e_{3},e_{4}} + s\,\uw{}_{T_{3},e_{4}},
\end{align}
with $c_{1}$ to be determined.

\noindent{\bf Step 4.} Consider two adjacent cells $T_{2}\cup T_{2}$. Utilizing the continuity of $\uv{}_{h}\cdot \mathbf{n}$ and $\fint_{e_{3}}\uv{}_{h}\cdot \mathbf{t}$ on $e_{3}$, and the representation of $\uv{}_{h}$ in~\eqref{eq:rep on T1T4} -- \eqref{eq:rep on T3}, we derive 
\begin{align}\label{eq:T3 and T2}
b_{4} = c_{1} = 0,\ s = -\frac{(S_{1}+S_{2})d_{4}}{(S_{4}+S_{3})d_{2}}\,r.
\end{align}
Therefore, $\uv{}_{h}$ is determined once the constant $r$ is given and ${\rm dim}({\rm ker}(\dv, \uV{}_{h0,\omega})) = 1$. 
\end{proof}

 If we take $r=\frac{-d_{2}}{S_{1}+S_{2}}$, then
\begin{equation}
\label{eq:Four-Cell ker represent}
\begin{split}
&\uv{}_{h}|_{T_{1}} =\frac{- d_{2}}{S_{1}+S_{2}}\uw{}_{T_{1},e_{2}},
\qquad\qquad \qquad \ \uv{}_{h}|_{T_{2}}  = \frac{-d_{2}}{S_{1}+S_{2}}\uw{}_{T_{2},e_{2}} - \uw{}_{T_{2},e_{3},e_{2}}, \\
& \uv{}_{h}|_{T_{3}}  = -\uw{}_{T_{3},e_{3},e_{4}} + \frac{d_{4}}{S_{3}+S_{4}}\uw{}_{T_{3},e_{4}},
\qquad \uv{}_{h}|_{T_{4}} = \frac{d_{4}}{S_{3}+S_{4}}\uw{}_{T_{4},e_{4}},
\end{split}
\end{equation}
\begin{remark}\label{rem:three cells no dimension}
{\rm
The pattern in Figure~\ref{fig:three_triangles} can degenerate to a patch which admits only trivial function, i.e., zero vector function; see Figure~\ref{fig:variant Three-Cell}. 
}
\begin{figure}[h!!]
\begin{minipage}[t]{0.5\linewidth}
\centerline{\includegraphics[width=1.8in]{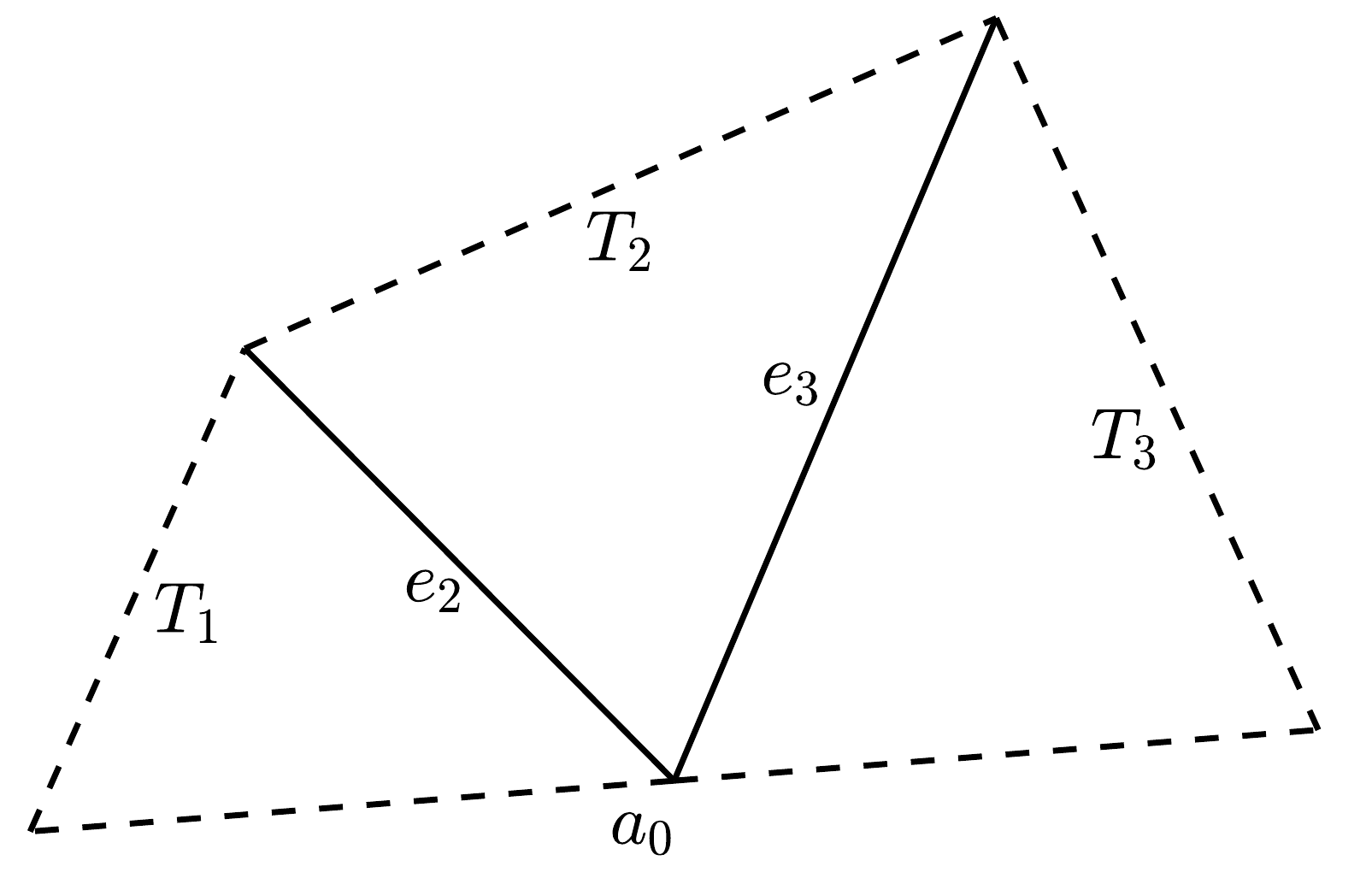}}
\end{minipage}%
\begin{minipage}[t]{0.5\linewidth}
\centerline{\includegraphics[width=1.8in]{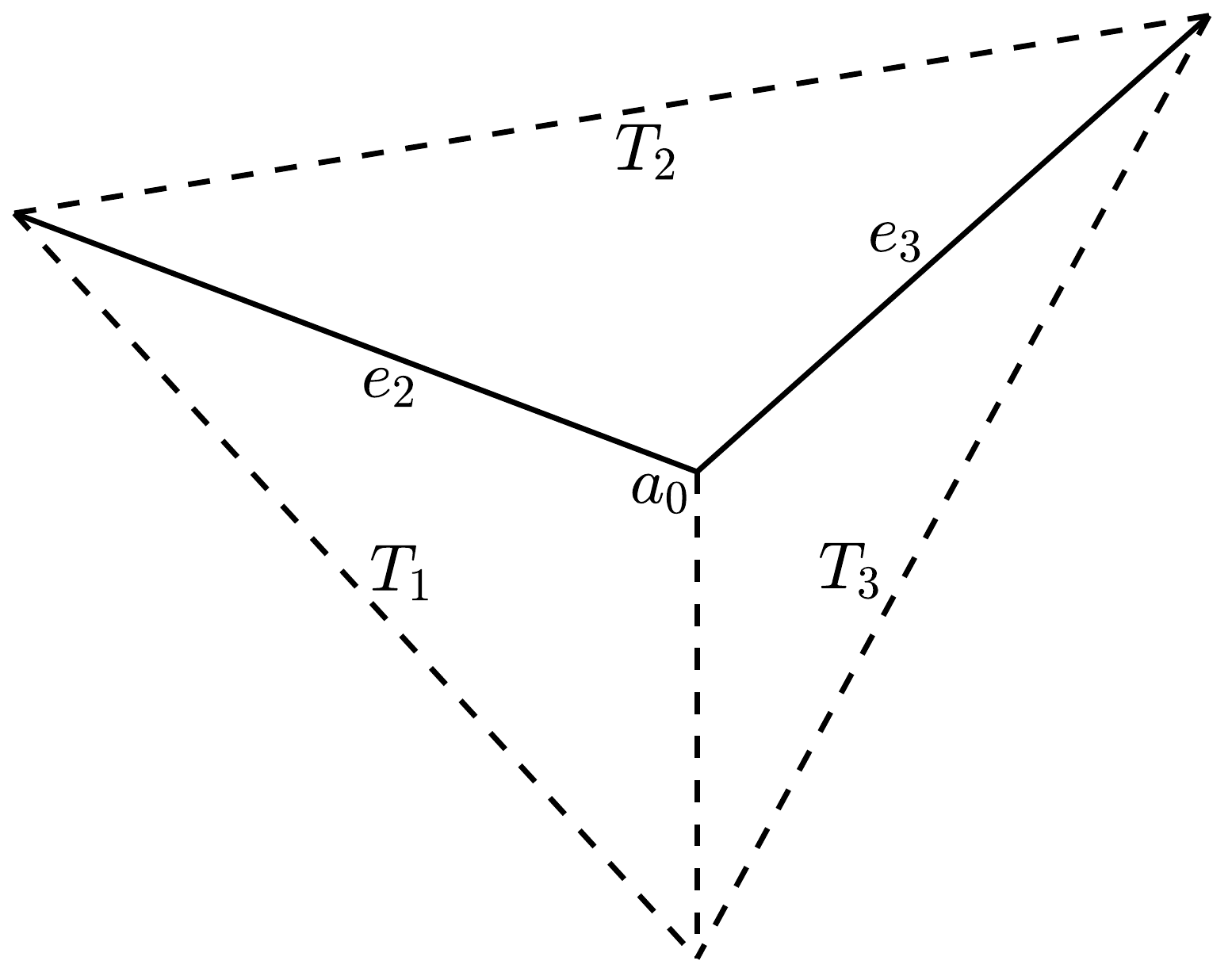}}
\end{minipage}%
\caption{Degenerate case: three cells form a patch.}\label{fig:variant Three-Cell}
\end{figure}
\end{remark}

\begin{remark}\label{rem:var atoms func}
{\rm
Let $T_1$, $T_2$ and $T_3$ be three adjacent cells, such that $\overline{T}_1\cap \overline{T}_2=e_{2}$, $\overline{T}_2\cap \overline{T}_3=e_{3}$, and $\overline{T}_3\cap \overline{T}_1=e_{1}$. We may treat this triple of cells as a degenerate case of a four cell sequence $T_1-T_2-T_3-T_1$, and the degenerate patch may inherit the divergence-free basis function on the patch  $T_1-T_2-T_3-T_1$.

\begin{figure}[htbp]
\begin{minipage}[t]{0.5\linewidth}
\centerline{\includegraphics[width=1.8in]{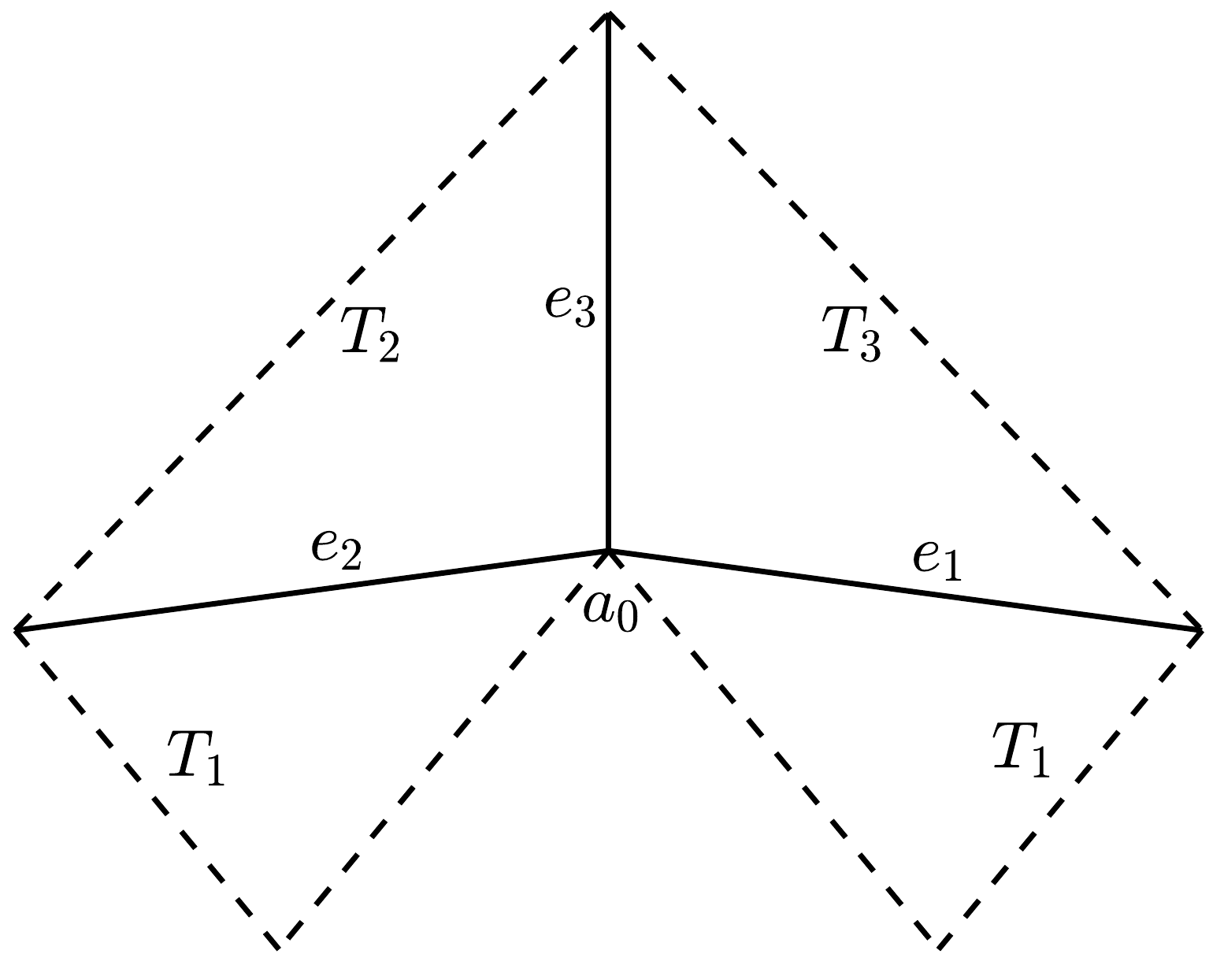}}
\end{minipage}%
\begin{minipage}[t]{0.5\linewidth}
\centerline{\includegraphics[width=1.8in]{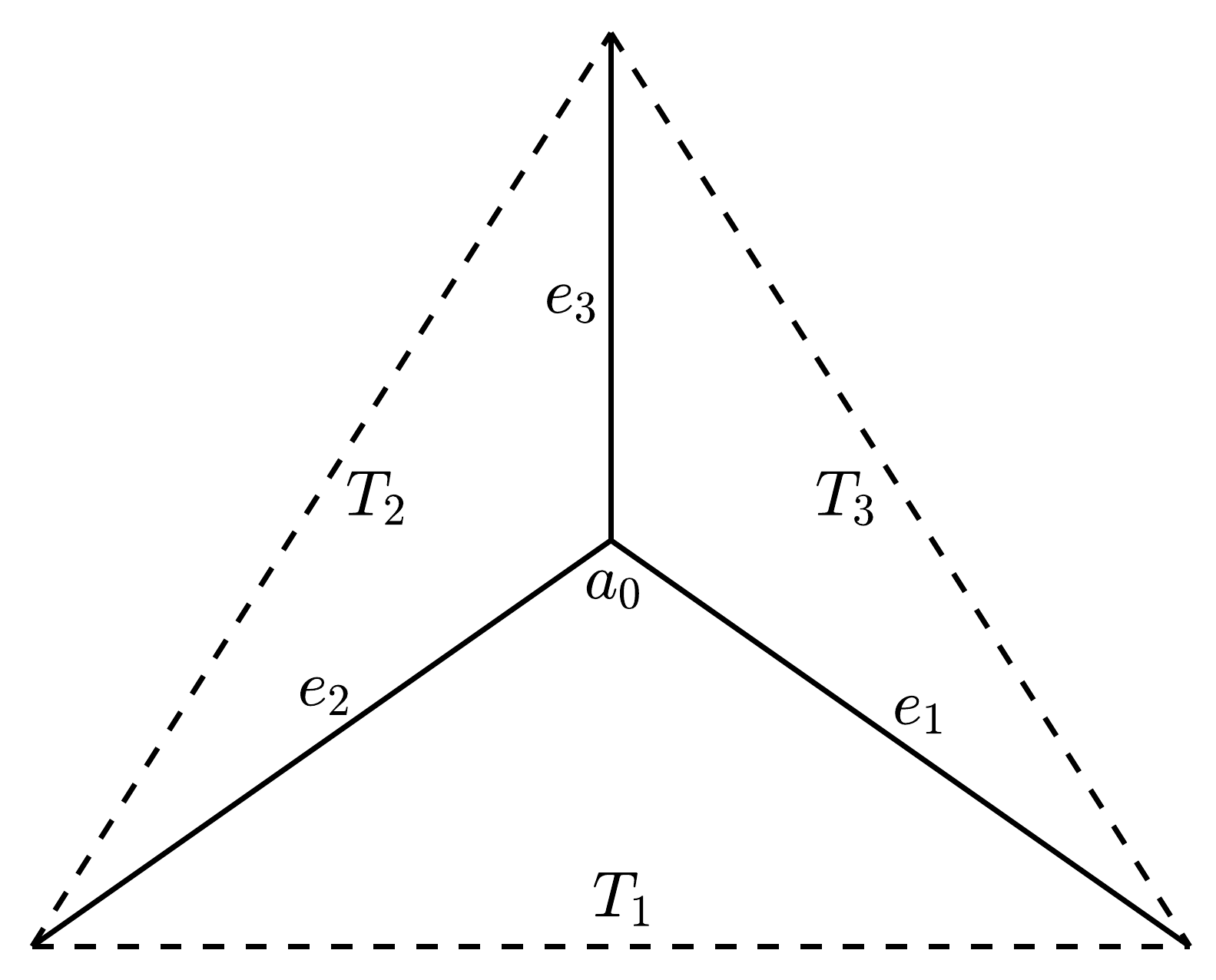}}
\end{minipage}%
\caption{The first and the last cells of the left pattern overlap to form the right pattern.}\label{fig:variant Four-Cell}
\end{figure}

Take $T_{1}$ to be the overlapping cell. The function (corresponding to~\eqref{eq:Four-Cell ker represent}), denoted by $\uv{}_{1}$, satisfies
\begin{equation}\label{eq:var Four-Cell ker represent}
\begin{split}
& \uv{}_{1}|_{T_{1}}  = \frac{-d_{2}}{S_{1}+S_{2}}\uw{}_{T_{1},e_{2}}  + \frac{d_{1}}{S_{3}+S_{1}}\uw{}_{T_{1},e_{1}}
 \qquad \qquad 
\uv{}_{1}|_{T_{2}} =  \frac{-d_{2}}{S_{1}+S_{2}}\uw{}_{T_{2},e_{2}}  - \uw{}_{T_{2},e_{3},e_{2}},\\ 
&\uv{}_{1}|_{T_{3}} = -\uw{}_{T_{3},e_{3},e_{1}} + \frac{d_{1}}{S_{3}+S_{1}}\uw{}_{T_{3},e_{1}}.
\end{split}
\end{equation}
The key ingredient is that $\uv{}_1|_{T_1}$ consists of $\uw{}_{T_1,e_2}$ and $\uw{}_{T_{1},e_1}$; note that this is just the dominant ingredients of the function (corresponding to $T_1-T_2-T_3-T_1$) on the two end cells. 
}
\end{remark}

\subsection{Structure of the kernel space on an {\bf${\mathbf m}$-macroelement}}
 Let $M =\cup_{s = 1:m}\ T_{s}$ be an {\bf${\mathbf m}$-macroelement}, and $\overline{T}_{s}\cap \overline{T}_{s+1} = e_{s+1}$. Particularly, $T_{m+1}=T_1$ and $e_{m+1}=e_1$. In the following context, the subscript~$s$ actually refers to $s-m$ if it is calculated to be great than $m$. 
 
\begin{figure}[htbp]
\centerline{\includegraphics[width=2.35in]{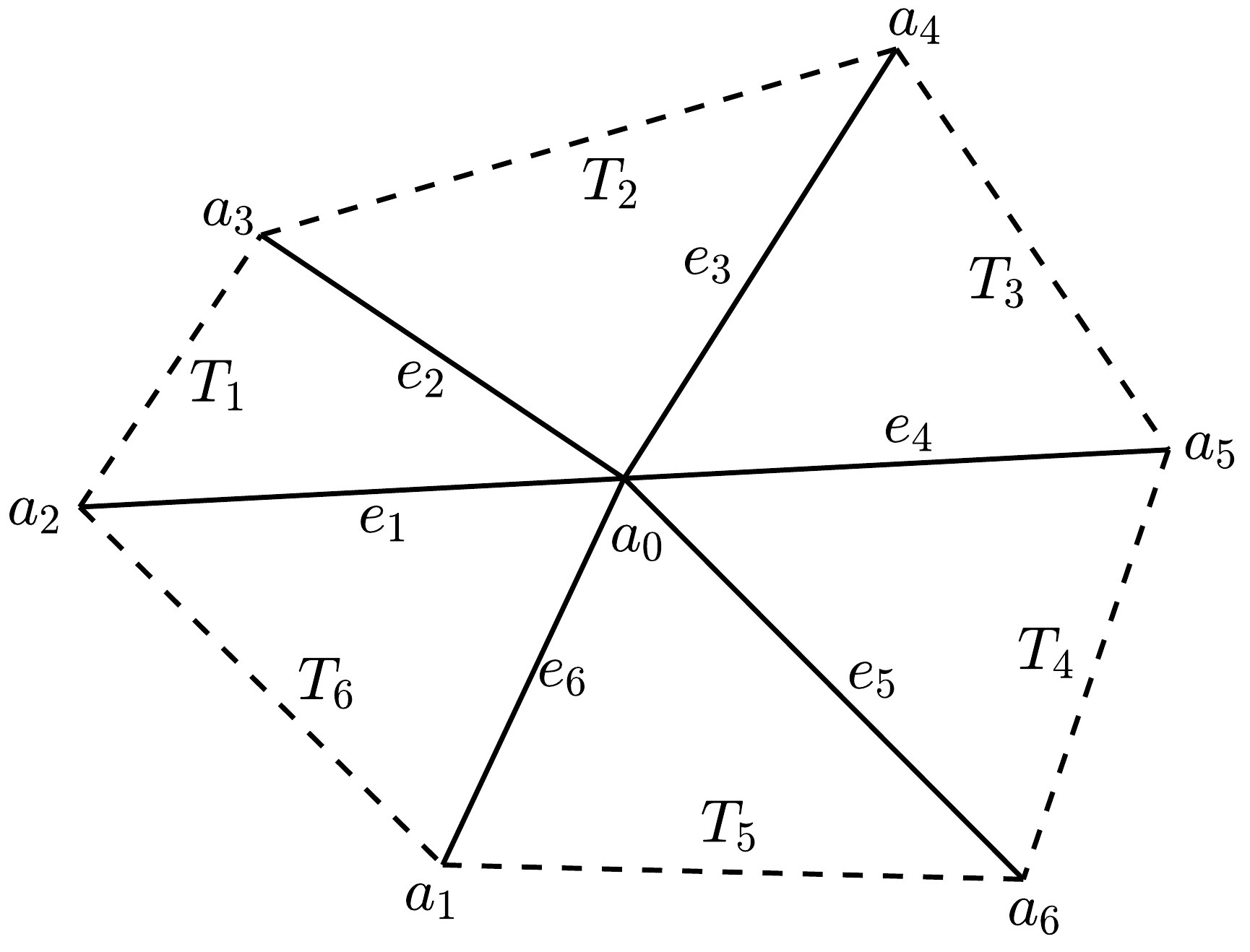}}
\caption{A macroelement composed of six cells with one interior vertex.}\label{fig:macroelement}
\end{figure}

\begin{definition}\label{def:atom function}
{\rm 
Let a subdomain $\omega_{l}$ be composed of continuous cells $T_{l}\cup T_{l+1} \cup T_{l+2}\cup T_{l+3}$, and $e_{l+1},\ e_{l+2}$, and $e_{l+3}$ be its three interior edges.  If 
\begin{itemize}
\item[(i)] for $m \geqslant 4$, $\uv{}_{l}|_{\omega_{l}} \in {\rm ker}(\dv, \uV{}_{h0,\omega_{l}})$ satisfies condition~\eqref{eq:Four-Cell ker represent}, and $\uv{}_{l}$ vanishes on ${M \backslash \omega_{l}}$,
 \item[(ii)] for $m = 3$, $\uv{}_{l}|_{\omega_{l}} \in {\rm ker}(\dv, \uV{}_{h0,\omega_{l}})$ satisfies condition~\eqref{eq:var Four-Cell ker represent},\end{itemize}
then $\uv{}_{l}$ is called an atom function on $M$.
}
\end{definition}
Therefore, there exist $m$ atom functions on an {\bf${\mathbf m}$-macroelement}.
Recall  
\begin{equation}
{\rm ker}(\dv, \uV{}_{h0,M}) =\Big\{\, \uv{}_h\in \uV{}_{h0}(M):\dv\,\uv{}_h=0\Big\}
\end{equation}
and denote 
\begin{equation}
\uZ{}_{h0}^\mathbf{n}(M):=\Big\{\, \uv{}_h\in{\rm ker}(\dv, \uV{}_{h0,M}):\int_{e_i}\uv{}_h\cdot\mathbf{n}_{e_i}=0,\ 1\leqslant i\leqslant m\Big\}.
\end{equation}

For an atom function $\uv{}_{i}$ on the {\bf${\mathbf m}$-macroelement} $M$, ${\rm supp}(\uv{}_{i}):= T_{i}\cup T_{i+1}\cup T_{i+2}\cup T_{i+3}$ ($m\geqslant 4$) or ${\rm supp}(\uv{}_{i} ):= T_{i}\cup T_{i+1}\cup T_{i+2}$ ($m = 3$). $T_{i+s}$ is called the $(s+1)$-th cell of ${\rm supp}(\uv{}_{i})$.
\begin{lemma}\label{lem:dim m}
It holds that
$\uZ{}_{h0}^\mathbf{n}(M)={\rm span}\{\uv{}_l\}_{1\leqslant l\leqslant m}$.
\end{lemma}
\begin{proof}
Here we adopt a sweeping procedure (c.f., Ref.~\cite{Shuo.Zhang2020}) to conduct the proof. Let $T$ be an arbitrary cell in $M$, and $T_{R}\cup T \cup T_{L}$ be three cells in $M$ arranged in a clockwise direction. Let $e_{R}: = \overline{T}_{R}\cap \overline{T} $ and $e_{L}: = \overline{T}\cap \overline{T}_{L}$. Use $S_{L}$ and $S_{R}$ to represent the areas of $T_{L}$ and $T_{R}$, respectively. Let $d_{L}$ and $d_{R}$ represent the lengths of $e_{L}$ and $e_{R}$.

Given $\uv{}_h\in \uZ{}_{h0}^\mathbf{n}$, there exists $\alpha_i\in\mathbb{R}$, $1\leqslant i\leqslant 4$, such that 
$$
\uv{}_h|_T=\alpha_1\uw{}_{T,e_L}+\alpha_2\uw{}_{T,e_L,e_R}+\alpha_3\uw{}_{T,e_R,e_L}+\alpha_4\uw{}_{T,e_R}.
$$
Let $\uv{}_{k_{1}}$, $\uv{}_{k_{2}}$, and $\uv{}_{k_{3}}$ be three atom functions satisfy that $T$ is the first, second, and third cell of ${\rm supp}(\uv{}_{k_{1}})$, ${\rm supp}(\uv{}_{k_{2}})$, and ${\rm supp}(\uv{}_{k_{3}})$, respectively. Then
\begin{align*}
& \uv{}_{k_{1}}|_{T} = \frac{-d_{L}}{S+S_{L}} \uw{}_{T,e_{L}} \quad \mbox{or } \ \uv{}_{k_{1}}|_{T} = \frac{-d_{L}}{S+S_{L}} \uw{}_{T,e_{L}} + \frac{d_{R}}{S+S_{R}} \uw{}_{T,e_{R}} \mbox{ when } m = 3, \\
& \uv{}_{k_{2}}|_{T} =  \frac{-d_{R}}{S+S_{R}} \uw{}_{T,e_{R}}  - \uw{}_{T,e_{L},e_{R}}, \quad \uv{}_{k_{3}}|_{T} =  \frac{d_{L}}{S+S_{L}} \uw{}_{T,e_{L}}  - \uw{}_{T,e_{R},e_{L}}.
\end{align*}
Notice that $\uw{}_{T,e_{R}}|_{e_{L}}$ = 0 in the sense of DOF, i.e., the degrees of freedom associated with $e_L$ of  $\uw{}_{T,e_{R}}$ all vanish.
Set 
$$
 \uz{}_{0} = \Big( -\alpha_{1}\frac{S+S_{L}}{d_{L}}\uv{}_{k_{1}} - \alpha_{2}\uv{}_{k_{2}} - \alpha_{3}(\uv{}_{k_{1}}+\uv{}_{k_{3}})\Big): = r_{1}\uv{}_{k_{1}} + r_{2}\uv{}_{k_{2}} + r_{3}\uv{}_{k_{3}}.
$$
then $(\uv{}_{h}-\uz{}_{0})|_{e_{L}} = 0$ in the sense of DOF. 

\noindent{\bf (i)} If $m = 3$, it holds $\uv{}_{h}-\uz{}_{0}$ vanish on $M$ by Remark~\ref{rem:three cells no dimension}. Hence $\uv{}_{h} = \sum_{l= 1}^{m}r_{i}\uv{}_{k_{i}}$.

\noindent{\bf (ii)} If $m>3$, consider the left cell $T_{L}$ adjacent to $T$. There exists $\uv{}_{k_{4}}$ such that $T_{L}$ is the first cell of ${\rm supp}(\uv{}_{k_{4}})$. Therefore, $(\uv{}_{h}-\uz{}_{0})|_{T_{L}} \in {\rm span}\{\uv{}_{k_{4}}|_{T_{L}}\}$. Hence there exists a constant $r_{4}$, such that $(\uv{}_{h}-\uz{}_{0})|_{T_{L}} = r_{4}\uv{}_{k_{4}}$, and then $\uv{}_{h}-\sum_{l = 1}^{4}r_{l}\uv{}_{k_{l}}$ vanishes on $T_{L}$. Therefore, the number of supporting cells is reduced to $m-1$. Conduct a similar analysis to the next left cell, and the number of supporting cells can also be reduced by one. Repeat this process until the number of supporting cells is smaller than three, and they form a pattern as shown in Figure~\ref{fig:variant Three-Cell} (left). Finally it can be derived that $\uv{}_{h} = \sum_{l = 1}^{m}r_{l}\uv{}_{k_{l}}$.
\end{proof}

\paragraph{\bf Proof of Lemma \ref{lem:localkernel}} It suffices for us to show $\dim({\rm ker}(\dv, \uV{}_{h0,M}))\leqslant \dim(\uZ{}_{h0}^\mathbf{n}(M))+1$, and Lemma \ref{lem:localkernel} follows by Lemma \ref{lem:dim m}.

Let $\mathbf{n}_{e_{l}}$ be the unit normal vector on an interior edge $e_{l}$ with $l = 1:m$, whose direction is from $a_{0}$ to $a_{m}$. 
Given $\upsi{}_h\in {\rm ker}(\dv, \uV{}_{h0,M})$, then the divergence theorem leads to $\int_{e_i}\upsi{}_h\cdot\mathbf{n}_{e_i}=\int_{e_j}\upsi{}_h\cdot\mathbf{n}_{e_j}$ with $1\leqslant i,j\leqslant m$. Assume there exists a function $\upsi{}_h \in {\rm ker}(\dv, \uV{}_{h0,M})$, such that $\int_{e_i}\upsi{}_h\cdot\mathbf{n}_{e_i}=1$ with $i = 1:m$. Then, $\uv{}_h\in {\rm ker}(\dv, \uV{}_{h0,M})$ can be uniquely decomposed into $\uv{}_h=\alpha\upsi{}_h+\uv{}_h^1\ \ \mbox{with}\ \uv{}_h^1\in \uZ{}_{h0}^\mathbf{n},$ where $\alpha$ represents a constant. Namely, ${\rm ker}(\dv, \uV{}_{h0,M})=\uZ{}_{h0}^\mathbf{n}(M)+{\rm span}\{\upsi{}_h\}$. If such a function $\upsi{}_h$ does not exist, then ${\rm ker}(\dv, \uV{}_{h0,M})=\uZ{}_{h0}^\mathbf{n}(M)$. In any event, $\dim({\rm ker}(\dv, \uV{}_{h0,M}))\leqslant \dim(\uZ{}_{h0}^\mathbf{n}(M))+~1$. The proof is completed. 

\bibliography{ZZZ_bib2021}

\begin{thebibliography}{10}

\bibitem{Arnold;Qin1992}
D.~N. Arnold and J.~Qin.
\newblock Quadratic velocity/linear pressure {Stokes} elements.
\newblock In {\em Advances in Computer Methods for Partial Differential
  Equations VII}, pages 28--34. IMACS, 1992.

\bibitem{Auricchio;Veiga;Lovadina;Reali2010}
F.~Auricchio, L.~{Beir\~{a}o da Veiga}, C.~Lovadina, and A.~Reali.
\newblock The importance of the exact satisfaction of the incompressibility
  constraint in nonlinear elasticity: mixed {FEMs} versus {NURBS-based}
  approximations.
\newblock {\em Computer Methods in Applied Mechanics and Engineering},
  199(5):314--323, 2010.

\bibitem{Auricchio2013}
F.~Auricchio, L.~Beir\~{a}o~da Veiga, C.~Lovadina, A.~Reali, R.~L. Taylor, and
  P.~Wriggers.
\newblock Approximation of incompressible large deformation elastic problems:
  some unresolved issues.
\newblock {\em Computational Mechanics}, 52(5):1153--1167, 2013.

\bibitem{Bernardi;Raugel1985}
C.~Bernardi and G.~Raugel.
\newblock Analysis of some finite elements for the {Stokes} problem.
\newblock {\em Mathematics of Computation}, 44(169):71--79, 1985.

\bibitem{Brenner;Scott2002}
S.~C. Brenner and L.~R. Scott.
\newblock {\em The mathematical theory of finite element methods}, volume~15 of
  {\em Texts in Applied Mathematics}.
\newblock Springer-Verlag, New York, second edition, 2002.

\bibitem{Brezzi1974}
F.~Brezzi.
\newblock On the existence, uniqueness and approximation of saddle-point
  problems arising from {L}agrangian multipliers.
\newblock {\em R.A.I.R.O. Analyse Num\'{e}rique}, 8({\rm R}-2):129--151, 1974.

\bibitem{BrezziDouglasMarini1985}
F.~Brezzi, J.~Douglas, Jr., and L.~D. Marini.
\newblock Recent results on mixed finite element methods for second order
  elliptic problems.
\newblock In {\em Vistas in applied mathematics}, Transl. Ser. Math. Engrg.,
  pages 25--43. Optimization Software, New York, 1986.

\bibitem{Brezzi;Fortin1991}
F.~Brezzi and M.~Fortin.
\newblock {\em Mixed and Hybrid Finite Element Methods}, volume~15.
\newblock Springer-Verlag, New York, 1991.

\bibitem{Bruneau;Saad2006}
C.-H. Bruneau and M.~Saad.
\newblock The {2D} lid-driven cavity problem revisited.
\newblock {\em Computers {\&} Fluids}, 35(3):326--348, 2006.

\bibitem{Chen;Dong;Qiao2013}
S.~Chen, L.~Dong, and Z.~Qiao.
\newblock Uniformly convergent {$H(\dv)$}-conforming rectangular elements for
  {Darcy--Stokes} problem.
\newblock {\em Science China Mathematics}, 56(12):2723--2736, 2013.

\bibitem{Ciarlet1978}
P.~G. Ciarlet.
\newblock {\em The finite element method for elliptic problems}, volume~4.
\newblock North-Holland Pub. Co, New York, Amsterdam, 1978.

\bibitem{Crouzeix1973}
M.~Crouzeix and P.-A. Raviart.
\newblock Conforming and nonconforming finite element methods for solving the
  stationary {Stokes} equations \uppercase\expandafter{\romannumeral1}.
\newblock {\em R.A.I.R.O.}, 7(R3):33--75, 1973.

\bibitem{Dahlen1974}
F.~A. Dahlen.
\newblock On the static deformation of an earth model with a fluid core.
\newblock {\em Geophysical Journal of the Royal Astronomical Society},
  36(2):461--485, 1974.

\bibitem{Dios;Brezzi;Marini;Xu;Zikatanov2014}
B.~A.~D. Dios, F.~Brezzi, L.~D. Marini, J.~Xu, and L.~Zikatanov.
\newblock A simple preconditioner for a discontinuous {Galerkin} method for the
  {Stokes} problem.
\newblock {\em Journal of Scientific Computing}, 58(3):517--547, 2014.

\bibitem{Falk;Morley;1990}
R.~S. Falk and E.~Morley.
\newblock Equivalence of finite element methods for problems in elasticity.
\newblock {\em {SIAM} Journal on Numerical Analysis}, 27:1486--1505, 1990.

\bibitem{FALK;NEILAN;2013}
R.~S. Falk and M.~Neilan.
\newblock Stokes complexes and the construction of stable finite elements with
  pointwise mass conservation.
\newblock {\em {SIAM} Journal on Numerical Analysis}, 51(2):1308--1326, 2013.

\bibitem{Fortin1977}
M.~Fortin.
\newblock An analysis of the convergence of mixed finite element methods.
\newblock {\em RAIRO Analyse Num\'{e}rique}, 11(4):341--354, 1977.

\bibitem{Nicolas;Alexander;Philipp2019}
N.~R. Gauger, A.~Linke, and P.~W. Schroeder.
\newblock On high-order pressure-robust space discretisations, their advantages
  for incompressible high {Reynolds} number generalised {Beltrami} flows and
  beyond.
\newblock {\em The SMAI Journal of Computational Mathematics}, 5:89--129, 2019.

\bibitem{Girault;Raviart1986}
V.~Girault and P.-A. Raviart.
\newblock {\em Finite element methods for {Navier--Stokes} equations}, volume~5
  of {\em Springer Series in Computational Mathematics}.
\newblock Springer-Verlag, Berlin, 1986.

\bibitem{Guz;Neilan2012}
J.~Guzm{\'a}n and M.~Neilan.
\newblock A family of nonconforming elements for the {Brinkman} problem.
\newblock {\em IMA Journal of Numerical Analysis}, 32(4):1484--1508, 2012.

\bibitem{Guzman;Neilan2013}
J.~Guzm{\'a}n and M.~Neilan.
\newblock Conforming and divergence-free {Stokes} elements in three dimensions.
\newblock {\em IMA Journal of Numerical Analysis}, 34:1489--1508, 10 2013.

\bibitem{Guzman;NeilanGeneral2014}
J.~Guzm{\'a}n and M.~Neilan.
\newblock Conforming and divergence-free {Stokes} elements on general
  triangular meshes.
\newblock {\em Mathematics of Computation}, 83(285):15--36, 2014.

\bibitem{Guzman;Neilan2018}
J.~Guzm{\'a}n and M.~Neilan.
\newblock Inf-sup stable finite elements on barycentric refinements producing
  divergence-free approximations in arbitrary dimensions.
\newblock {\em SIAM Journal on Numerical Analysis}, 56(5):2826--2844, 2018.

\bibitem{Hiptmair;Li;Mao;Zheng2018}
R.~Hiptmair, L.~Li, S.~Mao, and W.~Zheng.
\newblock A fully divergence-free finite element method for magnetohydrodynamic
  equations.
\newblock {\em Mathematical Models and Methods in Applied Sciences}, 28:1--37,
  2018.

\bibitem{Hong;Wang;Wu;Xu}
Q.~Hong, F.~Wang, S.~Wu, and J.~Xu.
\newblock A unified study of continuous and discontinuous {Galerkin} methods.
\newblock {\em Science China. Mathematics}, 62(1):1--32, 2019.

\bibitem{Hu;Ma;Xu2017}
K.~Hu, Y.~Ma, and J.~Xu.
\newblock Stable finite element methods preserving $\nabla \cdot \mathbf{B}=0$
  exactly for {MHD} models.
\newblock {\em Numerische Mathematik}, 135(2):371--396, 2017.

\bibitem{Hu;Xu2019}
K.~Hu and J.~Xu.
\newblock Structure-preserving finite element methods for stationary {MHD}
  models.
\newblock {\em Mathematics of Computation}, 88(316):553--581, 03 2019.

\bibitem{Huang;Zhang2011}
Y.~Huang and S.~Zhang.
\newblock A lowest order divergence-free finite element on rectangular grids.
\newblock {\em Frontiers of Mathematics in China}, 6(002):253--270, 2011.

\bibitem{Neilan2017}
V.~John, A.~Linke, C.~Merdon, M.~Neilan, and L.~G. Rebholz.
\newblock On the divergence constraint in mixed finite element methods for
  incompressible flows.
\newblock {\em SIAM Review}, 59(3):492--544, 2017.

\bibitem{Kouhia;Stenberg1995}
R.~Kouhia and R.~Stenberg.
\newblock A linear nonconforming finite element method for nearly
  incompressible elasticity and {Stokes} flow.
\newblock {\em Computer Methods in Applied Mechanics \& Engineering},
  124(3):195--212, 1995.

\bibitem{Linke;Merdon2020}
A.~Linke and C.~Merdon.
\newblock Well-balanced discretisation for the compressible {Stokes} problem by
  gradient-robustness.
\newblock In R.~Kl{\"o}fkorn, E.~Keilegavlen, F.~A. Radu, and J.~Fuhrmann,
  editors, {\em Finite Volumes for Complex Applications IX - Methods,
  Theoretical Aspects, Examples}, pages 113--121, Cham, 2020. Springer
  International Publishing.

\bibitem{Linke;Rebholz2019}
A.~Linke and L.~G. Rebholz.
\newblock Pressure-induced locking in mixed methods for time-dependent
  {(Navier--)Stokes} equations.
\newblock {\em Journal of Computational Physics}, 388:350 -- 356, 2019.

\bibitem{Mardal;Tai;Winther2002}
K.~A. Mardal, X.-C. Tai, and R.~Winther.
\newblock A robust finite element method for {Darcy--Stokes} flow.
\newblock {\em SIAM Journal on Numerical Analysis}, 40(5):1605--1631, 2002.

\bibitem{Neilan;Sap2016}
M.~Neilan and D.~Sap.
\newblock Stokes elements on cubic meshes yielding divergence-free
  approximations.
\newblock {\em Calcolo}, 53(3):263--283, 2016.

\bibitem{Qin;Zhang2007}
J.~Qin and S.~Zhang.
\newblock Stability and approximability of the {$P_{1}-P_{0}$} element for
  {S}tokes equations.
\newblock {\em International Journal for Numerical Methods in Fluids},
  54(5):497--515, 2007.

\bibitem{Schroeder;Lube2018}
P.~W. Schroeder and G.~Lube.
\newblock Divergence-free {H(div)-FEM} for time-dependent incompressible flows
  with applications to high {Reynolds} number vortex dynamics.
\newblock {\em Journal of Scientific Computing}, 75:830--858, 05 2018.

\bibitem{Scott;Vogelius2009}
L.~R. Scott and M.~Vogelius.
\newblock Norm estimates for a maximal right inverse of the divergence operator
  in spaces of piecewise polynomials.
\newblock {\em RAIRO Mod\'{e}lisation Math\'{e}matique et Analyse
  Num\'{e}rique}, 19(1):111--143, 2009.

\bibitem{Stenberg1990technique}
R.~Stenberg.
\newblock A technique for analysing finite element methods for viscous
  incompressible flow.
\newblock {\em International Journal for Numerical Methods in Fluids},
  11(6):935--948, 1990.

\bibitem{Tai;Winther2006}
X.-C. Tai and R.~Winther.
\newblock A discrete {de Rham} complex with enhanced smoothness.
\newblock {\em Calcolo}, 43(4):287--306, 2006.

\bibitem{Uchiumi2019}
S.~Uchiumi.
\newblock A viscosity-independent error estimate of a pressure-stabilized
  {Lagrange-Galerkin} scheme for the {Oseen} problem.
\newblock {\em Journal of Scientific Computing}, 80(2):834--858, 2019.

\bibitem{Xie;Xu;Xue2008}
X.~Xie, J.~Xu, and G.~Xue.
\newblock Uniformly stable finite element methods for {Darcy--Stokes--Brinkman}
  models.
\newblock {\em Journal of Computational Mathematics}, 26(3):437--455, 05 2008.

\bibitem{J.Xu1992}
J.~Xu.
\newblock Iterative methods by space decomposition and subspace correction.
\newblock {\em SIAM Review}, 34(4):581--613, 1992.

\bibitem{XuZhang2010}
X.~Xu and S.~Zhang.
\newblock A new divergence-free interpolation operator with applications to the
  {Darcy--Stokes--Brinkman} equations.
\newblock {\em SIAM Journal on Scientific Computing}, 32(2):855--874, 2010.

\bibitem{ZZZ2021}
H.~Zeng, C.-S. Zhang, and S.~Zhang.
\newblock Optimal quadratic element on rectangular grids for ${H}^1$ problems.
\newblock {\em BIT Numerical Mathematics}, published online, 2020.

\bibitem{Zhang3D2005}
S.~Zhang.
\newblock A new family of stable mixed finite elements for the {3D} {Stokes}
  equations.
\newblock {\em Mathematics of Computation}, 74(250):543--554, 2005.

\bibitem{Zhang2008}
S.~Zhang.
\newblock On the ${P_{1}}$ {P}owell-{S}abin divergence-free finite element for
  the {S}tokes equations.
\newblock {\em Journal of Computational Mathematics}, 26(003):456--470, 2008.

\bibitem{ZhangSY2009}
S.~Zhang.
\newblock A family of {$Q_{k+1,k}\times Q_{k,k+1}$} divergence-free finite
  elements on rectangular grids.
\newblock {\em {SIAM} Journal on Numerical Analysis}, 47(3):2090--2107, 01
  2009.

\bibitem{Zhang3D2011}
S.~Zhang.
\newblock Divergence-free finite elements on tetrahedral grids for $k\geqslant
  6$.
\newblock {\em Mathematics of Computation}, 80(274):669--695, 2011.

\bibitem{ZhangPS2011}
S.~Zhang.
\newblock Quadratic divergence-free finite elements on {Powell--Sabin}
  tetrahedral grids.
\newblock {\em Calcolo}, 48(3):211--244, Sept. 2011.

\bibitem{ShuoZhang2016}
S.~Zhang.
\newblock Stable finite element pair for stokes problem and discrete stokes
  complex on quadrilateral grids.
\newblock {\em Numerische Mathematik}, 133:371--408, 2016.

\bibitem{Shuo.Zhang2020}
S.~Zhang.
\newblock Minimal consistent finite element space for the biharmonic equation
  on quadrilateral grids.
\newblock {\em IMA Journal of Numerical Analysis}, 40(2):1390--1406, 2020.

\end{thebibliography}
\bibliographystyle{abbrv}

\end{document}